\documentclass[10pt,leqno]{amsart}%article}
\usepackage[backref=page,colorlinks=true]{hyperref}
\usepackage[english,francais]{babel}
\usepackage{amsmath,amsthm,amsfonts,amssymb,amscd,url}
\usepackage{graphics}
\numberwithin{equation}{section}
\input{xy} \xyoption{all}
\usepackage[latin1]{inputenc}
\usepackage[capitalize]{cleveref}

\usepackage{enumerate,xcolor}

\usepackage{hyperref, mathrsfs}
\usepackage{epsfig}
\input{xy} \xyoption{all}

\usepackage[scr=boondoxo,scrscaled=1.05]{mathalfa}
\makeindex
\begin{document}
\def\blue{black}
\def\ov{\overline}
\def\u{\underline}
\def\Per{\mathrm{Per}}
\def\diam {\mathrm{diam\, }}
\def\sign{\mathrm{sign\, }}
\def\Card{\mathrm{Card\, }}
\def\exp{\mathrm{exp}}
\def\S{\mathbb S}
\def \mC{{i\rho}}
\def\inv{^{-1}}
\def\Max{\text{max}}
\def\cal{\mathcal}
\def\Graph{\mathrm{Graph\, }}
\def \sfw{{\mathsf w}}
\def \sfh{{\mathsf h}}
\def\R{\mathbb R}
\def\Q{\mathbb Q}
\def\N{\mathbb N}
\def\Z{\mathbb Z}
\def\I{\mathbb I}
\def\K{\mathbb K}
\def\P{\mathbb P}
\def\D{\mathbb D}
\def\A{{\mathbb A}}
\def\B{\mathbb B}
\def\C{\mathbb C}
\def\O{\mathbb O}
\def\cC{\mathcal C}
\def\T{\mathbb T}
\def\a{{\underline a}}
\def\b{{\underline b}}
\def\c{{\underline c}}
\def\Log{\mathrm{log}}
\def\loc{\mathrm{loc}}
\def\inta{\mathrm{int }}
\def é{\'e}
\def \Lip{\mathrm{Lip}}
\def \Diff{\mathrm{Diff}}
\def\det{\mathrm{det}}
\def\Re{\mathrm{Re}}
\def\lip{\mathrm{Lip}}
\def\leb{\mathrm{Leb}}
\def\dom{\mathrm{Dom}}
\def\diam{\mathrm{diam}\:}
\def\supp{\mathrm{supp}\:}
\newcommand{\ovfork}{{\overline{\pitchfork}}}
\newcommand{\ovforki}{{\overline{\pitchfork}_{I}}}
\newcommand{\Tfork}{{\cap\!\!\!\!^\mathrm{T}}}
\newcommand{\whforki}{{\widehat{\pitchfork}_{I}}}
\newcommand{\marginal}[1]{\marginpar{{\scriptsize {#1}}}}
\def \np{{\color{red} 4}}
\def \npm{{\color{red} 3}}
\def \npmm{{\color{red} 2}}
\def \det{\mathrm{det}\, }
\def\cU{{\mathscr U}}
\def\cP{{\mathcal P}}
\def\cF{{\mathscr F}}
\def\cH{{\mathcal H}}
\def\cV{{\mathcal V}}
\def\cW{{\mathcal W}}
\def\cX{{\mathcal X}}
\def\cY{{\mathcal Y}}
\def\cZ{{\mathcal Z}}
\def\qand{\quad \text{and} \quad}

\def\1{{1\!\! 1}}

\def\sA{{\mathscr A}}
\def\sB{{\mathscr B}}
\def\sC{{\mathscr C}}
\def\sD{{\mathscr D}}
\def\sE{{\mathscr E}}
\def\sG{{\mathscr G}}
\def\sJ{{I}}
\def\si{{\mathscr i}}
\def\sj{{\mathscr j}}
\def\sL{{\mathscr L}}
\def\sM{{\mathscr M}}
\def\sN{{\mathscr N}}
\def\sP{{\mathscr P}}
\def\sR{{\mathscr R}}
\def\sS{{\mathscr S}}
\def \sT{{\mathscr T}}
\def\sU{{\mathscr U}}
\def\sV{{\mathscr V}}
\def\sX{{\mathscr X}}
\def\sY{{\mathscr Y}}
\def\sZ{{\mathscr Z}}

\def\sa{{\mathscr a}}
\def\sb{{\mathscr b}}
\def\sc{{\mathscr c}}
\def\sd{{\mathscr d}}
\def\se{{\mathscr e}}
\def\sf{t}
\def\sg{{\mathscr g}}
\def\sh{{\mathscr h}}
\def\sm{{\mathscr m}}
\def\sn{{\mathscr n}}
\def\sq{{\mathscr q}}

\def\so{{\diamond}}
%{{\odot}}
%{{\diamond}}
%{{\newmoon}}
\def\sp{{\mathscr p}}
\def\sl{{\mathscr {l}}}
\def\sr{{\mathscr {r}}}
\def\ss{{\mathscr {s}}}
\def\st{{\mathscr {t}}}
\def\su{{\mathscr {u}}}
\def\sv{{\mathscr {v}}}
\def\spp{{\mathscr {p}}}
\def\sw{{\mathscr {w}}}
\def\arr{\overleftarrow}
\def\avv{\overrightarrow}
\theoremstyle{plain}

\def\graph {{\mathrm{Graph}\, }}
\def\Graph {{\mathrm{Graph}\, }}
\def\Emb {{\mathrm{Emb}}}

\def\Cr#1{\overset{\tilde Y}{#1}}

\def\End {{\mathrm{End} }}

\newtheorem{theorem}{\bf Theorem}[section]

\newtheorem*{conjecture*}{\bf Conjecture}

\newtheorem{conj}[theorem]{\bf Conjecture}
\newtheorem{claim}[theorem]{\bf Claim}
\newtheorem{assumption}[theorem]{\bf Assumption}

\newtheorem{problem}[theorem]{\bf Problem}
\newtheorem{question}[theorem]{\bf Question}
\newtheorem{proposition}[theorem]{\bf Proposition}
\newtheorem{corollary}[theorem]{\bf Corollary}%[section]
\newtheorem{lemma}[theorem]{\bf Lemma}

\newtheorem{sublemma}[theorem]{\bf Sublemma}
\newtheorem*{Takens prbm}{Takens' Last Problem} 
\newtheorem{remark}[theorem]{\bf Remark}
\newtheorem{fact}[theorem]{\bf Fact}
\newtheorem{exem}[theorem]{\bf Example}
\newtheorem{definition}[theorem]{\bf Definition}
\newtheorem*{definition*}{\bf Definition}
\newtheorem{defi}[theorem]{\bf Definition}
\newtheorem{theo}{\bf Theorem}[section]
\renewcommand{\thetheo}{\Alph{theo}}
\newtheorem{coro}{\bf Corollary}[section]
\renewcommand{\thecoro}{\Alph{coro}}
\newtheorem{example}[theorem]{\bf Example}

%\newcommand\relatif{{\rm \rlap Z\kern 3pt Z}}
%\makeatletter
%\renewcommand\theequation{\thesection.\arabic{equation}}
%\@addtoreset{equation}{section}
%\makeatother

\renewcommand*{\backref}[1]{}
\renewcommand*{\backrefalt}[4]{\quad \tiny 
  \ifcase #1 (\textbf{NOT CITED.})%
  \or    (Cited on page~#2.)%
  \else   (Cited on pages~#2.)%
  \fi}
  \begin{otherlanguage}{english}

\title{ Generic family displaying robustly a fast growth of the number of periodic points}

\author{Pierre Berger%\footnote{CNRS-LAGA, Universit\'e Paris 13, USPC}
%{\bf PRELIMINARY VERSION, NOT FOR CIRCULATION}}
}

\date{\today}

\abstract{
For any $ 2 \le    r \le \infty$, $n\ge2$, we prove the existence of an open set $U$ of $C^r$-self-mappings of any $n$-manifold so that a generic map $f$ in $U$ displays a fast growth of the number of periodic points: the number of its $n$-periodic points grows as fast as asked. 
This complements the works of Martens-de Melo-van Strien,    Kaloshin, Bonatti-D\' iaz-Fisher and Turaev, to give a full answer 
to questions  asked by Smale in 1967, Bowen in 1978 and Arnold in 1989, for any manifold of any dimension and for any smoothness.

Furthermore for any $1\le r<\infty$ and any $k\ge 0$, we prove the existence of an open set $\hat U$ of $k$-parameter families in $U$ so that for a generic $(f_p)_p\in \hat U$, for every $\|p\|\le 1$, the map $f_p$ displays a fast growth of the number of  periodic points. 
This gives a negative answer to a problem asked by Arnold in 1992 in the finitely smooth case.}
} 
\maketitle
\tableofcontents

\section*{Introduction and statement of the main results}
Given a differentiable self-map $f$  of a compact manifold $M$, we denote by  $\Per_n f:= \{x\in M: f^n(x)=x\}$ the set of its $n$-periodic points. To study its cardinality, we shall consider the subset $\Per_n^0 f\subset \Per_n f$ of isolated $n$-periodic points. The cardinality of  $\Per^0_n f$ is invariant under conjugacy. Hence it is natural to study the growth of this cardinality with $n$.

Clearly, if $f$ is a polynomial map, the cardinality of $\Per^0_n f$ is bounded by the degree of $f^n$, which grows at most exponentially fast (see the generalization  \cite{DNT16}). %This holds true also for meromorphic maps \cite{DNT16}.
%; this is even true for meromorphic maps \cite{DNT16}. 
The first study in the $C^\infty$-case goes back to Artin and Mazur \cite{AM65} who proved the existence of  a dense set $\mathcal D$ in $\Diff^\infty(M)$  formed by diffeomorphisms  $f$ such that the cardinality of $ \Per^0_n f$
% {\color{\blue} sequence of} numbers $(\Card  \Per^0_n f)_{n\ge 1}$
  grows  at most  exponentially fast:
%, {\it i.e.} , there exists  $K(f)>0$ so that:
\begin{equation}\tag{A-M} \limsup_{n\to \infty} \frac1n \log \Card \Per^0_n f<\infty%\le K(f)\quad {\color{\blue} \text{ for every }  {n\ge 1}} 
%\exp(K(f)\cdot n)
\; .\end{equation}
This leads Smale \cite{Sm} and Bowen \cite{Bo78} to wonder whether a  topologically  generic diffeomorphism satisfies property (A-M). Later Arnold asked the following problem:
\begin{problem}[Smale 1967, Bowen  1978, Arnold  Pb. 1989-2 \cite{Ar00}] \label{problem1} Can the number of fixed points of the $n^{th}$ iteration of a topologically generic infinitely smooth self-mapping of a compact manifold grow, as $n$ increases, faster than any prescribed sequence $(a_n)_n$ (for some subsequences of time values n)? 
\end{problem}
Given $\infty\ge r\ge 1$, we recall that a property is $C^r$-\emph{topologically generic} if it holds true in a countable intersection of open and dense sets of $C^r(M,M)$. The topology on the space of $C^\infty$-maps is the union of the ones induced  by the $C^r$-topologies  for finite $r\ge 0$.

 {\color{\blue} Behind this problem is the following basic one: Is there an interesting dynamical property which is valid for any polynomial dynamics but not for a typical differentiable dynamics? 
 
   We will see that Problem \ref{problem1} motivated a large number of important contributions and that our first main result complements these to give a complete positive answer when  $\dim M\ge 2$.% and $\color{red} \infty \ge r\ge 2$. 
   
   However, it is well known that  topological genericity is not a so good notion of typicality from the metric viewpoint. Indeed, there are topologically generic subsets of $\R$ which have Lebesgue measure 0. }   
Another notion of {\color{\blue} typicality} was sketched by Kolmogorov during his plenary talk at the ICM in 1954. His student Arnold then formalized it as follows \cite{IL99, KH07}:
\begin{definition}[Arnold's Typicality]\label{Deftyp}
A property $(\mathcal P)$ on the set of $C^r$-self-mappings  of $M$ is \emph{typical} if for $k\ge 1$, with $B_k$ the unit closed ball of $\R^k$,  for a topologically generic $C^r$-family $(f_p)_{p\in B_k}$ of  maps $f_p\in C^r(M,M)$,   the map $f_p$ satisfies property~$(\mathcal P)$ for  Lebesgue almost every~$p\in B_k$.
\end{definition}
We recall that $(f_p)_{p\in B_k}$ is of class $C^r$ if the map $(p,z)\mapsto f_p(z)$ is of class $C^r$. When $r<\infty$, the topology on this space is equal to the uniform $C^{r}$-topology of $C^{r}(B_k\times M,M)$.  The topology on the space of $C^\infty$-families is the one given by the union of those induced  by the $C^{r'}$-topologies for finite  $r'\ge 0$. In this work we address the following:
\begin{problem}[Arnold 1992-13 \cite{Ar00}]\label{problem2} Prove that for a typical smooth  self-map $f$ of a compact manifold,    the cardinality of $\Per_n f$ 
grows  at most exponentially fast.  
\end{problem}
This problem inspired Arnold to formulate many  related  ones \cite[1994-47, 1994-48, 1992-14]{Ar00}.
The second and main result of this article is a negative answer to Problem \ref{problem2}  in the finite differentiable case and dimension $\ge 2$. {\color{\blue} Before stating this, let us relate (some of) the long tradition of works on  these problems.} 
 
 In dimension 1, Martens-de Melo-van Strien \cite{MdMvS92} showed that for any $\infty\ge r\ge 2$, for an open and dense set\footnote{whose complement is the infinite codimentional manifold formed by maps with at least one flat critical point.} of $C^r$-maps, the number of periodic points grows at most exponentially fast.

Kaloshin \cite{K99} answered a question of Artin and Mazur (in the finitely smooth case) by proving that for a dense set  $\mathcal D$ in $\Diff^r(M)$, $r< \infty$, the set $\Per_n f$ is finite for every $n$ (so equal to $\Per_n^0 f$) and its cardinality grows  at most  exponentially fast. On the other hand he  proved in \cite{K00} that whenever $2\le r<\infty$ and $\dim M\ge 2$,  \emph{a locally topologically generic diffeomorphism displays a fast growth of the number of periodic points}: 
there exists a nonempty  open set $U\subset \Diff^r(M)$,  so that for any sequence of integers $(a_n)_n$, a topologically generic $f\in U$ satisfies:
\begin{equation}\tag{$\star$}
\limsup_{n\to \infty} \frac{\Card\, P_n^0(f)}{a_n}\ge 1\; .\end{equation} 
This answered positively Problem \ref{problem1} in the case $\dim M\ge 2$ and $2\le r<\infty$. Later  Bonatti-Diaz-Ficher \cite{BDF08} extended this positive answer to the $C^1$-case in dimension $\ge 3$.  Also the recent seminal work
 of Turaev \cite{T15} also implies that among $C^\infty$-surface diffeomorphisms,  fast growth of the number of periodic points is locally a topologically generic property\footnote{
Turaev proved the existence of a locally dense set of $C^\infty$-surface diffeomorphisms which display a periodic spot: an open set 
 formed by periodic points of same period. These are  easy to  perturb to one displaying a fast growth of the number of periodic points.
%  As Kaloshin argument is based on a result by Gonchenko-Shilnikov-Turaev \cite{GST93,GST99} (stating that   degenerate semi parabolic points form a locally dense subset  of $\Diff^r(M^2)$) which holds true also when $r=\infty$, these actually imply that fast growth of the number of periodic points is locally a topologically generic property.
 }.

{\color{\blue} A new argument will enable us to carry the $C^\infty$-case in dimension $\ge 3$ which  remained open%\footnote{In \cite{GST93b}, Theorem 7,  Gonchenko-Shilnikov-Turaev theorem was claimed to be true for any dimension but no proof has ever been published.}
. More precisely, we will  show that for any $2\le r \le \infty$, among $C^r$-diffeomorphisms of  manifold of dimension $\ge 3$ among $C^\infty$-self-mapppings of surface,  the fast growth of the number of periodic points is locally a topologically generic property. This yields a full answer to Problem \ref{problem1}, in any regularity $\infty\ge r \ge 2$ and any dimension (our contribution here is the $C^\infty$-case): }
%However the $C^\infty$-case in dimension $\ge 3$ remained open\footnote{In \cite{GST93b}, Theorem 7,  Gonchenko-Shilnikov-Turaev theorem was claimed to be true for any dimension but no proof has ever been published.}. In this term, our first result accomplishes the study of problem \ref{problem1}, in any smoothness $\ge 2$ and any dimension: 
\begin{theo}\label{theoA}
Let $\infty\ge r\ge 2$ and  let $M$ be a compact manifold of dimension $n$. 

If $n=1$, Property $(AM)$ is satisfied by an open and dense set of $C^r$-self-mapping.

If $n\ge 2$, there exists a (nonempty) open set $U\subset \Diff^r(M)$ so that given any sequence $(a_n)_n$  of integers, a topologically generic $f$ in $U$ satisfies $(\star)$.
 \end{theo}
{\color{\blue} To deal with  the $C^\infty$-case, we will prove in \cref{prop1}, that it suffices to show the local density of dynamics with a normally hyperbolic periodic circle on which the dynamics acts as a smooth rotation.   Indeed such a rotation  can be perturbed to a rational one, or equivalently having an iterate equal to the identity, and so by adding a $C^r$-small sine function of high frequency, we can create a perturbation with plenty of  $n$-periodic saddle points. %which are  greater than $a_n$ for some large $n$. %As such periodic points persists by perturbation we will obtain Theorem \ref{theoA}.
In order to show that the dynamics acts as a smooth rotation on such a circle, we will use  the Arnold-Herman-Yoccoz \cref{AHthm} which explains that we only have to prescribe a Diophantine property for the rotation number of the circle dynamics.
To obtain a Diophantine rotation number, we will give evidence in \cref{Density of rotations from blender IFS} that it suffices to couple 
 a Bonatti-Diaz blender  with a north-south dynamics on a normally hyperbolic circle. However, we will use a different argument to obtain such a rotation number,  which  more sophisticated but generalizable to prove a parametric counterpart of Theorem~\ref{theoA}. 
 
%To prepare for the parametric counterpart of Theorem~\ref{theoA}, we will go beyond this argument.
 We will focus on parabolic circle diffeomorphisms: these are diffeomorphisms $g$ with a unique fixed  point $P$ and which satisfy $D_Pg=1$ and $D^2_Pg\neq 0$. Since these are easy to perturb to smooth rotations, we will prove in \cref{Densiteparabolic}  that it suffices to show the local density of dynamics with a normally hyperbolic periodic circle at which the dynamics is parabolic.  To show the local density of such assumptions,  we will introduce a new  object: the $\lambda$-blender in \cref{sec: lambda-blender IFS,def lambda blender} and introduce a new dynamical rescaling technique in \cref{Density of rotations from lambda-blender IFS,sec:Theorem  theoA}.
%. They will enable us to obtain a dense set of perturbation acting on a periodic circle as a \emph{parabolic map}. These are  circle  diffeomorphisms $g$ with a unique fixed  point $P$ and which satisfies  
% $D_Pg=1$ and $D^2_Pg\neq 0$. Then it is easy to vary continuously their Poincar\'e rotation number (which are zero). This gives perturbations $\tilde g$   Diophantine rotation numbers.
 
We will first introduce  these new techniques and concepts involving blender  (including their para version) in the easier setting of IFS defined by a semi-group of circle diffeomorphisms in \cref{sec : Locally dense properties of finitely generated group of circle diffeomorphisms}.
%
%facilitate the understanding of the concepts of $(\lambda)$-blender, their parametric versions, and the  we will start  in \cref{sec : Locally dense properties of finitely generated group of circle diffeomorphisms} by introducing their IFS versions in the setting of   finitely generated semi-group of circle diffeomorphisms. 
Recall that given a finite alphabet $\sA$, the semi-group spanned by   $(g^\sa)_{\sa\in \sA}\in\Diff^\infty(\R/\Z)^\sA$ is:
\[<g^\sa: \sa\in \sA>:= \{g^\sm:= g^{\sm_n}\circ \cdots \circ  g^{\sm_1}: \sm=(\sm_i)_{1\le i\le n}\in \sA^n, n\ge 0\}\]
%the following  and $\sm=(\sm_i)_{1\le i\le n}\in \sA^n$, put  $g^\sm:= g^{\sm_n}\circ \cdots \circ  g^{\sm_1} $.
An immediate consequence of \cref{proprot,coroImportantgroup} stated in \cref{Density of rotations from blender IFS,Density of rotations from lambda-blender IFS} is:
\begin{coro}\label{coroA}
% There is a finite set $\sA$ such that the generators  $(g^\sa)_{\sa\in \sA} \in \Diff^\omega(\R/\Z)^\sA$ spanning a semi-group containing a parabolic map and a smooth rotation form a set $D\subset \Diff^\omega(\R/\Z)^\sA$ whose $C^ r$-closure has nonn-empty interior for every $2\le r\le \infty$ or $r=\omega$.

For a finite set $\sA$, there is a  subset $D\subset \Diff^\omega(\R/\Z)^\sA$ satisfying:
\begin{enumerate} 
\item  Every  $(g^\sa)_{\sa\in \sA}\in D $ spans a semi-group which contains parabolic maps and smooth~rotations\footnote{A smooth rotation is a circle diffeomorphism conjugate to a rotation via a smooth diffeomorphism.}.
\item for any $2\le r\le \infty$ or $r=\omega$, the set $D$ is $C^r$-locally dense: its closure has nonempty interior.
\end{enumerate} 
\end{coro}
These methods might also  be also useful to prove that the bifurcation locus of some semi-groups of rational maps arising in Ising models have nonempty interior \footnote{The presence of parabolic points is used to study the boundary of the bifurcation locus in \cite{bencs2019leeyang}}. We will explain below how this is corollary might be also interesting to extend the concept of universal semi-group \cite{AST16}. 
}

 \medskip 

The study of circle diffeomorphisms (and more precisely  non-singular flows on the 2-torus) was certainly the original motivation  of Kolmogorov to introduce his new notion of typicality. Indeed being  Morse-Smale is an open and  dense (and so topologically generic) property among circle diffeomorphisms but is not typical in the sense of Arnold. This was proved by   Kolmogorov by  introducing the KAM theory (in which the Arnold-Herman-Yoccoz \cref{AHthm} belongs).
%s, see   \cref{AHthm,KAM}). 

 KAM theory was already implemented to establish  fast growth of the number of periodic points for \emph{conservative}\footnote{To the best of my knowledge, the present work is the first to use KAM for Problems \ref{problem1} and \ref{problem2}.}   dynamical systems. 
Indeed an immediate consequence of Gelfreich-Turaev theorem \cite{GT10} is that a locally tologically generic conservative diffeomorphism of surface displays a fast growth of the number of periodic points\footnote{See also  Kaloshin-Saprykina theorem \cite{KS06} which solves   Problem \ref{problem1} for conservative dynamics of $3$-manifold.}.  More recently Asaoka  \cite{As16} used also KAM Theory to show the existence of an  open set\footnote{He showed also a version of this result for analytic mappings.}  of \emph{conservative} $C^\infty$-surface diffeomorphisms in which typically in the sense of Arnold, a map displays a fast growth of the number of periodic points. 

While Asaoka gave a negative answer to the conservative counterpart of  Arnold's Problem \ref{problem2},  in the (original) dissipative setting the trend was more in favor of a positive answer. Indeed, Hunt and Kaloshin \cite{KH07,KH072} used a method described in \cite{GHK06} to show that for $\infty\ge r>1$, a \emph{prevalent}  $C^r$-diffeomorphisms satisfies:
\begin{equation}\tag{$\Diamond$}
\limsup_{\infty} \frac{\log P_n(f)}{n^{1+\delta}}= 0,\quad \forall \delta>0\; .\end{equation}
The notion of prevalence was introduced by  Hunt, Sauer and Yorke \cite{HSY92}. A property is \emph{prevalent} if \emph{roughly speaking} almost all perturbations in the embedding of a Hilbert cube at every point of a Banach manifold (like $C^r(M,M)$), the property holds true. We notice that  $(\Diamond)$ is satisfied for a prevalent diffeomorphism but not for a topologically generic diffeomorphism  (see other examples of mixed outcome in \cite{HK10}). 

However the latter did not completely solve Arnold's Problem \ref{problem2}, in particular because the notion of prevalence is \emph{a priori}  independent to the notion of typicality initially meant by Arnold. Indeed his problem was formulated for typicality in the sense of \cref{Deftyp} (see the explanation after Problem 1.1.5 in \cite{ KH07}).  In this term the second and main result of this work is surprising since it provides a negative answer to  Arnold's Problem \ref{problem2} in the case of finite smoothness:

\begin{theo}\label{theoB}
Let $\infty> r\ge 1$ and $0\le k<\infty$, let $M$ be a manifold of dimension $\ge 2$.  
Then there exists a (nonempty) open set $\hat U$ of $C^r$-families $(f_p)_{p\in B_k}$ of $C^r$-self-mappings $f_p$ of $M$ so that, for any sequence of integers $(a_n)_n$, a topologically generic $(f_p)_p\in \hat U$ consists of maps 
$f_p$ satisfying $(\star)$,  for every $p$ in the ball $B_k$.
%, the map . 
Moreover if $dim \, M\ge 3$, the set   $\hat U$ contains   families of diffeomorphisms.
\end{theo}

We will  prove this theorem by basically following  the same scheme as for Theorem \ref{theoA}. We will show that it suffices to construct a locally dense set $\hat D$ of families of  dynamics $(f_p)_{p}$ which  are normally hyperbolic at a  fibration by circles  and display the following property:
\begin{enumerate}\item[$ (\hat {\cal P})  $]  \emph{There is a finite covering $(U_i)_i$ of $B_k$ by open subsets $U_i$ of parameters $p$  for which $f_p$ acts on a   periodic fiber as a parabolic dynamics. }
\end{enumerate}
\cref{ThmS3} assets  that if such a locally dense set $D$ exists, then a topologically generic family in the interior of the closure $\mathrm{cl\,}(D)$ of $D$ displays a fast growth of the number of periodic points at \emph{every} parameter, as claimed in Theorem \ref{theoB}. The proof of \cref{ThmS3} is similar to its parameter free version (\cref{Densiteparabolic}) but will need moreover  \cref{thmS4} stating that a family of parabolic maps can be smoothly approximated by one having a Diophantine rotation number. The proof of \cref{thmS4} will occupy the whole of \cref{ProofthmS4}. The techniques of this latter proof might be generalized to describe the geometry of the parameter space of analytic circle diffeomorphisms, and more precisely how the one-co-dimensional manifolds formed the parabolic maps are analytically accumulated by those formed by Diophantine rotations, as seen numerically nearby Arnold tongues.  

To prove the local density of families of dynamics satisfying $(\hat {\cal P})$, we will introduce a new object: the  $C^r$-$\lambda$-parablender. It is a generalization of both the $\lambda$-blender and the $C^r$-parablender introduced in \cite{BE15}. Again its  IFS counterpart will be first introduced in 
 \cref{Parablender IFS,sec: lambda-parablender IFS} since easier to understand.  It is possible to use them to obtain a  parametric version of Corollary \ref{coroA}. We will focus  on their counterpart for families of differentiable dynamics of a manifold. In \cref{intrinsic def parablender} we will give the intrinsic definition of $C^r$-$(\lambda)$-parablender as embedded into a normally hyperbolic fibration for a  family of differentiable maps.  In the appendix we will give an extrinsic definition of them.  In \cref{sec:Theorem  theoB}, we will give the parametric counterpart of the argument of \cref{sec:Theorem  theoA} to achieve the proof of Theorem \ref{theoB}. \medskip
 
To conclude this introduction, let me recall that Arnold's philosophy  was not to propose \emph{problems of binary type admitting a ``yes-no" answer}, but rather to propose \emph{wide-scope programs of explorations of new mathematical (and not only mathematical) continents, where reaching new peaks reveals new perspectives, and where a preconceived formulation of problems would substantially restrict the field of investigations that
have been caused by these perspectives. [...]
Evolution is more important than achieving records,} as he explained in his preface \cite{Ar00}.

In this sense the contrast between the result of Kaloshin-Hunt and Theorem \ref{theoB} is interesting since they shed light on how an answer to a question might depend on the definition of typicality. 

Let us emphasize that the $C^\infty$-case of Problem \ref{problem2}  (or  \cite[Conj. 1994-47]{Ar00}) remains open, although in view of Theorem~\ref{theoB}, I would bet for a negative answer; I would even dare to propose:
\begin{conj}\label{conjprincipal}
For every $r\in \{1,..., \infty,\omega\}$,  there exists an open set of diffeomorphisms $U\in \Diff^r(M)$, so that given any $k\ge 0$, for any $C^r$-generic family $(f_p)_{p\in \R^k}$ with $f_p\in U$, for every $p$ small, the growth of the number of periodic points of $f_p$ is fast.
\end{conj}
{In favor of this conjecture, the local density  of conservative analytic maps satisfying $(\star)$ \cite{As16}. Also \cref{analytic} of the proof of Theorem \ref{theoA}  shows the local density of analytic dynamics with a normally hyperbolic periodic circles at which the dynamics is a smooth Diophantine rotation.} We explain there that  modulo a solution of  Problem \ref{analytic2}, this implies the existence of a locally dense of analytic (non-conservative) dynamics displaying a fast growth of the number of periodic points.  Note that a positive answer to Conjecture \ref{conjprincipal}  would give one of the very first example of an interesting property   satisfied for a typical analytic map of a compact manifold but not for a polynomial. \medskip

Also Corollary \ref{coroA}  should be useful  to prove the typicality of dynamics with universal behaviors.
Indeed parabolic maps and rotations were used in \cite{Av08,CF06} to show the existence of finitely generated groups containing any prescribed countable subset of $\Diff^\infty(\R/\Z)$. This  leads to:
\begin{conj}\label{conj semigroup}
There is a finite set $\sA$ such that for any countable subset $C\subset \Diff^\infty(\R/\Z)$, there exists $D\subset \Diff^\infty(\R/\Z)^\sA$ such that:
\begin{enumerate} 
\item if $(g^\sa)_{\sa\in \sA}\in~D$,  the semi-group $<g^\sa: \sa\in \sA>$ contains $C$.
\item for every $2\le r\le \infty$, the set $D$ is $C^r$-locally dense: $\mathrm{int\, cl\, } (D) \neq \emptyset$.
\end{enumerate} 
\end{conj} 
This conjecture applied with $C$ equal to a dense subset of $\Diff^\infty(\R/\Z)$ would imply  that a locally topologically generic free semi-group of circle diffeomorphisms is dense in  $\Diff^\infty(\R/\Z)$. This can be seen as a global version of  the main result of Asaoka-Shinohara-Turaev \cite{AST16} on 
universality of locally topologically generic  free semi-group of  interval diffeomorphisms. Universal semi-group are roughly speaking those which  contain  a dense set of germs of diffeomorphisms. 
The concept of universality  helped Turaev \cite{Tu03} to claim  that a global understanding of most of the differentiable dynamics is impossible.   However this claim can be attacked by arguing that the domains of the germs might be  ``too small to be seen''. 
It is not anymore the case with the above global formulation, which can be certainly generalized in higher dimensions. 

Also the techniques and notions introduced  in \cref{sec : Locally dense properties of finitely generated group of circle diffeomorphisms} should lead to  a natural parametric counterpart of \ref{coroA} and so 
%
%Also the techniques presented should be useful to prove a parametric counterpart of %a parametric version of Corollary \ref{coroA} would lead to a parametric generalization of 
Conjecture~\ref{conj semigroup}. Such  would be helpful to prove \cite[Conj. II.5]{Be19} stating roughly speaking that among dynamical systems of high-dimensional manifolds,  any property which is local and locally topologically generic is then locally typical in the sense of Arnold. The latter conjecture was motivated by our program on emergence \cite{Be16}.

\textbf{Acknowledgment.}  {I am thankful to Abed Bounemoura,   H\r akan Eliasson, Bassam Fayad, Vadim Kaloshin and Rafaël Krikorian for our interesting conversations. I am grateful to Sylvain Crovisier and Enrique Pujals for our inspiring discussions and Dimitry Turaev for his important comments on the parabolic renormalization for circle diffeomorphisms. {  I thank Masayuki  Asaoka for pointing me a minor mistake in a Lemma, Romain Dujardin for his advice on the introduction and Sébastien Biebler for his careful proofreading of the whole manuscript. 
I thanks the referees for their advices.
The  author  was  partially  supported  by  the  ERC  project  818737  Emergence  of  wild  differentiable  dynamical systems.}}

\section{From parabolic maps to fast growths of the number of periodic points}\label{section:From parabolic maps to fast growth of the number of periodic points}
In this section we show that in order to prove Theorems \ref{theoA} and 
\ref{theoB}, it suffices to construct a normally hyperbolic fibration by circles, at which a dense set of perturbations act  parabolically  on some periodic fibers. This is stated at \cref{Densiteparabolic,ThmS3}. The statements and the proof of these theorems need a few results. First we recall the notion of Poincaré rotation number and Arnold-Herman-Yoccoz' \cref{AHthm} in the next subsection. In subsection \ref{Dynamics conjugated to rotation}, we will recall the notion of parabolic map and state our new \cref{thmS4} on perturbations of families of parabolic maps which are smooth rotations. Then we will recall Hirsch-Pugh-Shub \cref{HPS} and its counterpart for endomorphisms on the persistence of normally hyperbolic fibration. In the last subsections \ref{section:suffi cond  theoA} and \ref{section:suffi cond  theoB}, we will  state and prove \cref{Densiteparabolic,ThmS3}.

\subsection{Smooth rotations and parabolic maps of the circle}\label{Dynamics conjugated to rotation}
Given a homeomorphism $g\in \Diff^0(\R/\Z)$ of the circle $\R/\Z$, one defines its rotation number $\rho_{g}$ as follows.  We fix 
 a lifting $G\in \Diff^0(\R)$ of $g$ for the covering   $\pi\colon \R\to \R/\Z$. 
Then Poincar\'e proved that the limit $\rho_G = \lim_{n\to \infty} G^n(0)/n$ is well defined and its projection  $\rho_g= \pi(\rho_G)$ does not depend on the lifting $G$ of $g$. The \emph{rotation number} of $g$ is $\rho_g$. It is easy to show that the rotation number depends continuously on $g$. The  number $\rho_g\in \R$ is \emph{Diophantine}, if there exist $\tau>0$ and $C>0$  such that: 
\[ |q \rho_g -p |\ge C q^{-\tau}\; , \quad \forall p,q\in \N\setminus \{0\} \; .\]
Let us recall that the set of Diophantine numbers is of full Lebesgue measure. Here is  Yoccoz' improvement of Herman's theorem of the Arnold  Conjecture:
\begin{theorem}[Arnold-Herman-Yoccoz \cite{He79,Yo84}]\label{AHthm}
If the rotation number $\rho$ of  $g\in \Diff^\infty(\R/\Z)$ is  Diophantine, then $g$ is conjugate to the rotation $R_\rho$  of angle $\rho$ via $h\in \Diff^\infty(\R/\Z)$:
\[h\circ g\circ h^{-1} = R_\rho\; .\]
Moreover, if $(g_p)_p$ is a $C^\infty$-family of diffeomorphisms with constant rotation number $\rho$ which is Diophantine,  then $(g_p)_p$ is conjugate to $R_\rho$ via  a $C^\infty$-family $(h_p)_p$ of diffeomorphisms $h_p$:
\[h_p\circ g_p\circ h_p^{-1} = R_\rho\; .\]
\end{theorem}
\begin{proof}
The first part of this theorem is the main result  of \cite{Yo84}. 
Since the conjugacy is uniquely defined up to a composition with a rotation, the second part of this theorem is a local problem. Hence by the first part, it suffices to show that if $(g_p)_p$ satisfies moreover $g_0= R_\rho$, then the family of conjugacy $(h_p)_p$ can be chosen smooth in a neighborhood of $p=0$.  This is a direct consequence of Theorem 3.1.1 of \cite{Bo84}. 
\end{proof}

%\subsection{Parabolic maps of the circle}\label{section:Parabolic maps of the circle}

A key new idea in this work is to exhibit circle diffeomorphisms with 
Diophantine rotation number by creating first \emph{parabolic diffeomorphisms of the circle}: 
\begin{definition} A $C^2$-diffeomorphism $g$ of a circle $\R/\Z$ is \emph{parabolic} if it displays a unique fixed point $p\in \R/\Z$, and this fixed point is non-degenerate parabolic:
\[g(p)=p,\quad D_pg= 1\quad ,\quad D^2_p g\not= 0\quad .\] 
\end{definition}
Using parabolic maps to obtain dynamics smoothly conjugate to Diophantine rotations might sound anti-intuitive since their rotation number are of two very different kind. Nevertheless the interest of parabolic maps of the circle is that they have a simple geometric definition and  are easy to perturb to irrational rotations.   Indeed  if $g$ is a $C^r$-parabolic circle map, for $r\ge 2$, then its rotation number is $0$. By  reversing the circle orientation if necessary, we can assume that $D^2_p g>0$. Then $g$ has a lifting $G$ such that $x\in \R\mapsto G(x)-x$ is non-negative and less than $1$. Thus for every $\epsilon>0$, the lifting $x\in \R\mapsto G(x)-x+\epsilon$ of the composition $R_\epsilon\circ g$ takes its value in $(0,1)$ and so $R_\epsilon\circ g$ has no fixed point, and hence a non-zero rotation number. Then, by continuity of the rotation number and density of Diophantine numbers in $\R$, we can choose $\epsilon>0$ arbitrarily small so that the rotation number $\rho(\epsilon)$ of $R_\epsilon \circ g$ is Diophantine.  This proves:
\begin{proposition}\label{prethmS4}
For every $r\ge 2$, the set $D^r$ of $C^r$-circle maps with Diophantine rotation number accumulates on the set $P^r$ of $C^r$-parabolic maps: $cl(D^r)\supset P^r\; .$
\end{proposition}

The above argument is topological. Hence the following is a non trivial extension of the latter proposition for parameter families:
\begin{theorem}\label{thmS4}
Let $k\in \N$ and let $B'\Subset B\subset \R^k$ be open  subsets. Given any $C^\infty$-family 
 $(g_p)_{p\in B}$  of circle maps so that for every $p\in B$ the map $g_p$ is parabolic, there exists an arbitrarily small  Diophantine number $\alpha>0$,  there exists a small $C^\infty$-perturbation $(\tilde g_p)_{p\in B}$ of $(g_p)_{p\in B}$ so that the rotation number of $\tilde g_p$ is $\alpha$ for every $p\in B'$.
% \begin{itemize}
% \item the rotation number of $\tilde g_p$ is $\alpha$ for every $p\in B'$.
% \item the family $(\tilde g_p)_{p\in B}$ is of class $C^\infty$.  
% \end{itemize}
\end{theorem}
%\marginal{Attention il faut revenir au $C^r$}
The proof this theorem will be done in  \cref{ProofthmS4}
using a parabolic renormalization. %, and the Arnold-Herman Theorem. 
\subsection{Normally hyperbolic fibrations}\label{sec: Normaly hyperbolic fibrations}
Let $2\le r\le \infty$,  let $f$ be a 
$C^r$-differentiable map of a manifold  $M$. 
The proofs of Theorems \ref{theoA} and \ref{theoB} involve   a continuous family $(L_y)_{y\in \Sigma}$ of disjoint, $C^r$-embedded submanifolds $L_y\subset M$, indexed by a compact  set $\Sigma$.  This defines a fibration $\mathcal L=\bigcup_{y\in \Sigma} L_y\to \Sigma$.    The map $f\in C^r(M,M)$ \emph{leaves invariant} the fibration $\mathcal L$ if for every $y\in \Sigma$, there exists $\sigma(y)\in \Sigma$ such that $f$ sends $L_y$ into $L_{\sigma(y)}$. Hence the following diagram commutes:
\[\begin{array}{ccc}
\cal L& \stackrel{f}{ \rightarrow }& \cal L\\
\downarrow  & & \downarrow\\
\Sigma&  \stackrel{\sigma}{ \rightarrow }& \Sigma\end{array}
\]
Then observe that $Df$ leaves invariant the tangent bundle of the fibers $T\mathcal L:=\bigcup_{y\in \Sigma} T L_y$.  Hence the action $[Df]$ of $Df$ on the normal bundle $\mathcal N=(TM|\mathcal L)/T\mathcal L$ is well defined.  
\begin{definition}
 For $R\ge 1$, an $f$-invariant fibration $\cal L$  is \emph{$R$-normally hyperbolic} if:
\begin{itemize}
\item  $[Df]$ leaves  invariant a splitting $E^s\oplus E^u = \mathcal N$:  $[Df](E^s)\subset E^s\quad \text{and}\quad [Df](E^u)\subset E^u\; ,$
\item There exists $N\ge 1$ such that  for all unit vectors  $v_s\in E^s$, $v_u\in E^u$, and $v\in T \mathcal L$ , it holds:
\[\|[Df^N](v_s)\|<1\quad ,
\quad \|[Df^N](v_u)\|>1 \quad 
 \text{and}  \quad \|[Df^N](v_u)]\|>\|Df^N(v)\|^R>\|[Df^N](v_s)\|\; .\]
\end{itemize}
The dynamics is \emph{$R$-normally expanding} at $\mathcal L$ if $E^s=0$.
\end{definition}
Normally hyperbolic fibrations are important since they are persistent:
\begin{theorem}[\cite{HPS}, \cite{berlam}]\label{HPS}
Let $r\ge 1$ and $1\le R\le r$.  Let $f$ be a $C^r$-map of $M$ which is $R$-normally hyperbolic at the bundle $\mathcal L=\bigcup_{y\in \Sigma} L_y$. Moreover, if $f$ is not a diffeomorphism, we assume that $f$ is $R$-normally expanding at $\mathcal L$.
% which is invertible if $\dim M\ge 3$. Suppose that $f$ is normally hyperbolic (and expanding if $\dim \, M=2$)  at a fibration $\mathcal L=\bigcup_{y\in \Sigma} N_y$.  
Then, for any $C^r$-perturbation $\tilde f$ of $f$, there exists a continuous family $(\tilde L_y)_{y\in \Sigma}$ of disjoint $C^r$-submanifolds so that:
\begin{itemize}
\item $\tilde f(\tilde L_y)= \tilde L_{\sigma(y)}$ for every $y\in \Sigma$.
\item $\tilde L_y$ is $C^r$-close to $L_y$ for every $y\in \Sigma$.
\end{itemize}
\end{theorem}

For our purposes, the map $\sigma$ of  $\Sigma$ will be  conjugate  to a full shift $\sigma$  on a finite alphabet $\sA$.  Then for a dense set of $y\in \Sigma$, the fiber $L_y$ is \emph{$q$-periodic}: $f^q(L_y)=L_y$ for some $q\ge 1$. 

\subsection{Conditions implying typicalities of  fast growths of the number of periodic points}
\label{section:suffi cond  theoA}
We recall that a map $f$ of a circle is a $C^r$-rotation if there is $\alpha\in \R/\Z$ such that $f$ is $C^r$-conjugate to the rotation $R_\alpha$ of angle $\alpha$. 
The following will be used to prove Theorem \ref{theoA}:
\begin{proposition}\label{prop1}
Let $2\le r\le \infty$ and $U$  a nonempty subset  of $C^r(M,M)$.
Assume that there is a dense set $D\subset U$ formed by maps $f$ 
displaying a $q$-periodic, normally hyperbolic, $C^r$-circle $\T\subset M$ such  that  $f^q|\T$ is a $C^r$-rotation.
Then for any $(a_n)_n\in \N^\N$, a topologically generic $f\in U$ satisfies:
\[\limsup_{n\to \infty} \frac1{a_n}\Card \Per_n^0(f)\ge 1\; .\]%(\star)$. 
\end{proposition}
\begin{proof}Let $f\in D$ and $\T$ be as in the statement. Let $q\ge 1$ be the minimal  period of $\T$:   $f^j(\T)$ is disjoint from $\T$ for every $1\le j<q$. 
Thus, there exists a small $C^r$-perturbation $\tilde f$ of $f$, the map $\tilde f^q$ leaves invariant $\T$ and $\tilde f^q|\T$ is a $C^r$-rational rotation. Then there exists a  $q'\in q\N\setminus \{0\}$ minimal so that $\tilde f^{q'}$ leaves invariant  $\T$, and $\tilde f^{q'}|\T=id_{\T}$.
Observe that  $q'$ must be  large when $\tilde f$ is close to $f$.

Also there exist non trivial segments $J\Subset I\subset \T$ such that  $(\tilde f^i(I))_{1\le i\le q}$ is a disjoint family.
Let $\rho \in C^\infty( \T , [0,1] )$ be a function supported by $I$ and satisfying  $\rho | J = 1$.  
We handle a small perturbation $\tilde f'$ of $\tilde f$ 
supported by a small neighborhood of $I$ and in the complement of  $\bigcup_{1\le k< q}\tilde f^k(I)$  such that $\tilde f'^q|J= x\in J \mapsto x+\epsilon \rho(x) \sin(2\pi \cdot a_{q} \cdot x/|J|)$ for an identification of $\T$ with $\R/\Z$ such that the left endpoint of $J$ is $0$ and $\epsilon>0$ is small.

Note that $\tilde f'$ displays at least  $ a_{q}$ hyperbolic periodic points of period $q$. By normal hyperbolicity, these periodic points are hyperbolic and so persist for small perturbations of $\tilde f'$.

This proves for every $N\ge 1$, the existence of an open and dense set $U_{N}\subset U$ formed by maps $f\in U_N$ displaying at least $ a_{q}$ hyperbolic periodic points of period $q$ for a certain $q \ge  N$.  The topologically generic set is $\mathcal R:=\bigcap_{N\ge 0} U_N$. As hyperbolic $q$-periodic points are isolated in the set of $q$ periodic points, 
  for every $f\in \mathcal  R$, there exists $q$ arbitrarily large so that $\Card\, \Per^0_q\, f \ge a_q$. 
\end{proof}
Although by KAM theory, it sounds very intuitive that for many normally hyperbolic circle bundles, the hypothesis of the above proposition holds true. Nevertheless this intuition is not satisfied on a topologically generic set. 
Indeed an open and dense set of $C^r$-diffeomorphisms  of the circle are Morse-Smale. As the set of periodic fibers  in a normally hyperbolic  bundle is countable, a subset $D$ of dynamics satisfying the hypotheses of \cref{prop1} must be topologically meager in $C^r(M,M)$! 

To show that the existence of a set $D$ satisfying the hypotheses of \cref{prop1} we will introduce a technique which uses a Bonatti-D\'iaz blender and a robust heterodimensional cycle. Actually, we are going to prove an equivalent  statement  to these hypotheses, 
but more convenient to be used in for the parameter counterpart Theorem  \ref{theoB} of Theorem  \ref{theoA}.
Here is the statement: 
\begin{theorem}\label{Densiteparabolic}
Let $2\le r\le \infty$ and $D\subset C^r(M,M)$ the subset of maps $f$ 
displaying a $q$-periodic, normally hyperbolic,  $C^r$-circle $\T\subset M$ such  that  $f^q|\T$ is parabolic.
Then for any $(a_n)_n\in \N^\N$, a topologically generic $f$ in the interior of $cl(D)$ satisfies:
\[\limsup_{n\to \infty} \frac1{a_n}\Card \Per_n^0(f)\ge 1\; .\]
\end{theorem}
\begin{proof} 
Let us show that we can assume that  $f^{q_i} | \T $ is of class $C^\infty$. First note that  the submanifold $\T $ is $r$-normally hyperbolic.   Thus by  \cite[\textsection Forced smoothness]{HPS}, it is of class $C^r$.  Therefore up to a $C^r$-coordinate change,  we can assume that $\T$ is of class $C^\infty$.  Let $C \in \T $ be the  parabolic periodic point and let us identify $\T  $ with $\R/\Z$. Note that  $x\in \T \mapsto f^{q}(x)-x$ is a unimodal map nearby $C $ and $0$ is its unique critical value.  We now perform a smoothing of $ f $ so that  $\T $ remains $q $-periodic.  Note that  the induced perturbation of  $x\in \T \mapsto f^{q } (x)-x$ is still  unimodal nearby $C$ with a small critical value. Thus by composing with a small local translation along this fiber, we can restor   the critical value to $0$. We obtained a $C^r$-perturbation of $ f $ whose $q_i$ iterate leaves invariant $  \T  $, and its restriction at it is both parabolic and of class $C^\infty$.

%As having a normally hyperbolic submanifold is an open property and since having a parabolic dynamics on it is a one-codimensional property, we can smooth 
%the element of $D$ and then compose to them a small translation nearby the circle to  assume that $D$ is a subset of $C^\infty(M,M)$. 

%Let $f\in D$ whose restriction to  a $q$-periodic  normally hyperbolic circle $\T$ is parabolic. For every $R\ge 1$, the dynamics $f^p$ is actually $R$-normally hyperbolic at $\T$ (for an adapted metric). As $f$ is of class $C^\infty$, the periodic circle is actually of regularity $C^{R}$ by \cite[\textsection Forced smoothness]{HPS}, and so of class $C^\infty$.

Then Proposition \ref{prethmS4} implies the existence of a $C^r$-perturbation at which  the dynamics on the normally hyperbolic, periodic, $C^\infty$-circle displays a Diophantine rotation number. 
Therefore, by Yoccoz' Theorem \ref{AHthm}, 
the dynamics on the normally hyperbolic, periodic, $C^\infty$-circle
is ($C^\infty$-conjugated to) an irrational rotation. Then   Proposition \ref{prop1} implies the sought result.
\end{proof}
 
Theorem \ref{theoA} will be proved by showing that the assumptions of \cref{Densiteparabolic} are satisfied. In order to do so we will introduce a new object, the $\lambda$-blender. This object  will be coupled to a North-South dynamics of the circle. Then a new technique involving a dynamical rescaling will reveal the density of the parabolic maps. These concepts and techniques will be first introduced in the context of semi-group of circle maps in \cref{sec : Locally dense properties of finitely generated group of circle diffeomorphisms}, then the $\lambda$-blender IFS will be embedded in a normally hyperbolic fibration in \cref{Embedding a semi-group into a normally hyperbolic fibration,def lambda blender}    and used in \cref{sec:Theorem  theoA} to prove Theorem~\ref{theoA}.

\label{section:suffi cond  theoB}\medskip

In parallel to the proof  Theorem \ref{theoA}, we will introduce the parametric counterpart of each studied object. The parametric counterpart of the $\lambda$-blender will be the $\lambda$-$C^r$-parablender. It will be introduced in \cref{sec : Locally dense properties of finitely generated group of circle diffeomorphisms} for semi-groups of families of maps, and then generalized for skew products over a shift in \cref{intrinsic def parablender}. This will be used in  \cref{sec:Theorem  theoB} to prove 
 Theorem \ref{theoB} by proving that the assumptions of the next theorem are satisfied.

 Let $k\ge 1$ and recall that $B_k$ denotes the closed unit ball of $\R^k$.  Let $2\le r<\infty$ and let $\widehat {\End^r_k}(M)$  be the set of $C^r$-families $(f_p)_{p\in B_k}$ of self-maps $f_p$ of $M$. 

 \begin{theorem}\label{ThmS3} 
Let $\hat D $ be a subset of  families 
 $(f_p)_{p\in B_k}\in  \widehat {\End^r_k}(M)$ 
displaying a persistent  $r$-normally hyperbolic fibration  $(\cal L_p)_{p\in B_k}$  by circles  with the following property: 
\begin{itemize}
\item[$(\hat {\cal P})$] The set   $B_k$ is covered by open subsets $U_i$ associated to a  $q_i$-periodic fiber  $\T_{y_i, p}$ of $\cal L$ at which 
 the restriction  $f^{q_i}_p|\T_{y_i, p}$ is parabolic for every $p\in U_i$.   
\end{itemize}
Then  for any $(a_n)_n\in \N^\N$, a $C^r$-topologically generic $(f_p)_{p\in B_k}$ in the interior of  $cl(\hat D)$ satisfies:
\[\limsup_{n\to \infty} \frac1{a_n}\Card \Per_n^0(f_p)\ge 1\quad \forall p\in B_k \; .\]
%Moreover $\hat U$ consists of families of  diffeomorphisms if $dim\, M\ge 3$.
    \end{theorem}
  \begin{proof}[Proof of \cref{ThmS3}] Let $(f_p)_p\in \hat D$. By definition of persistence (see \cref{HPS}), for every $i$, the circle $\T_{y_i, p} $ depends continuously on $p\in B_k$.    By compactness of $B_k$,  we can assume that the covering $(U_i)_{i\in I}$ is finite. Up to merge the elements $U_i$ of the covering associated to a same torus,  
  %  replacing $U_i$ by $U_i\cup U_j$,
   we can assume the orbits of  $\T_{y_i, p}$ and $\T_{y_j, p}$ are disjoint for $i\not= j\in I $ and $p\in U_i\cap U_j$.  
 By continuity, the distance between  $\T_{y_i, p}$ and $\T_{y_j, p}$ is positive for $i\not= j\in I $ and $p\in B_k$, and so it is uniformly bounded from below. 

Let us show that we can assume that  $(f^{q_i}_p| \T_{y_i, p})_p$ is of class $C^\infty$. First note that  the submanifold $\bigcup_{p\in U_i} \{p\}\times \T_{y_i, p}$ is $r$-normally hyperbolic for $(p,z)\mapsto (p, f^{q_i}_p(z))$.   Thus by  \cite[\textsection Forced smoothness]{HPS},
the family $(\T_{y_i, p})_{p\in U_i}$ is of class $C^{r}$. Therefore up to 
conjugate $(f_p)_p$ with a $C^r$-family of conjugacy, we can assume that $(\T_{y_i, p})_{p\in U_i}$ is of class $C^\infty$.  Let $C_i(p)\in \T_{y_i, p}$ be the  parabolic periodic point and let us identify $\T_{y_i, p} $ with $\R/\Z$. Note that  $x\in \T_{y_i, p}\mapsto f^{q_i}_p(x)-x$ is a unimodal map nearby $C_i(p)$ and $0$ is its unique critical value.  We now perform a smoothing of $(f_p)_p$ so that  $\T_{y_i, p} $ remains $q_i$-periodic.  Note that  the induced perturbation of  $x\in \T_{y_i, p}\mapsto f^{q_i}_p(x)-x$ is still  unimodal nearby $C_i(p)$, but its critical value is in general   different to $0$. However the critical values of a $C^r$-family of unimodal maps depend $C^r$ on the parameter   and so the critical value of  the smoothed map is a $C^r$-small function of $p$. Thus by composing with a local  translation along this fiber, we can restor the critical value to $0$. Then we obtain a $C^r$-perturbation of $(f_p)_p$ whose $q_i$ iterate leaves invariant $ (\T_{y_i, p})_{p\in U_i}$, and its restriction at it is both   parabolic and of class $C^\infty$.
%   We notice that for every $i$ and $r'\ge 1$, the submanifold $\bigcup_{p\in U_i} \{p\}\times \T_{y_i, p}$ is $r'$-normally hyperbolic for $(p,z)\mapsto (p, f^{q_i}_p(z))$.   Then by  \cite[\textsection Forced smoothness]{HPS},
%the family $(\T_{y_i, p})_{p\in U_i}$ is of class $C^{r'}$ for every $r'\ge 2$ and so of class $C^\infty$.
%Thus we can perturb independently the restriction of the dynamics on each of these families $(\T_{y_i, p})_{p\in U_i}$.
  Thus \cref{thmS4} implies:
\begin{fact}\label{Fact40} There exists a $C^r$-perturbation $(\breve f_p)_{p\in B_k}$ of $(  f_p)_{p\in B_k}$ such that for every $i$,  for every  $p\in U_i$, 
the map $\breve  f^{q_i}_p$ is normally hyperbolic at
 $\T_{y_i, p}$, the family of restrictions $(\breve  f^{q_i}_p|\T_{y_i, p})_{p\in U_i}$ is of class $C^\infty$   and the rotation number of $\breve f^{q_i}_p|\T_{y_i, p}$ is a Diophantine number $\alpha_i$ independent of $p\in U_i$. 
\end{fact}
Then by the second part of Herman-Yoccoz' Theorem \ref{AHthm}, the family  $\breve f^{q_i}_p|\T_{y_i, p}$ is $C^\infty$-conjugated to the rotation $R_{\alpha_i}$ for every $p\in U_i$, and the conjugacy $h_{i\, a}$ depends infinitely smoothly on $p$. Hence by the same reasoning as for the proof of Proposition \ref{prop1}, 
there exists an integer $r_i\ge 1$  arbitrarily large and  a $C^r$-perturbation $(\mathring f_p)_p$ of $(f_p)_p$ 
so that  for every $p\in U_i$,   $\mathring f^{q_i}_p$ leaves invariant $\T_{y_i, p}$ and:
\[\mathring f^{q_i\cdot r_i}_p|\T_{y_i, p} = id_{\T_{y_i, p} }\qand 
\mathring f^{q_i\cdot r}_p|\T_{y_i, p} \neq  id_{\T_{y_i, p} }, \; \forall 1\le r<r_i\; .\]
Thus, for every $i$, there exists a non trivial segment $  J_{i\, p}\Subset \T_{y_i, p}$ depending smoothly on $p$ and such that the  family of segments  $(\mathring f_p^k(J_{i\, p})) _{0\le k <q_i\cdot r_i}$ is disjoint   for every $p\in U_i$.

%For every $i\in I$, let $\rho_i: \bigcup_j \bigcup_{p\in B_j} \{p\} \times \T_{y_j, p}\to [0, \infty)$ be a smooth function  such that  
%$\rho_i(p,x)> 0$ iff $(p,x)\in \bigcup_{p\in B_k} \{p\} \times \T_{y_i, p}$. As in the proof of Proposition \ref{prop1}, we handle a small perturbation $\tilde f_p$ of $\breve f_p$ supported by a small neighborhood of $(J_{i\, p})_p$ (and so disjoint to the union of  $\bigcup_{1\le k< q_i\cdot p_i}f^k(\T_{y_i\, p})$ with the orbit of the others $\T_{y_j, p}$ for $j\not=i$) 

Hence we can perform a last perturbation $(\tilde f_p)_p$ of $(\mathring f_p)_p$ so that for every $p\in U_i$, it holds:  $\tilde f_p^{r_i\cdot q_i} |J_{i\, p}= x \mapsto x+\epsilon \rho_i (x) \sin(2\pi \cdot a_{r_i\cdot q_i} \cdot x/|J_{i\, p}|)$ for an identification of $\T_{y_i\, p}$ with $\R/\Z$ such that the left endpoint of $J_{i\, p}$ is $0$ and $\epsilon>0$ is small.

Then for every $i$ and $p\in U_i$, the map  $\tilde f_p$ displays at least  $   a_{q_i\cdot r_i} $ saddle points of period $r_i\cdot q_i$.  As these periodic points persist for small perturbations of $\tilde f_p$, this proves the existence of an open and dense set $\hat {\cal M}_{N}\subset \mathrm{int}\, (\mathrm{ cl} \hat D)=: \hat {\cal M}$ such that that for every $(\tilde f_p)_p\in \hat {\cal M}_N$, for every $p\in B_k$, the map $\tilde f_p$ displays at least $a_{r_i\cdot q_i}$ saddle points of period $r_i\cdot q_i\ge N$. 

 Note that the intersection $\hat{\mathcal R}:=\bigcap_{N\ge 0} \hat {\cal M}_N$ is  $C^r$-topologically generic in $\hat {\cal M}$  and  for every $(f_p)_p\in \hat {\mathcal  R}$, for every $p\in B_k$, there exists $n$ arbitrarily large such that $\Card\, \Per^0_n\, f_p \ge a_n$. 
\end{proof}
In the $C^1$-topology, we can perturb a neutral fixed point  to make it locally equal to the identity. Thus we can perform the same trick by adding a small oscillation of high frequency to obtain:
\begin{corollary}\label{ThmS3coro} 
When $r=1$, the conclusion of  \cref{ThmS3} holds true if we replace $(\hat {\cal P})$ by:
\begin{itemize}
\item[$(\hat {\cal P}')$] The set   $B_k$ is covered by open subsets $U_i$ associated to a  $q_i$-periodic fiber  $\T_{y_i, p}$ of $\cal L$ at which 
 the restriction  $f^{q_i}_p|\T_{y_i, p}$ displays a parabolic fixed point varying continuously  for every $p\in U_i$.   
\end{itemize}
%
%only ask that  the restriction  $f^{q_i}_p|\T_{y_i, p}$ displays a parabolic fixed point depending continuously on $p\in U_i$ instead of $f^{q_i}_p|\T_{y_i, p}$ being parabolic.
\end{corollary}
\section{Locally dense properties of finitely generated groups of circle diffeomorphisms}\label{sec : Locally dense properties of finitely generated group of circle diffeomorphisms}
The aim of this section is to develop and introduce notions and techniques  which will be generalized to  construct  an open set of self-maps and families of self-maps satisfying the assumptions of 
\cref{Densiteparabolic,ThmS3}. In particular,  we will introduce two new objects, the $\lambda$-blender and the $C^r$-$\lambda$-parablender, which are  development of the Bonatti-Diaz blender \cite{BD96} and the $C^r$-parablender introduced in \cite{BE15} (see also \cite{Be16}). To make these notions more apprehensive, we push forward some of the ingredients of \cite{BCP16} to develop the notions of  blender and parablenders for IFS in \cref{Blender IFS,Parablender IFS}, and introduce their $\lambda$-version in \cref{sec: lambda-blender IFS,sec: lambda-parablender IFS}. We will see how to use them to find locally dense sets of semi-groups of circle diffeomorphisms  with rotations in \cref{Density of rotations from blender IFS} and parabolic maps in \cref{Density of rotations from lambda-blender IFS}.

\subsection{Blender and parablender IFS}\label{Blender IFS}
Let $M$ be a manifold and let $\End^r(M)$ be the space of self-maps of $M$ of class $C^r$ for every $\infty\ge r\ge 0$ {or $r=\omega$}.  Given a semi-group $G\subset \End^0(M)$,  a compact subset $\Lambda$ is $G$-\emph{invariant} if 
$\Lambda= \bigcup_{g\in G} g(\Lambda)$. It is \emph{transitive} if it displays a dense $G$-orbit.  The set $\Lambda$  is a  \emph{contracting attractor} if 
$G\subset \End^1(M)$ and   $g|\Lambda$ is  contracting for every $g\in G$. 
 
  If there exists a finite set $\sB$ and  $(g^\sb)_{\sb\in \sB}\in \End^1(M)^\sB$ which generates $G$, then $\Lambda$ is a \emph{contracting attractor for the iterated function system} (IFS) $(g^\sb)_{\sb\in \sB}$. Given $\sm\in \sB^{(\N)}$ we denote $|\sm|\in \N$ its number of letters in $\sB$.  The following is the IFS counterpart of the Bonatti-Diaz blender \cite{BD96}.
\begin{definition}[$C^1$-Blender IFS] 
A  contracting attractor $\Lambda\subset M$ for an IFS   $(g^{\sb})_{  \sb \in \sB}\in \End^1(M)^\sB$ is a   \emph{blender} if its  interior is nonempty and  for any subset $K\Subset \Lambda$,   every $C^1$-perturbation of $(\tilde g^\sb)_{\sb\in \sB}$ has a contracting attractor which contains $K$. 
\end{definition}

\begin{example}\label{blender IFS}
Let $\sB_0:=\{-, +\}$. The set $\Lambda= [-1,1]$ is a  blender for the IFS $(g^\sb)_\sb$ with   $g^{-}: x\in \R
%P^1(\R)
\mapsto  \frac23 (x+1) -1\in \R
%\in P^1(\R)
$ and $g^{+}: x\in \R
%P^1(\R)
\mapsto \frac23 (x-1)+1\in \R
%\in P^1(\R)
$. 
\end{example}
 \begin{proof}
For any $\eta>0$,  the convex subset   $K:=[-1+\eta, 1-\eta]$ is sent by   $g^{-}$ and $g^{+}$ onto respectively $[-1+\frac23\eta, \frac13-\frac23\eta]$ and $[-\frac13+\frac23\eta , 1-\frac23\eta]$. Thus the following \emph{covering property} holds true: 
\[ K  \subset  \mathrm{int\, }  g^+ (K)  \cup\mathrm{int\, }  g^-(K) \; .\] 
This property is stable for any $C^1$-perturbation $(\tilde g^+, \tilde g^-)$ of $(g^+, g^-)$.  Hence, for any perturbation of the dynamics, for any $z_0\in K$, there exists a \emph{preorbit} $(z_i)_{i\le 0}$ such that that $z_i$ belongs to $K$. By preorbit we mean that $z_{i+1}= g^{\sa_i}(z_i)$ for a certain $\sa_i \in \sB_0$. Thus  
the continuation $\tilde \Lambda$ of $\Lambda$  contains $K$. 
The transitivity is left as an exercise to the reader. 
\end{proof}

%\subsection{Parablender IFS}
\label{Parablender IFS}
For $1\le r<\infty$, we now study semi-groups of parameter $C^r$-families $(g_p)_{p\in B_k}$ of self-maps  $g_p\in \End^r(M)$ of a manifold $M$, parametrized by the closed unit $k$-ball $B_k$. 
We recall that $\widehat {\End_k^r}(M)$ denotes the space of such families (which is itself a semi-group for the composition rule). 
As we will work with $C^r$-perturbation,  we shall deal with the action of such families  on \emph{$C^r$-jets}. We recall that the 
\emph{$C^r$-jet at $p_0\in B_k$} of a   $C^r$-family of points $z = (z_p)_{p\in B_k }$ is:
\label{notationJdM}
\[J^r_{p_0}   z= \sum_{j=0}^r \frac{\partial^j_p z_{p_0}}{j!}  (p-p_0)^{\otimes j}\; .\]
  Let $J^r_{p_0} M$ be the space of $C^r$-jets at $p_0$ of  $C^r$-family $z = (z_p)_{p\in B_k }$ of points  $z_p\in M$.  
We notice that any $C^{r}$-family $  g = (g_p)_p\in  \widehat {\End_k^r}(M)$  acts canonically on $J^r_{p_0} M$ as:
\[J^r_{p_0}  g \colon  J^r_{p_0}( z_p)_p \in J^r_{p_0}M\mapsto 
J^r_{p_0} (g_p(z_p))_p\in J^r_{p_0} M\; .\]
%We note $\Pi: J^r_{p_0}M\to M$ the  first coordinate projection.

Let $\sB$ be a finite set and for every $\sb\in \sB$, let  $g^\sb =(g^\sb_p)_{p\in B_k}\in \widehat {\End_k^r}(M)$. 
 Let $p_0\in B_k$.  We recall that if $\Lambda_{p_0}$ is a contracting attractor for  the IFS  $(g_{p_0}^\sb)_{\sb\in \sB}$. 
Then for every  point $\Omega$ nearby $\Lambda_{p_0}$, the set $\Lambda_{p_0}$ consists of  the points $X_{p_0}(\u \sb) = \lim_{i \to -\infty}  g_{p_0}^{\sb_{-i} \cdots \sb_{-1}}(\Omega)$ among  $\u \sb =(\sb_i)_{i<0}\in \sB^{\Z^-}$. 
This limit does not depend on $\Omega$. Also for every $p$ nearby $p_0$, the set of points $X_{p}(\u \sb) = \lim_{\infty}  g_{p}^{\sb_{-i} \cdots \sb_{-1}}(\Omega)$ among  $\u \sb  \in \sB^{\Z^-}$ is a contracting attractor for $ (g^\sb_p)_{\sb \in \sB}$.  The 
family $(\Lambda_p)_p$ is call the \emph{continuation} of $\Lambda_{p_0}$.  Furthermore  the  family $(X_{p}(\u \sb))_p$  is of class $C^r$ for every $\u \sb  \in \sB^{\Z^-}$ and depends continuously on $\u \sb$. We consider the subset   
\[   J^r_{p_0}\Lambda:=   \{J^r_{p_0} X(\u \sb): {\u \sb \in \sB^{\Z^-}}\}\subset J^r_{p_0} M\; .\]
This set is compact, invariant, and transitive for the IFS  $(J^r_{p_0}  g^\sb)_{\sb \in \sB}$.    Roughly speaking, the continuation $(\Lambda_p)_p$ is a $C^r$-parablender at $p_0$ if $J ^r_{p_0}\Lambda $ is a blender for $(J^r_{p_0}  g^\sb)_{\sb \in \sB}$. However there is one tricky point:  the maps $J^r_{p_0}  g^\sb$ are in general only continuous and not differentiable nor contracting at $J^r_{p_0}\Lambda$. 
However these maps are \emph{topologicaly contracting}:

%two following properties remain true:
\begin{proposition} \label{prop 1} There exists a neighborhood    $U$ of $J^r_{p_0}\Lambda$ such that for $(\tilde g^\sb)_\sb$  $C^r$-close to $(g^\sb)_\sb$:
\begin{enumerate}
\item   When the length of $\sm \in \sB^{(\N)}$  is large,  the diameter of $J^r \tilde g ^\sm (U)$ approaches  $0$.
\item The families $(J^r_{p_0} \tilde X_{p}(\u \sb))_{\u \sb \in \sB^{\Z^-}}$  and $(J^r_{p_0}  X_{p}(\u \sb))_{\u \sb \in \sB^{\Z^-}}$ are uniformly close  when $(\tilde g^\sb)_{\sb}$ is  $C^r$-close to $(g^\sb)_\sb$.
\end{enumerate}\end{proposition}
\begin{proof}
First note that $(1)$ implies $(2)$. Let us show $(1)$. 
At a neighborhood of $\Lambda$, the map $ \tilde g^\sm_p$ is a large composition of contractions when $|\sm|$ is large. So $\partial_x \tilde g^\sm_p$ is small and even $C^{r-1}$-small. Also, when $|\sm|\to \infty$, the derivative  $\partial_p^k \tilde g^\sm_p$ is a sum of compositions  of an exponentially large  number of maps which are almost all contractions. So $\partial_p^k \tilde g^\sm_p$ is close to a constant $\ell_k$.
By the Faa-di-Bruno formula, there are polynomials $B_{j,k}$ such that:
\[\partial_p^k \tilde g^\sm_p(x_p)=\sum _{j=0}^k  
\partial_x^k\partial_p^j g^\sm_p
\cdot B_{j,k}\left(\partial_p x_p,\partial_p^2 x_p,\dots ,\partial_p^{k-j} x_p\right)= (\partial_p^k g^\sm_p)(x_p) +O(\|\partial_x \tilde g_p\|_{C^{k-1}}\cdot \|x_p\|_{C^{k}})
\; ,\]
Thus at a neighborhood $U$ of $J^r_{p_0}\Lambda$, when $|\sm|$ is large,
the map  $J^r_{p_0} g^\sm$ is close to the constant function $\sum_{ k=0}^r  \zeta_k \cdot (p-p_0)^k\mapsto  \sum_{k=0}^r \frac{\ell_k}{k!} \cdot (p-p_0)^k$. \end{proof}
%Furthermore this property holds true for any  $C^r$-perturbation of $(\tilde g^\sb)_\sb$ of $(g^\sb)_\sb$.  This implies:
%\begin{enumerate}[(2)]
%\item %For every $C^r$-perturbation of $(\tilde g^\sb)_\sb$ of $(g^\sb)_\sb$, 
%The jets 
%  $J^r_{p_0} \tilde X_{p}(\u \sb) = J^r_{p_0}\lim_{\infty}  \tilde g^{\sb_{-i} \cdots \sb_{-1}}(\Omega)$ are uniformly close to $J^r_{p_0}X_{p}(\u \sb)
% $ among $\u \sb\in \sB^{\Z^-}$. \end{enumerate}
% That is why we generalize the notion of $C^1$-blender to $J^r$-blender. 
% \begin{definition}[$J^r$-Blender IFS] 
%An invariant, transitive compact set  $J^r_{p_0}\Lambda\subset J^r_{p_0}M$ for the IFS   $(J^r_{p_0}g^{\sb})_{  \sb \in \sB} $ is $J^r_{p_0}$-contracting attractor if the projection $\Lambda_{p_0}= \Pi(J^r_{p_0}\Lambda)$ is a $C^1$-contracting attractor for $(g^\sb_{p_0})_{\sb\in \sB}$. It is a $J^r_{p_0}$-blender if moreover it has nonempty interior and   for any open subset of $K\Subset  J^r_{p_0} \Lambda$,   every $C^r$-perturbation of $(\tilde g^\sb)_{\sb\in \sB}$ has a $J^r_{p_0}$-contracting attractor which contains $K$. 
% \end{definition} 
%The following is a useful criterium. 
From this we obtain the following generalization of the covering property (compare to \cite{BCP16}):
\begin{proposition}\label{jr covering prop} The continuation $(\Lambda_p)_p$ satisfies the \emph{$J^r_{p_0}$-covering property} if the  interior of $J^r_{p_0} \Lambda $ is nonempty and there exists 
an increasing  sequence of subsets $K_n\Subset J^r_{p_0} \Lambda $ whose union is the interior of $J^r_{p_0} \Lambda $ and such that 
\[cl(K_n)\subset \bigcup_\sB\mathrm{int\, }  J^r_{p_0} g^\sb (K_n) \subset J^r_{p_0}  \Lambda\; .\]
Then for any $n\ge 0$, for any $C^r$-small perturbation  $(\tilde g^\sb)_{\sb \in \sB}$ of $(g^\sb)_{\sb \in \sB}$, the set $J^r_{p_0}\tilde \Lambda$ contains $K_n$. \end{proposition}

The following is the IFS counterpart of the $C^r$-parablender introduced in \cite{BE15}.
\begin{definition}
Given a $C^r$-family of IFS $(g^\sb)_{\sb \in \sB}$, the continuation   $(\Lambda_{p})_{p }$ of a blender $\Lambda_{p_0}$  is a parablender at $p_0$ if 
 $J^r _{p_0}\Lambda$ has nonempty interior and for every $K \Subset J^r_{p_0} \Lambda$, the  continuation
$(\tilde \Lambda_{p})_{p }$ of a sufficiently small  $C^r$-perturbation $(\tilde g^\sb)_{\sb\in \sB}$ of $(g^\sb)_{\sb\in \sB}$  satisfies  $J^r_{p_0}\tilde \Lambda\Supset K$. 
\end{definition}
Given   $i=(i_1, \dots, i_k)\in \N^k$  and $p=(p_1, \dots, p_k)\in \R^k$, we work with the multi-index notation $p^i:= p_1 ^{i_1}\cdots p_k^{i_k}$ and $|i|=i_1+\cdots + i_k$. For $r\in \N$, let $E_r:= \{i \in \N^k: |i| \le r\}$.
\begin{example}  \label{expparablendergroupe} 
Let $ \sB_r:=  \{-1,1\}^{E_r}$ and for every $\sb=(\sb_i)_{i\in E_r} \in  \sB_r$,  let:
\[  P_\sb (p) = \sum_{i\in E_r } \sb_i \cdot p^i\qand g^\sb_p:= x\in  \R  \mapsto   \frac 23( x- P_\sb(p)) + P_\sb(p)\; .\]
Then the continuation of  $\Lambda_0:=[-1,1]$ is a $C^r$-parablender  at $p=0$  for    $((g^\sb_p)_{p\in B_k} )_{\sb \in  \sB_r}$  with: 
\[J^r_0\Lambda:=  \left \{ \sum_{i\in E_r } \xi_i \cdot p ^{i}: (\xi_i)_i\in [-1,1]^{E_r}\right\}\]
  \end{example} 
\begin{proof} First note that the set of diffeomorphisms
$\{g^\sb_0: \sb \in  \sB_r\}$ equals the one   \cref{blender IFS} for which $\Lambda_0:=[-1,1]$ is a blender. 
Via the conjugacy $(x_j)_{j\in E_r} \mapsto \sum_{E^r} x_j p^j$,  the set of maps $(J^r_0 g^\sb)_{\sb\in \sB_r}$  is conjugated   to the Cartesian product of $\Card E_r$ times the set of maps of \cref{blender IFS}:
\[\{J^r_0 g^\sb :{\sb\in  \sB_r}\}\approx \{ (x_j)_{j\in E_r}\mapsto (g^{\sb_j}(x_j))_{
j\in E_r}
: (\sb_i)_i \in \sB^{E_r}\}\quad \text{ using }  \sB_r \approx  \sB^{E_r}\; .\]
Thus $J^r_0\Lambda\approx [-1,1]^{E_r}$ is a product of blenders IFS satisfying the covering property and so the continuation $(\Lambda_p)_p$  satisfies the assumptions of \cref{jr covering prop}.
%
% and so is a blender. Let $K\Subset J^r_0\Lambda$.
%For a $C^r$-perturbation $(\tilde g_p^\sb)_{\sb\in  \sB_r}$ of $(\tilde g_p^\sb)_{\sb\in  \sB_r}$, the family   $(J^r_0 \tilde g_p^\sb)_{\sb\in  \sB_r}$ is contracting and $C^0$-close to $(\tilde g_p^\sb)_{\sb\in  \sB_r}$. Hence the open covering property implies that it has a contracting attractor contianing $K$. 
 \end{proof}
% \begin{remark}\color{red} 
%Si la famille est de class $ C^{r,r+1}$...
% \end{remark}

\subsection{Density of rotations from blender IFS}\label{Density of rotations from blender IFS}
Let us explain how to use a blender to obtain the local density of semi-groups of finitely generated circle diffeormorphisms containing a rotation.
 We recall that  the projectivized $P^1(\R)$ of $\R^2$ is a smooth circle which is canonically  identified to the one-point compactification $\R\cup \{\infty\}$ of $\R$ via the inclusion $x\in \R\hookrightarrow [x; 1] \in P^1(\R) $. This enables to extend analytically the dynamics of any non-zero real polynomial or even rational map of $\R$ to this circle. These are the \emph{projectivized} corresponding maps. 
 
\begin{proposition}\label{proprot}
There exists a neighborhood $N$ in  $\Diff^1(P^1(\R))$ of a triplet of homographies   $(g^\sa, g^\sb, g^\sc) $, such that  
any $(\tilde g^\sa, \tilde g^\sb, \tilde g^\sc)\in N\cap \Diff^\infty(P^1(\R)) $ can be perturbed by conjugacies with homographies to span a  semi-group  containing a    smooth,   irrational rotation.\end{proposition}
\begin{proof}
Let $(g^\sa,g^\sb)_{\sa, \sb\in \sB_0}$ be the projectivized action of the affine maps of    \Cref{blender IFS} and let $V$ be a $C^1$-neighborhood of this pair of diffeomorphisms such that every $(\tilde g^\sa, \tilde g^\sb)\in V$ has a blender  $\tilde \Lambda $ containing $[-2/3,2/3]$.  Let $g^\sc: x\in P^1(\R)=\frac32 \frac x{x+1}$ and  
let $V'$ be a small neighborhood of $g^\sc$ formed by maps $\tilde g^\sc$  which  displays exactly two periodic points: an attracting  $x_a\approx 1/2$  and a repulsive     $x_r\approx 0$. 

Let  $(\tilde g^\sb)_{\sb\in \sB_0}  \in V$ and $\tilde g^\sc\in V'$ be of class $C^\infty$.   Using the transitivity of the blender there exists $\sm\in \sB_0^{(\N)}$ such that $\tilde g^\sm$ sends a point  close to $x_a$  to a point close to $x_r$. Hence up to a small perturbation of $\tilde g^c$ (using a composition with a homography), we can assume that $\tilde g^\sm(x_a)=x_r$.
  Now we consider a small unfolding $(\tilde g^\sc_\epsilon)_{\epsilon\in [-\epsilon_0, \epsilon_0]}$ of $\tilde g^\sc=:\tilde g^\sc_0$, such that with  $x_a(\epsilon)$ and $x_r(\epsilon)$ the continuations of the fixed points  $x_a=:x_a(0)$ and $x_r=:x_r(0)$, it holds:
  \[x_a(-\epsilon_0)<x_a<x_a(\epsilon_0)\qand x_r(\epsilon_0)<x_r<x_r(-\epsilon_0).
\]
This can be done by conjugating the map $\tilde g^\sc$ by  $(x\mapsto (1+\epsilon)(x-1/4)+1/4)_\epsilon$. Thus $x_a$ is sent to $x_r$ by $\tilde g^\sm$. Then for $k$ large,  the  points   $((\tilde g_{\epsilon_0}^\sc)^j(x_r) )_{k\ge j\ge 0} $ go clockwise to land nearby $x_a(\epsilon_0)>x_a$:
\[\tilde g^\sm(x_a)=x_r< \tilde g_{\epsilon_0}^\sc(x_r)< \cdots < (\tilde g_{\epsilon_0}^\sc)^j(x_r) < \cdots  <  (\tilde g_{\epsilon_0}^\sc)^k(x_r) > x_a\]
 On the other hand,  the  points   $((\tilde g_{-\epsilon_0}^\sc)^j(x_r) )_{k\ge j\ge 0} $ go anti-clockwise to land nearby $x_a(-\epsilon_0)<x_a$:
 \[\tilde g^\sm(x_a)=x_r> \tilde g_{-\epsilon_0}^\sc(x_r)> \cdots > (\tilde g_{-\epsilon_0}^\sc)^j(x_r)>  \cdots >  (\tilde g_{-\epsilon_0}^\sc)^k(x_r) < x_a\; .\]
  Thus when $\epsilon $ varies between $-\epsilon_0$  and $\epsilon_0$, the point $(\tilde g^\sc_\epsilon)^k\circ \tilde g^\sm(x_a)$  makes at least one whole turn around $P^1(\R)$.  This implies that the rotation number of $(\tilde g_\epsilon^\sc)^k\circ \tilde g^\sm$ is not constant when $\epsilon $ varies in $[-\epsilon_0,\epsilon_0]$. 
Thus  by continuity of the rotation number and density of the Diophantine ones, there exists $\epsilon\in (-\epsilon_0 , \epsilon_0 )$ such that the rotation number of $ \tilde  g_{\epsilon}^{\sc^n}\circ \tilde g^\sm$ is Diophantine. By Arnold-Herman-Yoccoz \cref{AHthm}, the map $\tilde  g_{\epsilon}^{\sc^n}\circ \tilde g^\sm$ is smoothly conjugate to the rotation  $R_\alpha$. 
\end{proof}
This proposition implies the first part of Corollary \ref{coroA}. However, the technique is not easy to generalize to obtain a parametric version of this result. 
%we do not know if a parabolic map  appears in the unfolding $(<\tilde  g_{\epsilon}, (\tilde g_\sb)_{\sb\in \sB_0}>)_\epsilon$  of the group  $<\tilde  g, (\tilde g_\sb)_{\sb\in \sB_0}>$  given by  the latter proof. More problematic: if
We know that the set $D$ of  Diophantine smooth rotations is an union of    one-codimensional submanifolds and these appear locally  $C^0$-densely from the above argument.  However this does not provide any 
control on the  tangent spaces of these manifolds, while it is crucial for Theorem \ref{theoB}. To obtain such a control,  we will focus on parabolic maps and introduce  the   $\lambda$-blender to show their density.
\subsection{$\lambda$-blender and $\lambda$-parablender
IFS}\label{sec: lambda-blender IFS}
Let $M$ be $\R$ or $P^1(\R)\approx \R/\Z$; its tangent bundle is  identified with $M\times \R$.  

\begin{definition} 
The  \emph{Lyapunov fibration} of a contracting attractor $\Lambda$ for an IFS $(g^\sb)_{\sb\in \sB}$ of $M$ is:
\[\lambda(\Lambda):= 
\bigcap_{n\ge 0} cl \left\{ \Big(  g^{\sm}(x), \frac1{|\sm|}\log |D_x   g^\sm|\Big) : x\in \Lambda, \; \sm\in \sB^{(\N)} \text{ and } |\sm|\ge n\right\} \subset \Lambda \times \R\; .
\]\end{definition} 
 To show the local density of finitely generated group of circle diffeomorphisms with a parabolic element, we introduce the following new notion:
\begin{definition}[$\lambda$-Blender  IFS] 
A blender $\Lambda$ for an IFS   $(g^{\sb})_{  \sb \in \sB}\in \End^1( M )^\sB$ is  a   $\lambda$-\emph{blender}  if its Lyapunov fibration $\lambda(\Lambda)$ has nonempty interior  and for any subset $ K \Subset\lambda(\Lambda) $,   for every $C^1$-small perturbation of $(\tilde g^\sb)_{\sb\in \sB}$,  the Lyapunov fibration of the continuation of $\Lambda$ contains  $ K $.
% 
%following limit set contains $\lambda(K)$:
%\[
%\bigcap_{n\ge 0} cl \left\{ \Big(\tilde g^{\sm}(x), \frac1{|\sm|}\log |D_x \tilde g^\sm|\Big) : (x,u)\in  \lambda(K) , \; \sm\in \sB^{(\N)} \text{ and } |\sm|\ge n\right\} \; .
%\]
\end{definition}

\begin{example} \label{exem lambda blender IFS}
Let $\sB^\lambda_0:= \{-1,+1\}^2 $.  Then for  $\epsilon>0$ small,  $[- 1,1]$ is a $\lambda$-blender for the  IFS $(g^\sb)_{\sb \in \sB^\lambda_0}$  where  $g^{\sb} : x\in \R  \mapsto  
 \frac23\cdot \exp (\epsilon\cdot \delta')\cdot (x+\delta)-\delta  $ for $  \sb=( \delta, \delta' )\in \sB^\lambda_0$ and  with Lyapunov fibration:
\[\lambda(\Lambda):= [-1,1]\times \left[\log \frac 23-\epsilon ,\log \frac 23+\epsilon \right]\subset \R \times \R\, .\]
 \end{example} 
  \begin{proof}  Let $\eta>0$ be small and let $K:=[-1+\eta, 1-\eta]\times \left[ \log \frac 23-\epsilon+\eta ,\log \frac 23+\epsilon -\eta\right]$.  
  
  Let $(\tilde g^\sb)_{\sb\in \sB^\lambda_0}$  be a  $C^1$-perturbation of $( g^\sb)_{\sb\in \sB^\lambda_0}$. It suffices to show for every $(x_0, \ell )\in  K $ and $n\ge 0$, the existence of a sequence of letters $(\sb_i)_{-n\le i< 0}\in (\sB^\lambda_0)^{n}$ and a sequence of points $(x_i)_{-n\le i<0}\in [-1+\eta ,1-\eta]^{n}$ such that:
  \begin{enumerate} 
  \item $x_0=x$ and  $\tilde g^{\sb_i}(x_i)= x_{i+1}$ for every $-n\le i<0$,
  \item $-2\epsilon\le \log  |D_{x_{-n}}g^{\sb_{-n} \cdots \sb_{-1}}| -n\cdot \ell\le  2\epsilon$ .
  \end{enumerate}
The step $n=0$ is trivial. Assume the result shown for $n\ge 0$. 
Then, 
with $\delta$ the sign of $x_{-n}$ and $\delta'$ the sign of $-(\log  |D_{x_{-n}}\tilde g^{\sb_{-n}\cdots \sb_{-1}}| -n\ell)$, with $\sb= (\delta,\delta')\in \{-1,1\}^2=\sB^\lambda_0$ and with $x_{-n-1}$ the preimage by $\tilde g^{\sb_{-n-1}}$ of $x_{-n}$, the induction hypothesis holds true by the covering property showed  in \cref{blender IFS}. Indeed, if $\delta'=+1$  then 
$\log |D_{x_{-n}}\tilde g^{\sb_{-n} \cdots \sb_{-1}}| -n \ell \le  0$ and 
 $ \log |D_{x_{-n-1}}\tilde g^{\sb_{-n-1}}| - \ell $ is approximately in $\log\frac23 +\epsilon- [+\log\frac23 -\epsilon+\eta , +\log\frac23 +\epsilon -\eta]= [ \eta ,  2\epsilon -\eta] $. Thus:
\[0\ge  \log |D_{x_{-n}}\tilde g^{\sb_{-n} \cdots \sb_0}| -n \ell >-2\epsilon\qand 
2\epsilon> \log |D_{x_{-n-1}}\tilde g^{\sb_{-n-1}}| - \ell > 0\]
Summing these two bounds, we obtain $(2)$.  The case $\delta'=-1$ is done similarly.
  \end{proof}
  
%  \subsection{$\lambda$-$C^r$-parablender IFS}
\label{sec: lambda-parablender IFS}   Let us now give the parametric counterpart of the $\lambda$-blender.  Let  $M$ be a curve; its tangent bundle is  identified with $M\times \R$.
  Let $\sB$ be a finite set and for every $\sb\in \sB$, let  $g^\sb =(g^\sb_p)_{p\in B_k}\in \widehat {\End_k^r}(M)$. 
 Let $p_0\in B_k$.  We recall that if $\Lambda_{p_0}$ is a contracting attractor for  the IFS  $(g_{p_0}^\sb)_{\sb\in \sB}$, then for every  point $\Omega$ nearby $\Lambda_{p_0}$ and $p$ nearby $p_0$, the following set:
 \[  \Lambda_p:= \{\lim_{\infty}  g_{p}^{\sb_{-i} \cdots \sb_{-1}}(\Omega): \u \sb  \in \sB^{\Z^-}\}= \bigcap_{n\ge 0}  cl( \{g_{p}^{\sm}(\Omega): \sm\in \sB^{(\N)} \text{ and } |\sm|\ge n \})\;  \]
is the continuation of $\Lambda_{p_0}$.  It does not depend on $\Omega$. % Assume that the contracting attractor $ \Lambda_p$ is well defined for every $p\in B_k$.
 Its Lyapunov fiber is:
 \[ \lambda( \Lambda_p)=  \bigcap_{n\ge 0}   cl \left\{\Big(g_{p}^{\sm}(\Omega), \frac1{|\sm|} \log |D_\Omega g_{p}^{\sm}|\Big): 
 \sm\in \sB^{(\N)} \text{ and } |\sm|\ge n\right\}\; .  \]

While studying the $C^r$-parablender we considered:
 \[   J^r_{p_0}\Lambda:= \{J^r_{p_0} \lim_{\infty}  g_{p}^{\sb_{-i} \cdots \sb_{-1}}(\Omega): \u \sb  \in \sB^{\Z^-}\}= \bigcap_{n\ge 0}   cl \left\{J^r_{p_0}  g_{p}^{\sm}(\Omega): 
 \sm\in \sB^{(\N)} \text{ and } |\sm|\ge n\right\}\;  \]
 It is thus natural to consider the following para-counterpart the Lyapunov fiber:
  \[   \lambda(J^r_{p_0}\Lambda):=  \bigcap_{n\ge 0}  cl   \left\{\Big(J^r_{p_0}g_{p}^{\sm}(\Omega), J^{r-1}_{p_0} \frac1{|\sm|} \log |D_\Omega g_{p}^{\sm}|\Big): 
 \sm\in \sB^{(\N)} \text{ and } |\sm|\ge n\right\}\; .\]
We notice that $\lambda(J^r_{p_0}\Lambda)$ is a compact subset of $J^r_{p_0} M\times J^{r-1}_{p_0} \R$. 

The following is a natural generalization of both the $C^r$-parablender and the $\lambda$-blender. This new notion can be used to prove the local density of finitely generated groups of families of circle diffeomorphisms with a parabolic element at every parameter:
\begin{definition}[$C^r$-$\lambda$-parablender  IFS] The family  $(\Lambda_p)_p$ of contracting attractors for $(g^{\sb})_{  \sb \in \sB} $ is  a $C^r$-$\lambda$-\emph{parablender}  at $p_0$  if $ \lambda(J^r_{p_0}\Lambda)$   has nonempty interior  and for any   subset $ K \Subset \lambda(J^r_{p_0}\Lambda)$, 
the  continuation $(\tilde \Lambda_{p})_{p }$ of every  $C^r$-perturbation $(\tilde g^\sb)_{\sb\in \sB}$ of $(g^\sb)_{\sb\in \sB}$  satisfies  $\lambda(J^r_{p_0}\tilde \Lambda)\Supset K $. 
\end{definition}
\begin{remark}
We notice that  $\Lambda_{p_0}$ is a $\lambda$-blender for $(g^\sb_{p_0})_\sb$ and 
$(\Lambda_p)_p$ is a $C^r$-parablender at $p=p_0$ for  $(g^\sb)_\sb$. 
\end{remark}
\begin{example} \label{lambda parablender IFS}
Let  $  \sB_r$ and $(P_\sb)_{\sb \in \sB_r}$  be defined as in \cref{expparablendergroupe} and put $\sB_r^\lambda := \sB_r\times \sB_{r-1}$.   Then for $\epsilon>0$ small, the continuation of  $\Lambda:=[- 1,1]$ is a $C^r$-$\lambda$-parablender at $p=0$ for  $(g^\sb)_{\sb \in \sB_r^\lambda}$  where $g^\sb=(g^\sb _p)_{p\in B_k}$ is defined by:  
\[g^{\sb}_p : x\in P^1(\R) \mapsto    \frac23\cdot \exp(\epsilon\cdot P_{\sb'} (p))\cdot (x-P_\sb(p))+P_\sb(p)\text{ for }   (\sb, \sb')\in  \sB_r\times  \sB_{r-1}=\sB_r^\lambda.\]
Moreover it holds:
\[\lambda(  J^r _0  \Lambda) = 
 \left \{ \Big(\sum_{i\in E_r } \xi_i \cdot p ^i,\log \frac23 +\epsilon \sum_{i\in E_{r-1} } \lambda_i \cdot p^i \Big)
 : (\xi_i)_i\in [-1,1]^{E_r}, (\lambda_i)_i\in [-1,1]^{E_{r-1}}\right\}\]
 \end{example} 
  \begin{proof}Let $\eta>0$ be small and put $I:= [-1+\eta,1-\eta]\Subset [-1,1]$ and:
\[  K:=
 \left \{ \Big(\sum_{i\in E_r } \xi_i \cdot p ^i,\log \frac23 +\epsilon \sum_{i\in E_{r-1} } \lambda_i \cdot p^i \Big)
 : (\xi_i)_i\in I^{E_r}, (\lambda_i)_i\in I^{E_{r-1}}\right\}\; .
\]% \left \{  \sum_{i \in E_r } x_i \cdot p ^i  :  x_i  \in [-1+\eta ,1-\eta] \right \}\qand  \lambda(K):=K\times \left \{\log\frac 23+\epsilon\cdot \ell: \ell \in  K \right \}\]
 
%  \{  (\sum_{i} x_i \cdot p ^i,\log \frac23 +\epsilon \sum_{i  } \lambda_i \cdot p^i  ) :  x_i ,\lambda_i \in [-1+\eta ,1-\eta]  \}$.
  Let $(\tilde g^\sb)_{\sb\in \sB^\lambda_0}$  be a  $C^r$-perturbation of $( g^\sb)_{\sb\in \sB^\lambda_0}$. It suffices to show that for every $ (x_0,\ell)  \in  K $ and $n\ge 0$, the existence of a sequence of letters $(\sb_j)_{-n\le j< 0}\in (\sB^\lambda_0)^{n }$ and a sequence of points $(  x_j, \ell_j)_{-n\le j< 0}\in  K^{n} $ such that:
  \begin{enumerate} 
  \item    $J^r_{p_0} \tilde g^{\sb_j}( x_j)=  x_{j+1}$ for every $-n\le j<0$,
 \item 
$\ell _n =J^{r-1}_{p_0}( \log  |D_{x_{-n}(p)}\tilde g^{\sb_{-n} \cdots \sb_{-1}}_p| )_p\in \left \{  n \ell + \sum_{i \in E_{r-1}  } \lambda_i \cdot p ^i 
 :  \lambda_i  \in [-2\epsilon, 2\epsilon]  \right \}$.
  \end{enumerate}
The step $n=0$ is obvious. Assume the result shown for $n\ge 0$. 
Let $\sb_{-n-1}\in  \sB^\lambda_r $ be the upplet whose coefficients are the sign of those of the polynomial $(x_{-n}, -\log  |D_{x_{-n}}g^{\sb_{-n}\cdots \sb_{-1}}| +n\cdot \ell )$. Then for the same reasons as for  
\cref{expparablendergroupe,exem lambda blender IFS}
 the induction hypothesis holds true at step $n+1$. 
  \end{proof}
\subsection{Density of parabolic maps from $\lambda$-blender IFS}\label{Density of rotations from lambda-blender IFS}
We now introduce a technique based on  a dynamical rescaling. 
It enables  to show  that a group contains a   parabolic map nearby  a prescribed one, say $h: x\in P^1(\R)\mapsto  x/(2x+1)\in P^1(\R)$, and this for a dense set of perturbations of its generator.  This method involves  the $\lambda$-parablender of  \cref{exem lambda blender IFS} extended to the projective space $P^1(\R)$ and the map  $g^{\sc}$ which is conjugate to $\Delta: x\mapsto \frac32 x$ via the parabolic map $h$:%= x/(2x+1)$:
\begin{equation}\label{def gc} g^{\sc}=h\circ \Delta\circ h^{-1}: x\in P^1(\R)\mapsto  \frac32 \cdot\frac{x}{x+1}\in P^1(\R)\; .\end{equation}
The map $g^\sc$ has an expanding fixed point at $x_r=0$ with eigenvalue  $  3/2$.  
Given a perturbation $\tilde g^\sc$ of $g^\sc$ we denote $\tilde x_r$  the continuation of $x_r$. Let $\tilde g^{\sc^n}$  be the composition of $n$-times $\tilde g^\sc$. 
The proof of the following contains the  aforementioned key technique. It  will be generalized   to prove the main theorems,  and should be fruitful to generalize to an even broader setting.

\begin{proposition}\label{Importantclaimgroup}For every $r\ge 2$, for every $n\ge 1$ large and $C^r$-perturbation $(\tilde g^\sc, (\tilde g^\sb)_{\sb \in \sB^\lambda_0})$ of $(g^\sc, (g^\sb)_{\sb \in \sB^\lambda_0})$,   for every   $\sb_n\in (\sB^\lambda_0)^n$
if:
% $ \breve h:= \tilde g^{\sc^n}\circ \tilde g^{\sb_n} $ has a parabolic fixed point in $[-2,2]$, then $ \breve h$ is parabolic and its  restriction to   $[-3/2, 3/2]$  is   $\epsilon$-$C^r$-close to $h|[-3/2, 3/2]$.  
\[\tilde g^{\sb_n}(\tilde x_r)=\tilde x_r=\tilde g^\sc(\tilde x_r)\qand  \frac1n \log |D_{\tilde x_r}\tilde g^{\sb_n}|=- \log |D_{\tilde x_r}\tilde g^\sc|,\]
then  the map $ \breve h:= \tilde g^{\sc^n}\circ \tilde g^{\sb_n} $ is parabolic 
 and its restriction to  $[-3/2, 3/2]$  is   $C^r$-close to $h|[-3/2, 3/2]$. 
\end{proposition}
As $(0, \log3/2) $ belongs to the Lyapunov fibration of the $\lambda$-blender defined by $ (  g^\sb)_{\sB^\lambda_0}$, for $(\tilde g^\sc, (\tilde g^\sb)_{\sB^\lambda_0})$ in a $C^r$-neighborhood of  $(  g^\sc, (  g^\sb)_{\sB^\lambda_0})$, when $n$ is large, there is $\sb_n\in (\sB_0^\lambda)^{n}$ such that $\tilde g^{\sb_n}$  displays a fixed point close to $\tilde x_r$ with Lyapunov exponent close to  $\log |D_{\tilde x_r}\tilde g^\sc|$.  Thus, there is a perturbation of $\tilde g^\sc$ of the form $(1+\epsilon')\cdot \tilde g^\sc+\epsilon$  
satisfying  the assumptions of the above proposition. This shows:
\begin{corollary}\label{coroImportantgroup}
For every $C^2$-perturbation  $(\tilde g^\sc, (\tilde g^\sb)_{\sB^\lambda_0})$ of $(g^\sc, (g^\sb)_{\sB^\lambda_0})$, there exist $\epsilon, \epsilon'\in \R$ arbitrarily small such that  
the  semi-group  spanned by  $(\exp(\epsilon')\cdot \tilde g^\sc+\epsilon, (\tilde g^\sb)_{\sB^\lambda_0})$  contains  a parabolic diffeomorphism.
%\\
%
%There is a $C^2$-neighborhood $\cal N$ of $(g^\sc, (g^\sb)_{\sB^\lambda_0})$ formed by upplets which can be perturbed  via composition with affine maps to span  a semi-group containing  a parabolic diffeomorphism.
\end{corollary}

\begin{proof}[Proof of \cref{Importantclaimgroup}]
First note that if $\breve h\in \Diff^2(P^1(\R))$ has its restriction to 
$[-3/2,3/2]$ which is $C^2$-close to $h$ and has a parabolic point  nearby $0$, then $\breve h$ cannot have another fixed point in $[-3/2,3/2]$, indeed  the parabolic fixed point of $h(x)=x/(2x+1)$ is non-degenerate. It cannot have another fixed point in $P^1(\R)\setminus [-3/2,3/2]$ since $h$ sends $P^1(\R)\setminus [-3/2,3/2]$ onto $ [3/8,3/4]\Subset  (-3/2,3/2)$, and so the same holds true for $\breve h$. As $\breve h$ has by assumption a parabolic point nearby zero,  to show the proposition, it suffices to prove:
\begin{claim}\label{Importantclaimgroup1}Under the assumptions of \cref{Importantclaimgroup}, the restriction of the map $ \breve h:= \tilde g^{\sc^n}\circ \tilde g^{\sb_n} $ to  $[-3/2, 3/2]$  is   $C^r$-close to $h|[-3/2, 3/2]$. 
\end{claim}
Up to a coordinate change close to the identity, we can assume that  the fixed point $\tilde x_r$ of $\tilde g^\sc$ is equal to $0$. Then in these coordinates, we shall prove that $ \breve h|[-5/3, 5/3]$ is $C^r$-close to the identity. 
 \begin{theorem}[Sternberg \cite{S57}]\label{Sternerberg} There exists a map $\tilde h$ which is $C^r$-close to $h$, for which $0$ is a parabolic fixed point and  such that $\tilde h|[-2,2]=  \tilde g^{\sc^n}  \circ \tilde h\circ (D_0\tilde g^{\sc^n}  )^{-1} |[-2,2]$.
\end{theorem}
\begin{proof}  
Let $B$ be the subset of $ C^r([-2,2], \R)$ formed by maps $\bar h$ such that $\bar h(0)=0$ and $D_0 \bar h=1$, is complete for the distance $d(h, \tilde h)= \max_{2\le k\le r} \|D^k h-D^k\tilde h\|_\infty$.
The following operator:
\[  \bar h\in B \mapsto \tilde g^\sc  \circ \bar h\circ (D_0\tilde g^{\sc }  )^{-1} \; 
\]
 depends continuously on $\tilde g^{\sc}$ and has a contracting iterate. 
  Its fixed point $\tilde h\in C^r([-2,2], \R)$  depends continuously on $\tilde g^{\sc}$ and satisfies  $\tilde h|[-2,2] = \tilde g \circ \tilde  h\circ (D_0\tilde g^{\sc }  )^{-1}|[-2,2] $.  As when $\tilde g^{ \sc}=g^\sc$ the fixed point is $h$ by \cref{def gc},   when $\tilde g^{ \sc}$ is close to $g^\sc$  the fixed point $\tilde h\in B$ is $C^r$-close to $h$. 
\end{proof}
%This theorem states that:
%\begin{equation} \label{iteration conjugaisaon}\tilde h(x)= \tilde g^{ \sc ^n}  \circ \tilde h\circ (D_0\tilde g^{\sc^n }  )^{-1}  (x)\quad \forall x\in [-2,2], \quad  \tilde h( 0)= 0\qand D_{0} \tilde h=1
% \; .\end{equation}
Then the dynamical rescaling techniques consists of noting that $\breve h= \tilde h\circ  \Phi$, where $\tilde h$ was bounded by the above Theorem and  $\Phi=D_0\tilde g^{\sc^n }   \circ   \tilde h^{-1}\circ \tilde g^{ \sb_{n} }  |[-2,2]$ is bounded by the following:
%
%replacing $(D_0\tilde g^{\sc^n })^{-1}$ by  $\tilde h^{-1}\circ \tilde g^{ \sb_{n} }$ in the above equation, which turns out to be uniformly close to be linear, even when $n$ is large:
\begin{lemma}\label{lemm2finalClaimgroup}
The map $ \Phi=D_0\tilde g^{\sc^n }   \circ   \tilde h^{-1}\circ \tilde g^{ \sb_{n} }  |[-2,2]$ is $C^r$-close to the identity. \end{lemma}
This lemma implies that $\Phi([-5/3,5/3])$ is included in $[-2 ,2 ]$, and so 
 the  composition $ \tilde h\circ  \Phi$ is well defined on $[-5/3,5/3]$, equal to $\breve h$, and  is  $C^r$-close to $\tilde h$ and so $h$. This proves the claim. \end{proof}
\begin{proof}[Proof of  \cref{lemm2finalClaimgroup}]
By assumption of \cref{Importantclaimgroup}, $\tilde x_r=0$   is a neutral fixed point of  $\Phi$ which is equal to $(D_0\tilde g^{\sb_n }   )^{-1}\circ   \tilde h^{-1}\circ \tilde g^{ \sb_{n} }  |[-2,2]$.  Thus  it suffices to show that $D^2\Phi$ is $C^{r-2}$-small: 
\[D^2\Phi = (D_0 \tilde g^{\sb_n})^{-1}  \cdot D^2\tilde h^{-1}\cdot (D   \tilde g^{ \sb_{n} })^2+
\sum_{i=1}^n D_0\tilde g^{\sb_n }\cdot D\tilde h^{-1}\cdot D\tilde g^{ \sb_{n}^{n, i+1} }\cdot  D^2   \tilde g^{ \sb_n ^i } \cdot ( 
D   \tilde g^{ \sb_{n}^{i-1, 1}} )^2,\]
where $\tilde g^{ \sb_{n}^{n, i+1} }:= \tilde g^{ \sb_{n}^{n}}\circ \cdots \circ  \tilde g^{ \sb_{n}^{i+1} }$ and $\tilde g^{ \sb_{n}^{i-1,1} }:= \tilde g^{ \sb_{n}^{i-1}}\circ \cdots \circ  \tilde g^{ \sb_{n}^{1} }$ with $\sb_n:= \sb_n^n \cdots \sb_n^i \cdots \sb_n^1$.

The first term $(D_0\tilde g^{\sb_n })^{-1}\cdot D^2\tilde h^{-1}\cdot (D  \tilde g^{ \sb_{n}})^2$ is of the order of $|D  \tilde g^{ \sb_{n}}|$ which is  small. 
In the sum is small since the derivatives $D^2  \tilde g^{\sb_n^i }$ are all small, whereas $|(D_0  \tilde g^{ \sb_{n}})^{-1}\cdot D\tilde h^{-1}\cdot D\tilde g^{ \sb_{n}^{n, i+1} }|\cdot |  ( 
D   \tilde g^{ \sb_{n}^{i-1, 1}} )^2|$ is of the order of $| 
D   \tilde g^{ \sb_{n}^{i-1, 1}}  |$ which is exponentially $C^{r-1}$-small when  $n-i$ is large.  
\end{proof}

\section{Intrinsic definition of $(\lambda)$-(para)-blenders}
In this section we are going to embed the  $\lambda$ and/or parablender IFS into normally hyperbolic fibrations for a single differentiable dynamics of a manifold. The attractors of these IFS will persist as hyperbolic basic sets. We are going to introduce a formalism to study their intrinsic properties.  These  will be used in the next section via a generalization of \cref{coroImportantgroup}  together with   \cref{Densiteparabolic,ThmS3} to prove   Theorems~\ref{theoA}~and~\ref{theoB}.
 \subsection{Embedding a semi-group into a normally hyperbolic fibration}\label{Embedding a semi-group into a normally hyperbolic fibration}
 Let $1\le r\le  \infty$. Let $\sA$ be a finite alphabet and let $\Sigma=\sA^\I$ with $\I= \N$ or $\Z$.   This is a compact space endowed with the product topology. Its shift dynamics is denoted by  $\sigma:\ss=(\sa_i)_i\in \Sigma\mapsto  (\sa_{i+1})_i\in \Sigma$. Let $N$ be a compact manifold. 
The aim of this subsection is to recall that the orbit of a finitely generated  semi-group can be `persistently' embedded into the following class of maps.
  \begin{definition}   A  \emph{$C^r$-endomorphism of $\Sigma\times N$ over $\sigma$} is a self-map of 
  $\Sigma\times N$ of the form 
  \[g: (\ss , x)\in \Sigma\times N \mapsto (\sigma(\ss), g^{\ss}(x)),\]
  where  $( g^{\ss})_{\ss\in \Sigma}$ is a continuous family of $C^r$-maps of $N$. The space  of   $C^r$-endomorphisms over $\sigma$ is denoted by $\End^r_\sigma(\Sigma\times N)$.
 Two $C^r$-endomorphisms $g, \tilde g \in \End^r_\sigma(\Sigma\times N)$ are close if their induced families $(g^\ss)_{\ss\in \Sigma}$ and $(\tilde g^\ss)_{\ss\in \Sigma}$ are uniformly close for the $C^r$-topology.  
   \end{definition}
 \begin{example}[Canonical map of $\End^r_\sigma(N)$ 
 associated to $(g^\sa)_{\sa\in \sA}$]
 \label{canonical map}
A family  $(g^\sa)_{\sa\in \sA}$ of $C^r$-maps of $N$ is canonically associated to the following endomorphism of $\Sigma\times N$:
 \[g= (\ss, x)\mapsto (\sigma(\ss), g^\ss(x)),\quad 
\text{where }g^\ss=g^\sa\text{ if }\sa\in \sA\text{ is the letter of }\ss\text{ at the }0\text{-position.}\] 
 \end{example}
For every $\sm\in \sA^{(\N)}$, the element $g^\sm$ of the semi-group spanned by $(g^\sa)_{\sa\in \sA}$ is equal to the restriction of  $g^{|\sm|}$ to the $|\sm|$-periodic fiber $\{\sm^\infty\}\times N$, 
with $\sm^\infty\in \Sigma$ the $|\sm|$-periodic point  whose $\sA$-letters at the $0,1, ..., |\sm|-1 $ positions are those of $\sm$. Hence any element of the semi-group is equal to the dynamics of a periodic fiber.  Now we would like to embed this dynamics into the one of a manifold $M$. This requires the following notion:
 \begin{definition}   A \emph{$C^r$-embedding of $\Sigma\times N$ into $M$} is a map of the form 
  \[j: (\ss , x)\in \Sigma\times N \mapsto j^{\ss}(x)\in M,\]
  where  $( j^{\ss})_{\ss\in \Sigma}$ is a continuous family of $C^r$-embedding of $N$ into $M$ with disjoint images.
 Two $C^r$-embeddings are $C^r$-close if their families of $C^r$-maps are uniformly $C^r$-close. The space  of   $C^r$-embeddings of $\Sigma\times N$ into $M$  is denoted by $\Emb^r(\Sigma\times N , M)$.
   \end{definition}

Let $f$ be a $C^r$-self-map of $M$ which leaves invariant a fibration $\cal L=\bigcup_{\ss\in \Sigma} L_{\ss}$ where   $L_{\ss}:= j^\ss(N)$ is defined by  $j\in \Emb^r(\Sigma\times N, M)$. We assume that the dynamics between the fibers is given by the shift $\sigma$:  for every $\ss\in  \Sigma$, the fiber $L_\ss$ is sent by $f$ into the fiber $L_{\sigma(\ss)}$.  
%
% Let $j\in \Emb^r(\Sigma\times N, M)$ and let $f\in \End^r(M)$  which leaves invariant the fibration $j(\Sigma\times M)$ and which sends each fiber $j^\ss(N)$ into $j^{\sigma(\ss)}(N)$. 
 Then   $g:=j^{-1}\circ f\circ j$   is a $C^r$-endomorphisms of  $\Sigma\times N$ over $\sigma$.  Let us  restate  \cref{HPS} in this terminology:
\begin{theorem}\label{HPS2}Assume $f$ is $r$-normally hyperbolic at $\cal L$.  Moreover, if $f$ is not a diffeomorphism, we assume that $f$ is $r$-normally expanding at $\mathcal L$.
Then for every $\tilde f$ $C^r$-close to $f$ there exists an embedding $\tilde j\in \Emb^r(\Sigma\times N, M) $  close to $j$ and $\tilde g\in  \End^r_\sigma(\Sigma\times N) $ close to $g$  such that:
\[\tilde f\circ \tilde j=  \tilde j\circ \tilde g \; .\]
%Let $f\in \End^r(M)$ be $r$-normally hyperbolic at the $C^r$-embedding  $j$ of $\Sigma\times N$ into $M$. Assume that  $j^{-1}\circ f\circ j\in \End_\sigma^r(\Sigma\times N)$. Moreover if $f$ is not a diffeomorphism, assume that $f$ is $r$-normally 
% expanding at $j(\Sigma\times N)$.
%Then for every $\tilde f$ $C^r$-close to $f$ there exists an embedding $\tilde j\in \Emb^r(\Sigma\times N, M) $  close to $j$ and $\tilde g\in  \End^r_\sigma(\Sigma\times N) $ close to $j^{-1}\circ f\circ j$  such that:
%\[\tilde f\circ \tilde j=  \tilde j\circ \tilde g \; .\]
\end{theorem}
\begin{remark} If $f$ is normally expanding at $j(\Sigma\times N)$, then $\I$ must be equal to $\N$. If $f$ is normally hyperbolic but not normally expanding at $j(\Sigma\times N)$, then $\I$ must be equal to $\Z$.\end{remark}
\cref{HPS2} is actually the main motivation to consider the class $\End^r_\sigma(\Sigma\times N)$. Together with \cref{canonical map,embedding normally hyp}, it will enable to show that a finitely generated group of $\Diff^r(N)$ can be robustly embedded into a fibration in any manifold $M$ of dim $\ge 2$, modulo perturbations in $\End^r_\sigma(\Sigma\times N)$.  
Let us prepare also the $k\ge 1$-parameter version of this.

%We will  need also a parametric version of this result, which requires the following:
\begin{definition} A family $(g_p)_{p\in B_k}$ of maps $g_p$ in $ \Emb^r(\Sigma\times N, M)$ or  $\End^r(\Sigma\times N, M)$ is of class $C^r$ if each map $(p, x)\mapsto g_p(\ss,x)$ is of class $C^r$ and depends continuously on $\ss\in \Sigma$. Two such families are close if their corresponding  latter maps are uniformly  $C^r$-close. We denote by $\widehat \Emb^r_k(\Sigma\times N, M)$ and  $\widehat \End^r_{\sigma\, k}(\Sigma\times N)$ these spaces of $C^r$-families of    embeddings and endomorphisms. 
%
%
%
%be the spaces of $C^r$-families $(g_p)_{p\in B_k}$ of maps $g_p$ in respectivelly $ \Emb^r(\Sigma\times N, M)$ and  $\End^r(\Sigma\times N, M)$ such that the maps $(p, x)\mapsto g_p(\ss,x)$ is of class $C^r$ and depends continuously on $\ss\in \Sigma$. Two such families are close if their corresponding  latter maps are $C^r$-close.
\end{definition}

\begin{corollary}\label{coroHPS2} If $(f_p)_p$ is a $C^r$-family of maps $f_p$ of $M$ satisfying the assumption of \cref{HPS2} with $\cal L_p=j_p(\Sigma\times N)$, where $(j_p)_{p\in B_k}$ is a $C^r$-family of embeddings of $\Sigma \times N$ into $M$, 
% and $(j_p)_{p\in B_k}$ is a $C^r$-family of embeddings of $\Sigma \times N$ into $M$, such that for every $p$, the map $f_p$ satisfies the assumption of \cref{HPS2} with $\cal L_p=j_p(\Sigma\times N)$   
%s $(f_p, j_p)$ satisfy the assumption of \cref{HPS2}
 then for every $(\tilde f_p)_p$ $C^r$-close to $(f_p)_p$, there exists a family $(\tilde j_p)_p$ of  embeddings $C^r$-close to $(j_p)_p$ and a family $(\tilde g_p)_p$ of endomorphisms of $\Sigma\times N $ $C^r$-close to $(j_p^{-1}\circ f_p\circ j_p)_p$  such that:
\[\tilde f_p\circ \tilde j_p=  \tilde j_p\circ \tilde g_p \; .\]
\end{corollary}
\begin{proof}
We first consider the $k$-ball $B_k$ as a subset of  $\hat B_k:= P^k(\R)$  and we extend the families    $(f_p)_p$ and $(j_p)_{p}$  to $C^r$-families of maps parametrized by $\hat B_k$. Then we notice that $\hat f:= (p, x)\mapsto (p, f_p(x))$ and $\hat j:= (\ss, p,x)\mapsto j_p^\ss(x)$  satisfies the assumption of \cref{HPS2} for the fibration $\Sigma\times  \hat N$ where $\hat N= \hat B_k\times N$.  As for every $C^r$-perturbation $(\tilde f_p)_p$, the map $\hat {\tilde f}:= (p, x)\mapsto (p, \tilde f_p(x))$ is $C^r$-close to $\hat f$, there exists an embedding  $\hat {\tilde j}$ of $\Sigma\times \hat N$ which is $C^r$-close to $\hat j$ and which is left invariant by $\hat {\tilde f}$. 

Thus for any $p\in B_k$ and $\ss\in \Sigma$, the submanifold $ \hat {\tilde j}(\{\ss\}\times \hat B_k\times N)$ intersects transversally $\{p\}\times M$ at a submanifold  $\tilde j_p^\ss(N)$. By transversality this defines a $C^r$-family $(\tilde j^\ss_p)_{p\in B_k}$ of embedding $\tilde j^\ss_p:N\hookrightarrow M$,
which depends continuously on $\ss\in \Sigma$. 
Put $\tilde j_p: ( \ss, x)\mapsto \tilde j^\ss_p(x)$; it is a $C^r$-embedding of $\Sigma\times N$ into $M$. Also the family  $(\tilde j_p)_{p\in B_k}$ is of class $C^r$. 
  Note that the submanifold $\tilde j_p^\ss(N)$ is sent by $\tilde f_p$ into $\tilde j_p^{\sigma(\ss)}(N)$, since   $\hat {\tilde f}$ leaves invariant  $\{p\}\times M$ and sends 
  $ \hat {\tilde j}(\{\ss\}\times \hat B_k\times N)$ into  $ \hat {\tilde j}(\{\sigma(\ss)\}\times \hat B_k\times N)$, while  $\{p\}\times j_p^{\ss}(N)= \{p\}\times M\cap  \hat {\tilde j}(\{\ss\}\times \hat B_k\times N)$ and $\{p\}\times j_p^{\sigma(\ss)}(N)= \{p\}\times M\cap  \hat {\tilde j}(\{\sigma(\ss)\}\times \hat B_k\times N)$.\end{proof}

The following enables to embed    a finitely generated semi-group of circle diffeomorphisms into the periodic fibers of a normally hyperbolic fibration of any manifold $M$ of dimension $n \ge 2$. A perturbation of this embedding persists is the sense of \cref{HPS2}. 
\begin{proposition}\label{embedding normally hyp}
For any $(g^\sa)_\sa\in \End^r(N)^\sA$ with $N$ a circle and $1\le R\le r \le  \infty$ with $ R<  \infty$,  there are a $C^r$-map $f$ of $M$ and a fibration $\cal L=j(\Sigma\times N)$ with $j\in \Emb^r (\Sigma\times N, M)$ such that:
\begin{enumerate}
    \item if $n\ge 2$ and $\I=\N$, then  the map $f$ is $R$-normally expanding at $\cal L$,
    \item if $n>2$, $\I=\Z$ and each $g^\sa$ is a diffeomorphism preserving the orientation, then  $f$ is a diffeomorphism  $R$-normally hyperbolic~at~$\cal L$,
    \item  $g=j^{-1}\circ f\circ j$ is the canonical endomorphism associated to $(g^\sa)_{\sa\in \sA}$  given by \cref{canonical map}.
\end{enumerate}      
Similarly, for any     $((g^\sa _{p})_{p\in \B_k} )_{\sa\in \sA}\in \widehat {\End_{ k}^r}(N)^\sA$,  there are an embedding  $j\in \Emb^r(\Sigma\times N, M)$ and  $C^r$-family $(f_p)_p$ of maps $f_p$ of $M$    such that  for every $p\in B_k$, 
the maps $f_p$ and $\cal L:=j(\Sigma\times N)$ satisfy   $(1)$, $(2)$ and $(3)$ with $(g^\sa _{p} )_{\sa\in \sA}$.
%the map  $g_p=j^{-1}\circ f_p\circ j$ is the canonical endomorphism associated to $(g^\sa_p)_\sA$ and it satisfies $(1)$ and $(2)$
\end{proposition}
\begin{proof}
It suffices to prove this proposition in the case where $M=(-1,1)^{n-1} \times N$ with $f$  (resp. each $f_p$) coincides with the identity  nearby the boundary of $M$. Indeed as $N$ is a circle, the set  $(-1,1)^{n-1} \times N$ can be embedded into any $n$-manifold   and the dynamics can be extended by the identity outside of this embedding. So let us assume that $M= (-1,1)^{n-1} \times N$. %Let us fix $\se\in \sA$. 
 Let us now focus on the parametric case (the parameter free version is obtained by taking $k=0$).

When $\I=\N$, let $S$ be a smooth map of $(-1,1)^{n-1}$ equal to the identity at the neighborhood of $\partial [-1,1]^{n-1}$ and such that $S$ leaves invariant an expanding  Cantor set $K$ which is conjugate to the shift  on $ \sA^\N$.  Up to replacing $S$ by an iterate and taking a $K$, we can assume that the restriction of  $DS|K$ is $R$-times more  expanding than any derivatives of $(g^\sa_p)_{\sa\in \sA,\,  p\in B_k}$. 

When $n \ge 3$ and $\I=\Z$, let $S$ be a smooth diffeomorphism of $(-1,1)^{n-1}$ equal to the identity at the neighborhood of $\partial [-1,1]^{n-1}$ and such that  $S$ leaves invariant a hyperbolic horseshoe   $K$ which is conjugate to the shift $\sigma$ on $\sA^\Z$. 
We are going to construct $f_p $ of the form: $f_{p}:= (\ss,x) \mapsto (S^m(\ss), f^{\ss}_{p}(x))$ for some $m\ge 1$. Up to replacing $S$ by an iterate and taking a subset of $K$, we can assume that the restriction of  $DS$ to the stable and unstable directions of $K$ are $R$-times more contracting and expanding than any derivativative of $(g^\sa_p)_{\sa\in \sA, \,  p\in B_k}$. 

\begin{fact}  There is  a $C^r$-family  $(f_p^\ss )_{\ss\in (-1,1)^{n-1}, p\in B_k}$ of self-maps  $f_p^\ss$ of $N$  such that for any $p\in B_k$:
\begin{enumerate}
    \item $f^\ss_p$ coincides with the identity nearby the boundary of $ (-1,1)^{n-1}$,
    \item  $f^\ss_p=g_p^{\sa}$ if $\ss$ belongs to $ K$ and corresponds to a $\sA$-sequence with $\sa\in   \sA$ at the $0$-position.
    \item If for every $p$, the maps $(g^\sa_p)_{\sa\in \sA} $ are orientation preserving diffeomorphims, then the maps $f^\ss_p$ are orientation preserving  diffeomorphims for every $\ss\in [-1,1]^{n-1}$. 
\end{enumerate}\end{fact}
   \begin{proof} The construction can be done using bump function and homotopies in the set of $C^r$-endomorphisms of $N$  or $C^r$-diffeomorphisms preserving the orientation of $N$. \end{proof}  
We can consider the immersion $j: \Sigma\times N\to K'\times N\subset K\times N\subset M$ which is $R$-normally expanding or hyperbolic for 
$f_{p}:= (\ss,x) \mapsto (S(\ss), f^{\ss}_{p}(x))$. 
\end{proof}
\begin{remark}\label{pour variete Wssuu}
In the proof of the latter proposition for the diffeomorphisms case, we can assume that the hyperbolic basic set $K$ has unstable dimension $1$ and that it is a subset of an attractor $A$  (for instance using a Plykin attractor). Then there is a neighborhood of $N\times A$ which is sent into itself  by the dynamics and on which the dynamics is partially hyperbolic. So this neighborhood is  $C^{r-1}$-foliated by local strong stable manifold of dimension $\dim M-2$, by \cite{HPS} or \cite[thm 3.2]{RHRHU97}. Also if $f$ depends $C^r$-on a parameters, this foliation depends $C^{r-1}$ on the parameter.
\end{remark}

It will be more comfortable while working with blenders to only deal with endomorphisms of $\sA^\N\times N$. In order to do so, we consider
the canonical projection $\rho:  \Sigma= \sA^\Z\to  \sA^\N$  and put $\ss\sim \ss'$ if $\rho(\ss)= \rho(\ss')$ for $\ss, \ss'\in \Sigma$. Let  $\tilde \Sigma:=\{(\ss, \ss')\in \Sigma^2:  \ss\sim \ss' \}$.% We notice the space of that $\Sigma/\sim\approx   \sA^\N$. 
%    Here is a result which will enable us to work intrinsically only with the case $\I=\N$.
\begin{proposition}\label{to monoinfinite shift}
Under the setting  of \cref{embedding normally hyp} with $\I=\Z$, there exists a continuous family $(\mathsf {hol}^{\ss,\ss'})_{(\ss,\ss')\in \tilde \Sigma}$ of homeomorphisms of $N$ such that:
\begin{enumerate} 
\item $\mathsf {hol}^{\ss, \ss'}$ is the identity if $\ss=\ss'$,
% for every  $\ss\in \iota(\sA^\N)$. 
\item $\mathsf {hol}^{\sigma (\ss),\sigma (\ss')}\circ \tilde g^{\ss}= \tilde g^{\ss'}
\circ  \mathsf {hol}^{\ss, \ss'}$ for every $(\ss,\ss')\in\tilde  \Sigma $.
\item If  $\dim N=1$ then  $\mathsf {hol}^{\ss,\ss'}$ is in $\Diff^{r-1}(N)$  and  depends continuously on $(\ss,\ss')\in \tilde \Sigma$ 
\end{enumerate} 
Under the setting of \cref{coroHPS2} with $\I=\Z$, there exists a continuous family $(\mathsf {hol^{\ss,\ss'}_p})_{p\in B_k}$ of  homeomorphisms $\mathsf {hol^{\ss,\ss'}_p}$ of $N$ satisfying $(1)$   and $(2)$ with $\tilde g_p$ for every $p\in B_k$. If  $\dim N=1$ then  $(\mathsf {hol}^{\ss,\ss'}_p)_p$ is a $C^{r-1}$ family of $C^{r-1}$-diffeomorphisms and  depends continuously on $(\ss,\ss')\in \tilde \Sigma$. 
%moreover 
%
%$C^{r-1}$ 
%\in \Diff^{r-1}(N)$  depending continuously on ${\ss,\ss'}\in \tilde \Sigma$ and  satisfying $(1)$, $(2)$  and $(3)$ with $\tilde g_p$ for every $p\in B_k$. 
\end{proposition}
\begin{proof}We define $\mathsf {hol}^{\ss,\ss'}$ as the holonomy from $\tilde j^\ss(N)$ and 
 $\tilde j^{ \ss'}(N)$ along the strong stable foliation defined in \cref{pour variete Wssuu}.%and use the regularity of this foliation, see \cite{HPS} or \cite[thm 3.2]{RHRHU97}.
\end{proof}
Given a section  $\iota :   \sA^\N\hookrightarrow \sA^\Z$ of $\rho$, the latter proposition defines a semi-conjugacy between $g$ and the map $(\ss, x)\in \sA^\N\times N \mapsto (\sigma(\ss), g^{\iota(\ss)}(x))$ and so enables to focus  on the case $\I=\N$.  
%
%The following is a section of $\rho$:
%\[\iota : (\sa_j)_{j\ge 0} \in \sA^\N\mapsto (\sa_j)_{j\in \Z}\quad \text{with }\sa_j= \sc,\; \forall j<0\]

 \subsection{Intrinsic definition of $(\lambda)$-blender% and $C^r$-$(\lambda)$-parablender
 }\label{def lambda blender}
The notion of blender was introduced by Bonatti and Diaz in \cite{BD96}, as a hyperbolic  basic set  for a  $C^1$-self-map  (resp. diffeomorphism) $f$ of a manifold. Such are included in a normally expanding (hyperbolic) fibration and so always included in a fibration $\Sigma\times N$ embedded into $M$, whose fibers are either compact or weakly contracted by the dynamics.  The dynamics induced by $f$ on this fibration is a map  $g\in \End^r_\sigma (\Sigma\times N)$ for a certain $\sigma\in C^0(\Sigma, \Sigma)$. This leads us to propose an intrinsic definition of blender: it regards only the map $g$.   For the sake of simplicity, we  restrict our study to the case where  $\sigma:\Sigma\to \Sigma$ is the shift on $\sA^\I$, with $\I= \N$ or $\Z$. In the case $\I=\Z$, using the holonomy along the strong stable manifolds defined  in \cref{to monoinfinite shift}, we recall that the map $g$ is semi-conjugated to one for which $\I=\N$. So we will focus on the case $\I=\N$.  For the sake of completeness the standard definition of blender is recalled in \cref{appendix}, it is not more general nor simpler to use than the intrinsic definition we shall introduce.
 
\medskip 

Let $g\in \End^1_\sigma(\Sigma\times N)$ with $\Sigma= \sA^\N$. A $g$-invariant compact subset $\Lambda\subset \Sigma \times N$   is a  \emph{basic set} if it is {transitive} and  \emph{locally maximal}: it is the maximal invariant set in one of its neighborhood $V_\Lambda$. The basic set $\Lambda$ is \emph{centrally contracting} if 
$g|\Lambda \cap \{\ss\}\times N$ is contracting   for every $\ss\in \Sigma$.  Then 
for every $C^1$-perturbation 
$\tilde g$ of  $g$, the maximal invariant $\tilde \Lambda:= \bigcap_{n\in \Z} \tilde g^n(V_\Lambda)$ is called the \emph{continuation of $\Lambda$}. One can show that it does not depend $V_\Lambda$ provided that $\tilde g$ is sufficiently $C^1$-close to $g$ and that $\tilde \Lambda$ is a  centrally contracting basic set for $\tilde g$.
Let $\arr V_\Lambda(\tilde g)$ be the set of $(v_n)_{n<  0} \in   V_\Lambda^{\Z^-}$ such that  $\tilde g (v_{n-1})=v_n$ for every $n< -1$. By central contraction and local maximality of $\tilde \Lambda$, one can show that $v_n \to \tilde \Lambda$ when $n\to -\infty$. Thus the following is called a \emph{local unstable set} of $\tilde \Lambda$:
\[W^u_{loc} (\tilde \Lambda):= \{v_0\in \Sigma\times N: 
\exists (v_n)_{n<  0} \in \arr V_\Lambda(\tilde g)\text{ and } \tilde g(v_{-1})=v_0\}\; .\] 
When $\tilde g=g$, we put $W^u_{loc} (\Lambda):=W^u_{loc} (\tilde \Lambda)$. When $\tilde g\neq g$, we say that $W^u_{loc} (\tilde \Lambda)$ is \emph{the continuation} of $W^u_{loc} (\Lambda)$. 
The following is an intrinsic definition of  the Bonatti-Diaz blender \cite{BD96}.
\begin{definition}[Blender for $\End^1_\sigma(\Sigma\times N)$] \label{def: blender intrinic}The centrally contracting, basic set $\Lambda$ is a \emph{blender} if a local unstable set $ W^u_{loc}(\Lambda) $ has nonempty interior and  for any  subset $K\Subset  W^u_{loc}(\Lambda)$, for every  $C^1$-perturbation $\tilde g$ of $g$, the continuation $ W^u_{loc}(\tilde \Lambda)$ contains $K$.  
\end{definition}
\begin{example} \label{example of blender skew}
Let  $(g^\sb)_{\sb\in \sB_0} $ be the family of maps of \cref{blender IFS}. Let   $\sA$ be a finite alphabet which contains $\sB_0:=\{-, +\}$.  
Let $g  \in \End^1(\Sigma \times \R )$  be such that $g| \sB_0^\N\times \R$ is the canonical map associated to  $(g^\sb)_{\sb\in \sB_0}$ (see \cref{canonical map}).  Then the set  $\Lambda= \sB_0^{\N}\times [-1,1]$ is a blender for $g$ and 
$\Sigma\times [-1,1]$ is a local unstable set of $\Lambda$. 
\end{example}
\begin{proof} 
Let $V_\Lambda$ be the subset of points $(\ss, x)\in \Sigma\times (-2,2)$ such that  the first letter of $\ss$  is in $\sB_0$. It is a neighborhood of $\Lambda$ and so its maximal invariant set must contain $\Lambda$. Also the maximal invariant set of  $V_\Lambda$ must be included in $\sB_0^\N\times (-2,2)$ and so is equal to $\Lambda$ by central contraction. Let $K:=(-1+\eta, 1-\eta)$ for $\eta>0$ small. 

By the open covering property proved in \cref{blender IFS}, for  $\tilde g$ sufficiently close to $g$, for all $\ss_0\in \Sigma= \sA^\N$ and $x_0\in K$, their exists $\sa_{-1}\in \sB_0$ and $x_{-1}\in K$  such that $g^{\ss_{-1}}(x_{-1})=x_0$ with $\ss_{-1}\in \sigma^{-1}(\ss_0)$ with $0$-coordinate equal to $\sa_{-1}$. We can iterate this process to construct a preorbit $(\ss_{-n}, x_{-n})_{n\le -1} \in \arr V_\Lambda(\tilde g)$ such that each  $x_{-n}$ is in $ K$ and $g(\ss_{-1}, x_{-1})=(\ss_0, x_0)$. 
\end{proof}

Under the assumptions of \cref{def: blender intrinic}, let $V_\Lambda$ be the neighborhood of $\Lambda$ defining the local unstable set $W^u_{loc}(  \Lambda)$. Let $\tilde g$ be a $C^1$-perturbation of $g$. The \emph{Lyapunov fiber} $\lambda( \tilde g, \arr v)$ of  $\arr v=(v_{j})_{j\le 0}\in   \arr V_\Lambda(\tilde g)$ is the  set of cluster values of $(\frac1{m}\log |\partial_x   \tilde  g^m(v_{-m})|)_m$; this is an interval of $\R$.  
The \emph{Lyapunov fibration} of $W^u_{loc} (\tilde \Lambda) $ is:
\[\lambda(W^u_{loc} (\tilde \Lambda)):= \bigcup_{\arr v \in \arr V_\Lambda(\tilde g)} \{ v_0 \}\times \lambda( \tilde g, \arr v)\]

Here is the counterpart of the notion of $\lambda$-blender introduced for finitely generated semi-groups of circle diffeomorphisms. This new notion is devoted to  prove Theorem \ref{theoA} using \cref{Densiteparabolic}.
 \begin{definition}[$\lambda$-Blender for  $\End^1_\sigma(\Sigma \times N)$]\label{def: lambda blender intrinic} 
A blender $\Lambda$ for  $g \in \End^1(\Sigma\times N ) $ is  a   $\lambda$-\emph{blender}  if the Lyapunov fibration of a local unstable set  $ W^u_{loc} (\Lambda) $ has nonempty interior  and given any  subset $ K \Subset \lambda(W^u_{loc} (\Lambda) )$, 
 for every $C^1$-perturbation of $\tilde g$,  the Lyapunov fibration of the continuation $ W^u_{loc} (\tilde \Lambda) $ contains $ K $.
\end{definition}
A generalization of the following will be  proved in \cref{exem lambda parablender skewproduct}.
\begin{example}\label{skewprod lambda blend}
Let $(g^\sb)_{\sb\in \sB_0^\lambda}$ be the family of maps of  \cref{exem lambda blender IFS}. Let   $\sA$ be a finite alphabet which contains $\sB^\lambda_0$ and $\Sigma:= \sA^\N$.   Let $g  \in \End^1(\Sigma \times \R )$  be such that $g| (\sB_0^\lambda)^\N\times P^1(\R)$ is the canonical map associated to  $(g^\sb)_{\sb\in \sB_0}$. Then  $\Lambda= (\sB_0^\lambda)^{\N}\times [-1,1]$ is a $\lambda$-blender for $g $. Then the set $ W^u_{loc} (\Lambda) =\Sigma\times [-1,1]$ is a local unstable set of $\Lambda$ and its Lyapunov fibration is:
\[ W^u_{loc} (\Lambda) \times \left[\log \frac 23-\epsilon ,\log \frac 23+\epsilon \right]\]
\end{example}
\subsection{Intrinsic definition of  $C^r$-$(\lambda)$-parablender}\label{intrinsic def parablender}
Let $ g= (g_p)_{p\in B_k}$ be a  $C^r$-family  of maps $g_p\in \End^r_\sigma(\Sigma\times N)$, for $\Sigma= \sA^\N$, $1\le r<\infty$ and $k\ge 0$. For $\ss\in \Sigma$ and $p\in B_k$, let $g^\ss_p$ be the second coordinate of the restriction  $g_p|\{\ss\}\times N$. Let $g^\ss=(g^\ss_p)_p$.
For every $p_0\in B_k$, we denote:
\[ J^r_{p_0}  g:= (\ss, J^r_{p_0} x)\in\Sigma \times J^r_{p_0} N\mapsto  (\sigma(\ss), J^r_{p_0} g^\ss(   J^r_{p_0}x))\in \Sigma \times J^r_{p_0} N\; .\]
 
Assume that  $g_{p_0}$ has a blender  $\Lambda_{p_0}$ which is locally maximal in a certain neighborhood $V_\Lambda $ of $\Lambda_{p_0}$. Let     $W^u_{loc}(\Lambda_{p_0})$ be the local unstable set of $\Lambda_{p_0}$ associated to $V_\Lambda$. 
Let    $\Sigma'$ be the image of $V_\Lambda$ via the first coordinate projection $\Sigma \times N\to \Sigma$.  Let $\arr \Sigma'$ be the set of $\sigma$-preorbits $(\ss_j)_{j< 0}\in \Sigma'^{\Z^-}$ and put $\ss_0= \sigma(\ss_{-1})$.  
 We choose a continuous family  $\Omega = (\Omega_\ss)_{\ss\in \Sigma'}$ such that $\{(\ss, \omega_\ss): \ss\in \Sigma'\}$ is included in $V_\Lambda$. 
Then the set $W^u_{loc} (\Lambda_{p_0})$ is formed by the points $v(\u \ss)= \lim_{i \to +\infty}  g_{p_0}^{ i } (\ss_{-i}, \Omega_{\ss_{-i} })$  among  $\u \ss =(\ss_j)_{j<0}\in \arr \Sigma' $.  Note that:  
\[v(\u \ss)= (\ss_0, X_{p_0}(\u \ss))\text{  with }  X_{p_0}(\u \ss) = \lim_{i \to +\infty}  g_{p_0}^{\ss_{-i}}\circ  \cdots \circ g_{p_0}^{s_{-1}}(\Omega_{\ss_{-i} })\text{ and } \ss_0=\sigma(\ss_{-1})\; .\] 
  This limit does not depend on $\Omega$. Also for every $p$ nearby $p_0$,  the continuation $W^u_{loc} (\Lambda_p)$ of $W^u_{loc} (\Lambda_{p_0})$ is formed by the  points $v(\u \ss)= (\ss_0, X_{p}(\u \ss))$ with $X_{p}(\u \ss) = \lim_{\infty}  g_{p}^{\ss_{-i} \cdots s_{-1}}(\Omega_{\ss_{-i}})$. 
Furthermore  the  family $(X_{p}(\u \ss))_p$  is of class $C^r$ for every $\u \ss  \in \arr \Sigma'$ and depends continuously on $\u \ss$. Put:
\[  J^r_{p_0}W^u_{loc}(  \Lambda) :=   \{(\ss_0, J^r_{p_0} X(\u \ss)): {\u \ss \in \Sigma'}\}\subset \Sigma\times J^r_{p_0} N\; .\]
Similarly for $\tilde f$ $C^r$-close to $g$ and $p$ close to $p_0$, we can define $\tilde X_{p}(\u \ss) = \lim_{\infty}  \tilde g_{p}^{\ss_{-i} \cdots s_{-1}}(\Omega_{\ss_{-i}})$ and:
  \[  J^r_{p_0}W^u_{loc}( \tilde  \Lambda) :=   \{(\ss_0, J^r_{p_0} \tilde X(\u \ss)): {\u \ss \in \Sigma'}\}\subset \Sigma\times J^r_{p_0} N\; .\]

\begin{definition}[$C^r$-parablender for $\widehat {\End^r_{\sigma, k}}$]\label{def parablender}
The continuation   $(\Lambda_{p})_{p }$  is a \emph{$C^r$-parablender  at $p_0$}  if 
 $J^r_{p_0}W^u_{loc}(  \Lambda) $ has nonempty interior and for every $K \Subset
 J^r_{p_0}W^u_{loc}(  \Lambda) $, the  continuation
$ J^r_{p_0}W^u_{loc}( \tilde  \Lambda) $ of  $J^r_{p_0}W^u_{loc}(  \Lambda) $
of any sufficiently small  $C^r$-perturbation $ \tilde g $ of $ g $  contains $ K$. 
\end{definition}
We give an example below of parablender. Let us continue by introducing the $C^r$-$\lambda$-parablender. Now $N$ is a curve with tangent space identified to $N\times \R$.  The \emph{Lyapunov fiber} of $(\ss_0, J^r_{p_0} X(\u \ss))$ for $\underline \ss \in \arr \Sigma'$ is the set $\lambda(g,  J^r_{p_0}  X(\u \ss))$   of cluster values of $(J^{r-1}_{p_0} \frac1{i} \log |\partial_x  g_{p}^{i}(\ss_{-i}, \Omega_{\ss_{-i}})|)_{i>0} $. The following is  the \emph{Lyapunov fiberation} of $J^r_{p_0}W^u_{loc}(\Lambda)$: 
  \[   \lambda(J^r_{p_0}W^u_{loc}(\Lambda)):=  \bigcup_{\u \ss\in \arr \Sigma'} \{(\ss_0, J^r_{p_0}   X(\u \ss))\}\times \lambda(g,  J^r_{p_0}  X(\u \ss))\; .\]
%\[  \bigcap_{n\ge 1}    cl \left\{\Big(J^r_{p_0}g_{p}^{i}(\ss_{-i}, \Omega_{\ss_{-i}}), J^{r-1}_{p_0} \frac1{i} \log |\partial_x  g_{p}^{i}(\ss_{-i}, \Omega_{\ss_{-i}})  |\Big): i\ge n\text{ and } 
%\u \ss= (\ss_{-i})_i\in \arr \Sigma'   
%  \right\}\; .\]
%
%Here is the para-counterpart of the Lyapunov fiber:
%  \[   \lambda(J^r_{p_0}W^u_{loc}(\Lambda)):=  \bigcap_{n\ge 1}    cl \left\{\Big(J^r_{p_0}g_{p}^{i}(\ss_{-i}, \Omega_{\ss_{-i}}), J^{r-1}_{p_0} \frac1{i} \log |\partial_x  g_{p}^{i}(\ss_{-i}, \Omega_{\ss_{-i}})  |\Big): i\ge n\text{ and } 
%\u \ss= (\ss_{-i})_i\in \arr \Sigma'   
%  \right\}\; .\]
We notice that $ \lambda(J^r_{p_0}W^u_{loc}(\Lambda))$ is a compact subset of $\Sigma \times J^r_{p_0} N\times J^{r-1}_{p_0} \R$. 
Here is the parametric counterpart of the $\lambda$-blender defined in \ref{def: lambda blender intrinic}; it  is devoted to  prove  Theorem \ref{theoB} using \cref{ThmS3}.
\begin{definition}[$\lambda$-parablender  for $\widehat {\End^r_{\sigma, k}}(\Sigma\times N )$]\label{def lambda parablender} The continuation  $(\Lambda_p)_p$  of a blender       
 is  a $C^r$-$\lambda$-\emph{parablender}  at $p_0$  if $ \lambda(J^r_{p_0}W^u_{loc}(\Lambda))$   has nonempty interior  and  any subset $ K \Subset \lambda(J^r_{p_0}W^u_{loc}(\Lambda) )$ is contained in the 
  Lyapunov fibration   of $J^r_{p_0}W^u_{loc}(\tilde \Lambda)$  for  any sufficiently small  $C^r$-perturbation $(\tilde g_p)_{p}$ of $(g_p)_{p}$.\end{definition}
\begin{remark}
We notice that  $\Lambda_{p_0}$ is a $\lambda$-blender for $g _{p_0} $ and 
$(\Lambda_p)_p$ is a $C^r$-parablender at $p=p_0$ for  $(g_p)_p$. When $k=0$,   $(\Lambda_p)_p$  is a  $C^r$-$\lambda$-\emph{parablender}  at $p_0$ iff $\Lambda_{p_0}$ is a $\lambda$-blender for $g _{p_0} $.\end{remark}
 
As promised here is an example of $C^r$-parablender.
\begin{example}\label{exem parablender skewproduct} Let  $(g^\sb_p)_{p\in B_k, \sb \in \sB_r}$ be the family of maps of  \cref{expparablendergroupe}. Let   $\sA$ be a finite alphabet which contains $\sB_r$.  
Let $(g_p)_p  \in \widehat {\End^r_{\sigma, k}}(\Sigma\times N )$  be such that $g_p| \sB_r^\N\times \R$ is the canonical map associated to  $(g^\sb_p)_{\sb\in \sB_r}$ for every $p\in B_k$.  Then the set  $\Lambda_0= \sB_r^{\N}\times [-1,1]$ is a blender for $g_0$ and its continuation for $(g_p)_p$  is a $C^r$-parablender  at $p=0$ with a local unstable set %$W^u_{loc}(\Lambda_0)$
  satisfying:
\[J^r_0W^u_{loc}(\Lambda):= \Sigma\times  \left \{ \sum_{i\in E_r } \xi_i \cdot p ^{i}: (\xi_i)_i\in [-1,1]^{E_r}\right\}\]
 \end{example}
We skip the proof of this example since it is similar to the one for IFS and useless for the proof of main Theorem \ref{theoB}. Let us focus on the following:

\begin{example} \label{exem lambda parablender skewproduct}
 Let  $(g^\sb_p)_{p\in B_k, \sb \in \sB_r^\lambda}$ be the family of maps of  \cref{lambda parablender IFS} for $\epsilon>0$ sufficiently small. Let   $\sA$ be a finite alphabet which contains $\sB_r$.  
Let $(g_p)_p  \in \widehat  {\End^r_{\sigma, k}}(\Sigma \times \R )$  be such that $g_p| (\sB_r^\lambda)^\N\times \R$ is the canonical map associated to  $(g^\sb_p)_{\sb\in \sB_r^\lambda}$ for every $p\in B_k$.  Then the set  $\Lambda_0=  (\sB_r^\lambda)^\N \times [-1,1]$ is a $\lambda$-blender for  $g_0$ and its continuation for $g=(g_p)_p$  is a $C^r$-$\lambda$-parablender  at $p=0$ with a local unstable set satisfying:
\[\lambda(J^r_0W^u_{loc}(\Lambda)) = \Sigma\times  
 \left \{ \Big(\sum_{i\in E_r } \xi_i \cdot p ^i,\log \frac23 +\epsilon \sum_{i\in E_{r-1} } \lambda_i \cdot p^i \Big)
 : (\xi_i)_i\in [-1,1]^{E_r}, (\lambda_i)_i\in [-1,1]^{E_{r-1}}\right\}\]
 \end{example}
 \begin{proof}
Let $V_\Lambda$ be the subset of points $(\ss, x)\in \Sigma\times (-2,2)$ such that  the $0$-coordinate of $\ss$  is in $\sB_r^\lambda$. It is a neighborhood of $\Lambda_0$ .  For the same reasons as in  \cref{example of blender skew},   the local unstable set associated to this neighborhood is  $W^u_{loc}( \Lambda_{0})= \Sigma \times  [-1,1]$.  Let $\Sigma'$ be the subset of points in $\Sigma$ whose $0$-coordinate is in $\sB_r^\lambda$. Let $\arr \Sigma'$ be the set of $\sigma$-preorbits  $(v_j)_{j<0}$ with $v_j\in \Sigma'$. Let $\Omega_\ss=0$   for every $\ss\in \Sigma$.  As for  \cref{lambda parablender IFS}, it is easy to see that  the Lyapunov fibration $\lambda(J^r_0W^u_{loc}(\Lambda))$ has the above form. It remains to show that it is a $C^r$-$\lambda$-parablender. For $\eta>0$ small, put:
\[K = \Sigma\times  
 \left \{ \Big(\sum_{i\in E_r } \xi_i \cdot p ^i,\log \frac23 +\epsilon \sum_{i\in E_{r-1} } \lambda_i \cdot p^i \Big)
 : (\xi_i)_i\in I^{E_r}, (\lambda_i)_i\in I^{E_{r-1}}\right\}\, \quad \text{with } I:= [-1+\eta ,1-\eta] \; .\]
%\[  K:= \left \{  \sum_{i \in E_r } x_i \cdot p ^i 
% :  x_i  \in [-1+\eta ,1-\eta] \right \}\qand  \lambda(K):=K\times \left \{\log\frac 23+\epsilon\cdot \ell: \ell \in  K \right \}\; .\]
 
%  \{  (\sum_{i} x_i \cdot p ^i,\log \frac23 +\epsilon \sum_{i  } \lambda_i \cdot p^i  ) :  x_i ,\lambda_i \in [-1+\eta ,1-\eta]  \}$.
  Let $\tilde g $  be a  $C^r$-perturbation of $g $. Let us show for every $ (\ss_0, x_0,\ell)  \in  K $ and $n> 0$, the existence of a sequence of letters $(\sb_j)_{-n\le j< 0}\in (\sB^\lambda_r)^{n}$ and a sequence of points $(\ss_j,   x_j, \ell_j)_{-n\le j< 0}\in  K^{n } $ such that with $\ell_0=0$ and  $\ss_{j}\in \Sigma$ equal to the concatenation of $\sb_{j}\cdots \sb_{-1}$ with $\ss_0$,  it holds:
  \begin{enumerate} 
  \item    $J^r_{p_0} \tilde g^{\ss_i}( x_i)=  x_{i+1}$ for every $-n\le i<0$,
 \item 
$\ell_{-n}=J^{r-1}_{p_0}( \log  |\partial_x \tilde g^{n}_p(\ss_{-n}, x_{-n})| )_p\in \left \{ n \cdot \ell + \sum_{i \in E_{r-1}  } \lambda_i \cdot p ^i 
 :  \lambda_i  \in [-2\epsilon, 2\epsilon]  \right \}$.
  \end{enumerate}
The step $n=0$ is obvious. Assume the result shown for $n\ge 0$. 
Then by the covering property showed on  in \cref{blender IFS,expparablendergroupe}, 
with $\sb_{-n-1}\in  \sB^\lambda_r $ the vector whose coefficients are the sign of those of the polynomial $(x_{-n}, -\log  |\partial_x \tilde g^{n}_p(\ss_{-n}, x_{-n})| +n\cdot \ell )$. Then for the same reasons as for  
\cref{blender IFS,exem lambda blender IFS}
 the induction hypothesis holds true at step $n+1$.  
 Using that $\tilde g$ is contracting nearby $\Lambda$, by \cref{prop 1}.(1), 
 %\color{red} Then this achieves the proof of since using that $\tilde g$ is contracting nearby $\Lambda$,  with $\u \ss= (\ss_{-j})_{j\le  0}\arr \Sigma'$, 
 it holds that $J^r_{p_0}\tilde g^{\ss_{-i}}\circ  \cdots \circ \tilde g^{\ss_{-1}}(\Omega_{\ss_{-i}})-J^r_{p_0}\tilde g^{\ss_{-i}}\circ  \cdots \circ J^{r-1}_{p_0}\tilde g^{\ss_{-1}}(x_{-i})$ is small when $i$ is large. This implies directly  first limit and by Ces\`aro means the second one:
\[J^r_{p_0} \tilde X(\u \ss)=\lim_{i\to \infty} J^r_{p_0}\tilde g_{p}^{\ss_{-i}}\circ  \cdots \circ \tilde g_p^{\ss_{-1}}(x_{-i})=
x_0\qand 
\lambda(\tilde g, J^r_0 \tilde X(\u \ss))=  \lim_{i\to \infty}\frac{\ell_{-i}} i  =\ell\, .
\]
%(\ss_0, x_0,\ell) \in\bigcap_{n\ge 0}    \left\{\Big(J^r_{p_0}\tilde g_{p}^{i}(\Omega_{\ss_{-i}}), J^{r-1}_{p_0} \frac1{i} \log |\partial_x  \tilde g_{p}^{i}(\ss_{-i}, \Omega_{\ss_{-i}})  |\Big): i\ge n \right\} \subset \lambda(J^r_{p_0}W^u_{loc}(\tilde \Lambda)) \; .
%  \] 
 \end{proof}
\section{Proof of the main Theorems} 
 We have introduced all the concepts and techniques needed to prove main Theorems \ref{theoA} and \ref{theoB}. Both  proofs are similar: we will show that the sets $D$ and $\hat  D$ involved in 
 \cref{Densiteparabolic,ThmS3} are locally dense. In order to do so, we are going to use  $\lambda$-($C^r$-para)-blender and implement the dynamical rescaling techniques of \cref{Density of rotations from lambda-blender IFS}. The density in  the proof of Theorem \ref{theoA} will be done along $2$-parameter families. This will enable to obtain even the local density of analytic maps with a normally hyperbolic, periodic circle at which $f$ is parabolic.
 
\subsection{Theorem  \ref{theoA}}\label{sec:Theorem  theoA}
\begin{proof}[Proof of Theorem \ref{theoA}] It suffices to show that the closure of the set $D$ of \cref{Densiteparabolic} has  nonempty  $C^r$-interior for any $2\le r\le \infty$ which  intersects $\Diff^r(M)$  when $n=\dim M\ge 3$. 

Let $\sB^\lambda_0$ and $(g^\sb)_{\sb \in \sB^\lambda_0}$ be as  in \cref{exem lambda blender IFS}. We consider the extension of  these maps to the projective space $P^1(\R)$. 
 Let $\sA= \sB^\lambda_0\cup \{\sc\}$ for a symbol $\sc$  and put: \[g^\sc: x\in P^1(\R)\mapsto \frac32 \frac x{x+1}\in P^1(\R).\]  
By \cref{embedding normally hyp} with $ R= 2$ and $r=\infty$, there are $j\in \Emb^\infty( \Sigma\times P^1(N), M)$ and $C^\infty$ self-map $f$ of $M$  (which is a diffeomorphism if $n\ge 3$) which is $R$-normally expanding (or hyperbolic if $n\ge 3$) at $\cal L= j(\Sigma\times P^1(\R))$ and such that $g=j^{-1}\circ f\circ j$ is the canonical map of $\End^\infty_\sigma(\Sigma\times P^1(\R))$ associated to $(g^\sA)_{\sa\in \sA}$ (see \cref{canonical map}).

The idea is to apply a similar perturbation scheme  as for \cref{coroImportantgroup}.  This leads us to unfold $f$ into a 2-parameter family. 
Let $W$ be a neighborhood of $0\in \R^2$ and for $p=(u,v)\in W$, let:
\[g^\sb_p:= x\mapsto g^\sb(x-v)+v\quad \forall \sb\in \sB_0^\lambda \qand g^\sc_{p}:= x\mapsto \exp(u)\cdot  g^\sc(x)\; .\] 
Note that $0$ is a fixed point of $g^\sc_{p}$ with Lyapunov exponent $  \log \frac 32+u$. It is repulsive for  $W$ is small. 
 
Again by \cref{embedding normally hyp}, there is a $C^\infty$-family $(f_p)_{p\in W}$ of self-maps such that $f_0=f$ and such that $f_p$ leaves invariant $\cal L$ and satisfies $g_p=j^{-1}\circ f_p\circ j$ is the canonical endomorphism of $\Sigma\times P^1(\R)$ associated to $((g^\sa_p)_{p\in W})_{\sa\in \sA}$.% Put  $f=f_0$.  

\medskip

Note that given any  $C^R$-neighborhood $\cal M_W$ of $(f_p)_p$, the set  $\cal M= \{\tilde f_0: (\tilde f_p)_p\in\cal M_W\}$ is a neighborhood of $f$. Thus  it suffices to show that for  $\cal M_W$ sufficiently small,  for every $(\tilde f_p)_p\in \cal M_W$, there is an arbitrarily small parameter $p$ such that $\tilde f_p$ has a periodic fiber at which it is   parabolic. Indeed this defines a set of self-maps of $M$ whose closure   contains $\cal M$. 

Let $(\tilde j_p)_p$ be the family of embeddings given by \cref{coroHPS2} and let $\tilde g_p:= \tilde j_p^{-1}\circ \tilde f_p \circ \tilde j_p\in \End^R_\sigma(\Sigma\times P^1(\R))$. Recall that $(\tilde g_p)_p$ is  $C^R$-close to $(g_p)_p$ when $\cal M_W$ is small and $\tilde g_p$  is of the form:
\[\tilde g_p(\ss, x)= (\sigma(\ss), \tilde g^\ss_p(x))\; ,\]

Let us simplify the setting:
\begin{fact}\label{simple setting}Up to a  reparametrization of $\Sigma\times P^1(\R)$ depending $C^2$  on $p$,  we can assume that for every $\ss\in \Sigma$ whose $0$ coordinate is $\sc$, it holds $g^\ss_p(0)=0$ for every $p$ small. 
\end{fact}
\begin{proof} 
We define the operator $(x^\ss_p)_{\ss\in \Sigma}\mapsto (\tilde x^\ss_p)_{\ss\in \Sigma}$ where 
$\tilde x^\ss_p= 0$  if the first letter of $\ss$ is not $\sc$ and 
$\tilde x^\ss_p= (\tilde g_p^\ss)^{-1} (x_p^{\sigma(\ss)})$ otherwise. It is a contracting operator at the neighborhood of $(0)_{\ss\in \Sigma}$. Let 
 $(\tilde O^\ss_p)_{\ss\in \Sigma}$ be its fixed point. It depends $C^2$ on $p$ since the operator does. Also when $\tilde g=g$, we have  
 $(\tilde O^\ss_p)_{\ss\in \Sigma}= (0)_{\ss\in \Sigma}$ for every $p$, thus   $(\tilde O^\ss_p)_{\ss\in \Sigma}$ is $C^2$-small. 
%For $p$ small, let $(\tilde O^\ss_p)_{\ss\in \Sigma}$ be a continuous family of points nearby $0$ such that for every $\ss\in \Sigma$, if the first letter of $\ss$ is not $\sc$ then $\tilde O_p^\ss=0$, otherwise put $ \tilde O_p^\ss =(\tilde g_p^\ss)^{-1} (\tilde O_p^{\sigma(\ss)})$. Note that  $(\tilde O^\ss_p)_{p}$ is a $C^2$-family which is $C^2$-close to  $0$ when   $(\tilde g_p)_p$ is $C^2$-close to $(g_p)_p$, and this depends continuously on $\ss\in \Sigma$. 
Thus using a coordinate change via a translation by $\tilde O^\ss_p$, we obtain the sought property.\end{proof}

The idea is to apply a variation of the proof of \cref{Importantclaimgroup}.  This involves  
 the space $\sP_m$ of  $2m$-periodic points $\ss\in \Sigma$ of the form $(\sb_m \cdot \sc^m)^\infty$: their letters at the positions $0, ..., m-1$ form the word $\sb_m\in (\sB_0^\lambda)^m$, the $m$ next letters are $\sc$. We will fix  $m\ge 1$ and $\ss\in \sP_m$ in function of the perturbation $(\tilde g_p)_p$ of $(g_p)_p$, and more precisely the properties of the following maps:
\begin{equation}\label{rceil rfloor0}
    \tilde g^{\sb_m }_p:=\tilde g^{\ss^{m-1}}_p\circ \cdots \circ \tilde g^{\ss^0}_p\qand \tilde g^{\sc^m}_p:= \tilde g^{\ss^{2m-1}}_p\circ \cdots \circ \tilde g^{\ss^{m}}_p,\quad \text{with  }\ss^j:= \sigma^{j}(\ss)\text{ for every }j\ge 0\, .\end{equation}
Note that when $\tilde g= g$, it holds $\tilde g^{\sb_m }_p:= x\mapsto g^{\sb_m}(x-v)+v$ where $g^{\sb_m}$ is the element of the group spanned by $(g^\sb)_{\sb\in \sB_0^\lambda}$
and $\tilde g^{\sc^m}_p=(g^{\sc}_p)^m$ for every $p\in W$.

Hence as $W$ is small, $\tilde g^{\sb_m }_p$ is a composition of contractions of $[-2,2]$ into its interior. Let 
$\tilde x_p(\sb_m ):= \tilde g^{\sb_m }_p(0)$.  On the other hand, by Fact \ref{simple setting}, the map  $(\tilde g_ w^{\sc^m})$ fixes the point $0$. Let $\tilde \lambda_p(\sb_m )$ and $\tilde \lambda_p(\sc^m)$ be the Lyapunov exponents of these fixed points:
\begin{equation}\label{def lambda cm} \tilde \lambda_p(\sb_m )= \frac1m \log |D_{0}\tilde g^{\sb_m }_p|\qand 
\tilde \lambda_p(\sc^m)=\frac1m \log |D_{0}\tilde g^{\sc^m}_p|\end{equation}
 Here is the generalization of \cref{Importantclaimgroup}:
\begin{proposition}\label{Importantclaimskewproduct}
For $\cal M_W$ small enough, for every $(\tilde f_p)_p\in \cal M_W$, if there exist $(\sb_m \cdot \sc^m)^\infty\in \sP_m$ with $m$ large and $p\in W$ small  such that: 
\[\tilde x_p(\sb_m )=0 \qand  \tilde \lambda_p(\sc^m )+\tilde \lambda_p(\sb_m )=0,\]
then  the map $ \breve h:= \tilde g_p^{\sc^m}\circ \tilde g_p^{\sb_m } $ is parabolic. 
 \end{proposition}

The proof is very similar to \cref{Importantclaimgroup} and done below.  Hence by  \cref{Importantclaimskewproduct,Densiteparabolic},   Theorem \ref{theoA} is a consequence of the next Claim.\end{proof}
\begin{claim} \label{le claim a prouver0}
For $\cal M_W$ $C^R$-small enough, for every $(f_p)_p\in \cal M_W$, there exist 
$p\in W$ small and $(\sb_m \cdot \sc^m)^\infty\in \sP_m$ with $m$ large  such that: 
\[\tilde x_p(\sb_m )=0 \qand  \tilde \lambda_p(\sc^m)+\tilde \lambda_p(\sb_m )=0.\]
\end{claim}
\begin{proof}[Proof of the Claim] 
First let us show that we can restrict our proof to the case where $\I=\N$. Indeed if  $\I=\Z$, by \cref{to monoinfinite shift}, for every $(\tilde f_p)\in \cal M_W$, for every $\ss\in \Sigma$, there exists a $C^{R-1}$-family  $(\mathsf {hol}^\ss_p)_p$ depending continuous on $\ss\in \Sigma$ such that:
\[\mathsf {hol}^{\sigma(\ss)}_p\circ \tilde g^\ss_p=\tilde g^{\iota \circ \rho (\ss)}_p\circ \mathsf {hol}^\ss_p\, ,\] 
with $\rho: \sA^\Z\to \sA^\N$ is the canonical retraction and  $\iota: \sA^\N\to \sA^\Z $ is the section which sends $(\sa_i)_{i\ge 0}$ to  $(\sa_i)_{i\in \Z}$ with $\sa_j=\sc$ for every $j<0$.  
As $(\tilde g^{\iota \circ \rho (\ss)}_p)_p$ is $C^{R}$ close to $(g^{\iota \circ \rho (\ss)}_p)_p=(g^{\ss}_p)_p$, the family $(\tilde g_p)_p$ is $C^{R-1}$-conjugated to
a family of endomorphisms of $\sA^\N\times P^1(\R)$ which is $C^R$-close to the canonical family of endomorphisms associated to $(g^\sa_p)_p$. Note that this is a $C^R$ bound although the conjugacy is $C^{R-1}$. As the conclusion of  the Claim is   invariant by $C^1$-intrinsic conjugacy and $R= 2$,   it suffices to show the following:
\begin{claim}\label{le claim a prouver}
Let $\Sigma= \sA^\N$. For every $(\tilde g_p)_p$ in a $C^R$-small neighborhood $\cal N_W$ of $( g_p)_p$, there exist 
$p\in W$ arbitrarily  small and $(\sb_m \cdot \sc^m)^\infty\in \sP_m$ with $m$  large such that: 
\[\tilde x_p(\sb_m )=0 \qand  \tilde \lambda_p(\sc^m)+\tilde \lambda_p(\sb_m)=0.\]
\end{claim}

Let $\sc^\infty=\sc\sc \sc \cdots   \in \Sigma$ be the $\sigma$-fixed point whose $\sA$-spelling has only the letter $\sc$.
Given  $(\tilde g_p)_p\in \cal N_W$, by Fact \ref{simple setting}, the point $0$ is fixed and repulsive for $\tilde g^{\sc^\infty}_p$;  let $\tilde \lambda_p(\sc^\infty):=
\log| D_0 g_p^{\sc^\infty} |$. 
\begin{fact}   When $\cal N_W$ is small,  % $(\tilde x_p(\sc^\infty))_p$ is $C^R$-close to $(v)_{w=(u,v)}$ and
 $(\tilde \lambda_p(\sc^\infty))_p$ is $C^{R-1}$-close to  $( \log\frac32 +u)_{p=(u,v)}$.\end{fact}

%Let $x_p:= (c^\infty, x_p(\sc^\infty))$ and $\tilde x_p:= (c^\infty, \tilde x_p(\sc^\infty))$ which are points in $\xigma\times P^1(\R)$. 

Let $V_\Lambda\subset \Sigma\times P^1(\R)$ be  formed by the points  $(\ss,x)\in \Sigma\times (-2,2)$ such that the  $0$-coordinate of $\ss$ is in $\sB_0^\lambda$. By \cref{skewprod lambda blend}, the $g$-maximal invariant set of $V_\Lambda$ is a $\lambda$-blender $\Lambda$, and the Lyapunov fibration of $ W^u_{loc} (\Lambda) = \Sigma\times [-1,1] $ contains $(c^\infty, 0  , \log 2/3)$ in its interior.  Thus for $\cal N_W$ sufficiently small, for every  $(\tilde g_p)_p \in \cal N_W$,  there is a $\tilde g_0$-preorbit $\arr v=(v_j)_{j\le 0} \in \arr V_\lambda (\tilde g)$ 
 such that:
\begin{equation} \label{application blender prop} v_0 = (c^\infty, 0)%\tilde x_0 
\qand 
\lambda_0(\tilde g_0, \arr v)+ \tilde \lambda_0 (\sc^\infty)=0\; .\end{equation}
The $\Sigma$-coordinate of $v_m$ is of the form 
$(\sa_{-m},..., \sa_{-1}, \sc, ..., \sc, ...)\in \Sigma$, with  $\sb_m:=\sa_{-m}\cdots \sa_{-1}\in \sB_0^\lambda)^m$. Fix $(\sb_m\cdot \sc^m)^\infty  \in \cal P_m$ for a large $m$.
\begin{lemma}\label{suite choisie}
It holds:
\[  \lim_{m\to \infty} \tilde x_0(\sb_m)=0
% \tilde x_0(\sc^\infty)=\lim_{m\to \infty} \tilde x_0(\sc^m)
\qand \lim_{m\to \infty} -\tilde \lambda_0(\sb_m)=
\tilde \lambda_0(\sc^\infty)=\lim_{m\to \infty} \tilde \lambda_0(\sc^m )\; .\]
\end{lemma}
\begin{proof} 
Most of maps in the composition 
 $\tilde g_0^{\sc\cdot 
(\sb_m\cdot \sc^m)^\infty}\circ \cdots \circ \tilde g_0^{\sc^m \cdot 
(\sb_m\cdot \sc^m)^\infty}$ are close to $\tilde g_0^{\sc ^\infty}$ when $m$ is large, so are the derivatives at $0$. In view of \cref{rceil rfloor0,def lambda cm}, this proves: 
\[\lim_{m\to \infty} \tilde \lambda_0(\sc^m )=\tilde \lambda_0(\sc^\infty)\, .\] 
The map $\tilde g_0^{\sb_m} $ is the composition  $\tilde g_0^{\sb_1\cdot 
(\sc^m\cdot \sb_m)^\infty}\circ \cdots \circ \tilde g_0^{\sb_m\cdot 
(\sc^m\cdot \sb_m)^\infty}$, with $\sb_j:= \sa_{-j}\cdots \sa_{-1}$. 
Each of the maps $\tilde g_0^{\sb_j\cdot 
(\sc^m\cdot \sb_m)^\infty}$ is a contraction of $[-2,2]$  which is uniformly close to $\tilde g_0^{\sb_j\cdot 
\sc ^\infty}$ when $m$ is large. 
Thus the iterates of  $0$ under these respective maps are uniformly close when $m$ is large. This proves that 
$\tilde x_0(\sb_m)$ is close to  $\tilde g_0^{\sb_1\cdot \sc^ \infty} \circ \cdots \circ \tilde g_0^{\sb_m\cdot 
 \sc^\infty} (0)$ when $m$ is large. Also as $\tilde g_0^{\sb_1\cdot \sc^ \infty} \circ \cdots \circ g_0^{\sb_m\cdot 
 \sc^\infty} $ is a large contraction of $(-2,2)$ 
which sends  the  $x$-coordinate $x_m\in  (-2,2)$ of $v_{m}$ to the 
% in which the  $x$-coordinate $x_m$ of $v_{m}$ belongs, it comes that $\tilde g_0^{\sb_1\cdot \sc^ \infty} \circ \cdots \circ \tilde g_0^{ \sb_m\cdot  \sc^\infty} (0)$ is close
 to the $x$-coordinate of $v_0$ which is $0$  by \cref{application blender prop}. So $\lim_{m\to \infty} \tilde x_0(\sb_m)=0$ by \cref{application blender prop}.
 Note that this argument implies that most of the iterates of $x_m$ and $0$  under $\tilde g_0^{\sb_1\cdot \sc^ \infty} \circ \cdots \circ \tilde g_0^{ \sb_m\cdot 
 \sc^\infty}$ are close, and these are  close to the iterates of $0$ under $\tilde g_0^{\sb_1\cdot 
(\sc^m\cdot \sb_m)^\infty}\circ \cdots \circ \tilde g_0^{\sb_m\cdot 
(\sc^m\cdot \sb_m)^\infty}$. So are the derivatives at these points. This proves  $\lim_{m\to \infty} \tilde \lambda_0(\sb_m)=\lambda_0(\tilde g_0, \arr v)$ which is equal to $-\tilde \lambda_0(\sc^\infty)$ by \cref{application blender prop}.
\end{proof}
The later  statement and proof  were done at the parameter $p=0$ but can be done at any $p\in W$. We also obtain the convergence of the  parametric jets  for the same reasons. This gives:
\begin{lemma}\label{suite choisie para}   For every $p_0\in W$,  with $\tilde G^m= (\tilde G_p^m)_p$ and $\tilde G_p^m:=\tilde g_p^{\sb_1\cdot \sc^ \infty}\circ \cdots \circ \tilde g_p^{\sb_m\cdot \sc^ \infty}  $, it holds:
\[   J^R_{p_0}(\tilde x_p(\sb_m))_p\underset{m\to \infty}{\approx} J^R_{p_0}\tilde G ^m(0) 
% \qand J^R_{p_0}(\tilde x_p(\sc^\infty))_p=\lim_{m\to \infty} J^R_{p_0}(\tilde x_p(\sc^m))_p
\; .\]
\[ J^{R-1}_{p_0}\tilde \lambda_0(\sb_m)\underset{m\to \infty}{\approx} J^{R-1}_{p_0}\frac1m \log |D_0\tilde G ^m|
\qand 
J^{R-1}_{p_0}\tilde \lambda(\sc^\infty)=\lim_{m\to \infty} J^{R-1}_{p_0}\tilde \lambda(\sc^m )\; .\]
\end{lemma}
  
We recall that $R=2$. 
Let $\underline \ss= (\sb_j\cdot \sc^ \infty)_{j<0}\in \arr \Sigma'$. 
By definition $J^R_{p_0}  \tilde X (\underline \ss) $ and $\lambda(\tilde g, J^R_{p_0} \tilde X(\underline \ss) )$ are the limit when $m\to \infty$ of  
$J ^R_{p_0} \tilde G^m(0) $ and $  J^{R-1}_{p_0}  \frac1m \log |D_0\tilde G^m| $. Hence by the latter lemma, the derivative 
$ \partial_p (\tilde x_p(\sb_m))_p$ is close to  $\partial_p  G ^m(0)$
and $\partial_p \tilde \lambda_0(\sb_m)$ is close to $\partial_p \lambda(\tilde g, J^R_{p_0}  X(\underline \ss))$. 
As   $G ^m_{p}(0)= \lim_{j\to \infty} g^{\sb_j \cdots \sb_{-1}}(-v)+v=\lim_{j\to \infty} g^{\sb_j \cdots \sb_{-1}}(0)+v$, its derivative is close to the second coordinate projection $p=(u,v)\mapsto v$. Also  $\partial_p \lambda(\tilde g, J^R_{p_0} X(\underline \ss))$ is zero. Thus 
$\partial_p  ( \tilde \lambda(\sb_m) , \tilde x(\sb_m) )$ is close to the first coordinate projection. 
%
% $J^r_{p_0} \tilde X(\underline \ss)$ and
%$J^{r-1}_{p_0} \lambda(\tilde g, J^r_{p_0} \tilde X(\underline \ss))$ are close to $J^r_{p_0} X(\underline \ss)$ and
%$J^{r-1}_{p_0} \lambda(g, J^r_{p_0}  X(\underline \ss))$ for every $p_0$. As $X_{u,v}(\underline \ss)=\lim_{j\to \infty} g^{\sb_j \cdot \sb_{-1}}(-v)+v=\lim_{j\to \infty} g^{\sb_j \cdot \sb_{-1}}(0)+v$, it comes that   $\partial_p   \tilde X(\underline \ss)$ is close to the first coordinate projection  $p=(u, v)\mapsto v$. 
%Also $\partial_ p J^{r-1}_{p_0} \lambda(g, J^r_{p_0}  X(\underline \ss))=0$, so $\partial   \lambda(\tilde g, J^r_{p_0}  X(\underline \ss))$ is small. 
On the other hand, 
$\partial_p \tilde \lambda(\sc^m) $ is close to $\partial_p \tilde \lambda(\sc^\infty) $ which is close to $\partial_p   \lambda(\sc^\infty)=(u,v)\mapsto u $. 
%By the definition of the canonical family associated to $((g^\sa_p)_p)_\sa$, when $\cal N_W$ is small the families $(\tilde g_p^{\sb_1\cdot \sc^ \infty})_p$ , ..., $(\tilde g_p^{\sb_m\cdot \sc^ \infty})_p $  have their $  J^R$-jets  close to the one of the translated families  $(g^{\sa_{-1}}(\cdot -v)+v)_{p= (u,v)}$, ..., $( g^{\sa_{-n}}(\cdot -v)+v)_{p= (u,v)}$ while ${p= (u,v)}\mapsto ( \tilde \lambda_p(\sc^\infty),0)$ has its derivative close to the first coordinate projection. 
Thus the derivative of the  following function is  close to the identity: 
\[\Upsilon_m : p=(u,v)\in W \mapsto  (\tilde \lambda_p(\sc^m)+\tilde \lambda_p(\sb_m) , \tilde x_p(\sb_m) )\; .\]

By \cref{suite choisie},   for $m$ large, the point $\Upsilon_m(0)$ is small.  Thus by the local inversion theorem, for $m$ large enough, there exists $p$ small such that  $\Upsilon_m (p)=0$. This proves Claim \ref{le claim a prouver}.
 \end{proof} 
\begin{proof}[Proof of \cref{Importantclaimskewproduct}] First note that when $(\tilde g_p)_p$ is $C^2$-close to $(g_p)_p$ and $p\in W$ is small, then $\tilde g:= \tilde g_p$ is $C^2$-close to $g=g_0$. 
  We recall that $g^\sc$ is conjugate to $x\mapsto \frac32 x$ via $h(x)= x/(2x+1)$.
For the same reasons as for \cref{Importantclaimgroup}, the following counterpart of Claim \ref{Importantclaimgroup1} implies  \cref{Importantclaimskewproduct}:
\begin{claim} Given $\tilde g$ $C^2$-close to $g$ and $\ss\in \cal P_n$ with $n$ large, if  $ \tilde g ^{\sb_n }(0)=0$ and the fixed point  $0$  is parabolic for $ \breve h:=\tilde g ^{\sc^m}\circ  \tilde g ^{\sb_n }$, then the restrictions to $[-3/2, 3/2]$ of  $ \breve h$ and $h$ are $C^2$-close. 
\end{claim}
% Let $(\tilde O^\ss)_{\ss\in \Sigma}$ be a continuous family of points nearby $0$ such that for every $\ss\in \Sigma$, if the first letter of $\ss$ is not $\sc$ then $\tilde O_\ss=0$, otherwise it holds $\tilde g^\ss(\tilde O^\ss)= (\tilde O^{\sigma(\ss)})$. Note that  $(\tilde O^\ss)_{\ss\in \Sigma}$ is close to $0$ when   $\tilde g$ is $C^2$-close to $g$. Thus using a coordinate change via a translation by $\tilde O_\ss$, we can assume that $\tilde O ^\ss=0$ for every $\ss$.
The following is the counterpart of Sternberg's \cref{Sternerberg}: 
\begin{lemma}\label{lemm1finalClaim}
There exist  $\tilde  h_1, \tilde h_2\in C^2([-2,2],\R)$ which are $C^2$-close to $h|[-2,2]$ and such that
 $\tilde  h(0)=0$, $D_0\tilde  h (0)=1$ and 
\[ \tilde  h_1=   \tilde g ^{\sc^m}\circ   \tilde h_2  \circ (D_0\tilde g^{\sc^m})^{-1}|[-2,2]\; .\] 
\end{lemma}
We show the latter lemma below. On the other hand the proof of the next lemma is skipped since it is the  same as for \cref{lemm2finalClaimgroup}.
\begin{lemma}\label{lemm2finalClaim}
The map $\Phi=
D_0\tilde g^{\sc^m}\circ  \tilde h^{-1}_2\circ 
\tilde g ^{\sb_m }|[-2,2]$ is $C^{2}$-close to the identity. \end{lemma}
By the latter Lemma, the image by $\Phi$ of $[-5/3, 5/3]$ is included in $[-2 ,2 ]$. Thus we can use Lemma \ref{lemm1finalClaim} to obtain that the map $\breve h|[-3/2,3/2]$ is equal to the composition of $\tilde h_1$ which is $C^2$-close to $h$  with a map $\Phi$ $C^2$-close to the identity. This proves the claim.  
\end{proof}
\begin{proof}[Proof of Lemma \ref{lemm1finalClaim}]
Let $B$ be the set of continuous families $(h^{\ss})_{\ss\in \Sigma}$  of $C^2$-diffeomorphisms $h^{\ss}$ of $[-2,2]$ into $\R$ such that 
$h^{\ss}(0)=0$ and $D_0h^{\ss}=1$. 
We notice that $B$ endowed with $C^2$-distance is complete. 
We consider the operator:
\[\Psi\colon H=(h^{\ss})_{\ss\in \Sigma} \in B \mapsto 
(\Psi(H)^\ss)_{\ss}\in B,\quad \text{with}
\;  \Psi(H)^{\ss} =
\tilde g^{\sc\cdot  \ss }\circ h^{\sc\cdot  \ss }\circ (D_0\tilde g^{\sc\cdot  \ss } )^{-1}
\, ,\]
where $\{\sc\}\times \sA^\N$ and  $\sB_0^\lambda\times \sA^\N$  are the subsets of $\sA^\N$ formed by points  whose first letters are  in respectively $\{\sc\}$ and   $\sB_0^\lambda$.
The map $\Psi$ has a contracting iterate;  let $(\tilde h^\ss)_\ss$   be its fixed point (which depends continuously on $\tilde g $). It satisfies:
\[\tilde h^\ss= 
\tilde g^{\sc\cdot  \ss }\circ \tilde h^{\sc\cdot  \ss }\circ (D_0\tilde g^{\sc\cdot  \ss } )^{-1}=
%\tilde g^{\sc\cdot  s }\circ\tilde g^{\sc^2\cdot  s }\circ h^{\sc^2\cdot  s }\circ (D_0\tilde g^{\sc\cdot  s }\circ\tilde g^{\sc^2\cdot  s } )^{-1}=
\tilde g^{\sc\cdot  \ss }\circ\cdots \circ \tilde g^{\sc^m\cdot  \ss }\circ \tilde  h^{\sc^m\cdot  \ss }\circ (D_0\tilde g^{\sc\cdot  \ss }\circ\cdots \circ D \tilde g^{\sc^m\cdot  \ss } )^{-1}\; .
\]
Thus with $\ss= ( \sb_m\cdot \sc^m)^\infty$ and $\tilde g^{\sc^m}$ defined in \cref{rceil rfloor0} \cpageref{rceil rfloor0}, we obtain:
\[\tilde h^{( \sb_m\cdot \sc^m)^\infty}= 
\tilde g^{\sc^m}\circ \tilde h^{( \sc^m\cdot \sb_m)^\infty}\circ (D_0\tilde g^{\sc^m}  )^{-1}\]
We conclude by putting $\tilde h_1=\tilde h^{( \sb_m\cdot \sc^m)^\infty}$ and $\tilde h_2=\tilde h^{( \sc^m\cdot \sb_m)^\infty}$.
\end{proof}

Let us assume that $M$ is an analytic manifold of dimension $\ge 2$. Let us choose a complex extension  $\tilde M$ of $M$. 
For $\epsilon>0$, let   $\tilde M_\epsilon$ be the $\epsilon$-neighborhood of $M$ in $\tilde M$. Let $C^\omega_\epsilon(M,M)$ be the space of analytic maps of $M$ in $M$ whose analytic extension is well defined on from $\tilde M_\epsilon$ into $\tilde M$ and $C^0$-bounded. This space is endowed with the uniform $C^0$-norm. The following is a consequence of the latter proof:
\begin{corollary}\label{analytic}
For $\epsilon>0$ sufficiently small, there are locally dense sets
 $D_p$ and $D_r$ in $C^\omega_\epsilon(M,M)$ formed by maps which display a normally hyperbolic smooth circle on which they act respectively as parabolic maps and Diophantine smooth rotations.  Moreover $D_p$ and $D_r$ are formed by diffeomorphisms if $n\ge 3$. 
\end{corollary}
\begin{proof}
The set of $2$-parameter $C^\omega_\epsilon$-families intersected with $\cal M_W$ is an open set $\cal M_{W, \epsilon}^\omega$. Also the set 
 $\cal M_{\epsilon}^\omega= \{\tilde f_0: (\tilde f_p)_p\in \cal M_{W, \epsilon}^\omega\}$ is open. For every $\tilde f_0\in \cal M_{\epsilon}^\omega$, by  \cref{Importantclaimskewproduct} and Claim \ref{le claim a prouver0}, there exist $p$ small and  $m\ge 1$  such that $\tilde f_p^{2m}$ restricted to the   fiber of $(\sb_m, \sc^m)^\infty$ is a  parabolic map $\breve h$. 
This implies the  density in $\cal M_{\epsilon}^\omega$ of analytic dynamics leaving invariant normally hyperbolic, periodic circles at which they are parabolic.
 
Now take  $p'=p+(0,v')$ with $v'>0$ small. Then the restriction $\breve h_{v'}$ of  $\tilde f^{2m}_{p'}$ restricted to the fiber of $(\sb_m, \sc^m)^\infty$ 
satisfies $x\le \breve h(x)< \breve h_{v'}(x)$ for every $x\in P^1(\R)$, so  $\breve h_{v'}$ has none  fixed points and so its rotation number differs to the one of $\breve h$ which is $0$. Thus by continuity of the rotation number, there exists $v'$ small such that the rotation number of  $\breve h_{v'}$  is Diophantine. This implies the  density in $\cal M_{\epsilon}^\omega$ of analytic dynamics leaving invariant normally hyperbolic, periodic circles at which they are Diophantine rotation.
%Finally note that these circle are $\infty$-normally hyperbolic and so by \cite[Forced smoothness]{HPS} these circles are of class $C^\infty$. 
\end{proof}
%\begin{remark}As the set $\cal M_W$ contains an open set of families which are analytic, \cref{Importantclaimskewproduct} and Claim \ref{le claim a prouver0} imply the local density of analytic dynamics leaving invariant normally hyperbolic, periodic circles at which they are parabolic.
%\end{remark}
Then a proof of the following problem would imply the local density of fast growth of the number of periodic points among analytic maps (for the inductive topology) by using \cref{analytic}  and then Grauert and Cartan B theorems.
\begin{problem}\label{analytic2} Let $f\in C^\omega (M,M)$  and let $C$ be a $C^\infty$-circle included in $M$ at which $f$ is  normally hyperbolic and such that the rotation number of $f|C$ is Diophantine. Show that $C$ is analytic. 
\end{problem}

\subsection{Theorem \ref{theoB}}\label{sec:Theorem  theoB}
The proof  of  this Theorem \ref{theoB} has a similar  structure as the one of Theorem \ref{theoA}, using   \cref{ThmS3,ThmS3coro} instead of \cref{Densiteparabolic} and  $C^r$-$\lambda$-parablenders instead of $\lambda$-blenders.
 However the perturbations will be more tricky to perform. 

\begin{proof}[Proof of Theorem \ref{theoB}]
Let us show  the existence of a locally dense subset $\hat D$ satisfying the assumptions of \cref{ThmS3,ThmS3coro}, and which is formed by families of diffeomorphisms when $n=\dim M\ge 3$. %However the set $\hat D$ is infinite codimensional and so we cannot produce the perturbation inside a family with a finite number of extra parameters. 

Let $\sB^\lambda_r$ and $(g^\sb_p)_{\sb \in \sB^\lambda_r}$ be as  in \cref{lambda parablender IFS} for every $p\in B_k$. Let $g_p^\sc: x\mapsto \frac32 \frac x{x+1}$ for every $p\in B_k$.   We consider  these  maps  as acting on the compactification $P^1(\R)$ of $\R$. Let $\sA= \sB^\lambda_r\cup \{\sc\}$ for a symbol $\sc$.
  By \cref{embedding normally hyp} for any $\infty>r\ge 1$, there is a $C^\infty$-embedding  $j$ of   $\Sigma\times P^1(\R)$ into the manifold $P^1(\R)\times (-1,1)^{n-1}\subset  M$ and a $C^\infty$-family $(f_p)_{p\in B_k}$ of $C^\infty$-self-maps $f_p$ of $M$  (which are diffeomorphisms if $n\ge 3$) which are $(r+2)$-normally expanding  (or hyperbolic if $n\ge 3$) at $\cal L:= j (\Sigma\times P^1(\R))$ and such that $g_p=j ^{-1}\circ f_p\circ j$ is the canonical map of $\End^{\infty}_\sigma (\Sigma\times P^1(\R))$ associated to $(g^\sA_p)_{\sa\in \sA}$.

Let $\hat {\cal M}$ be a small $C^r$-open neighborhood of $(f_{p})_{p}$ intersected with the subset of $C^\infty$-families. Given $(\tilde f_{p})_p\in \hat {\cal M}$, let $(\tilde j_{p})_{p}$ be the family of embeddings given by \cref{coroHPS2} and let $\tilde g_{p}:= \tilde j_{p}^{-1}\circ \tilde f_{p} \circ \tilde j_{p}\in \End^r_\sigma (\Sigma\times P^1(\R))$ and $\tilde {\cal L}_p:= \tilde j_p (\Sigma\times P^1(\R))$.  The assumptions of \cref{ThmS3} are implied by the following shown below.
 
\begin{claim}\label{ParaImportantclaim} There is a neighborhood $B_k'$ of $0\in B_k$ such that for any  $\eta>0$ and $(\mathring f_{p})_{p} \in \hat {\cal M}$,  
there is $\eta'>0$ satisfying the following property. For every $p_0\in B'_k$, there exists $m\ge 1$,  a $2m$-periodic $\ss\in \Sigma$  
and   $(\tilde  f_{p})_{p}\in \hat {\cal M}$ which is  $C^r$-$\eta$-close to  $(\mathring f_{p})_{p}$, satisfies  $\tilde  f_{p}=\mathring  f_{p}$ for every   $p\notin ([-\eta',\eta']^k+ p_0) $,
% and there exists a $2m$-periodic $\ss\in \Sigma$ at which
  and such that for every $p\in  [-\frac23\eta', \frac23\eta']^k+ p_0$, the restriction  $\tilde  f_p^{2m}| \tilde j_p(\{\ss\}\times P^1(\R)) $ displays a parabolic point depending continuously on $p$ and if $r\ge 2$ this circle diffeomoprhism is parabolic.
%      if $r\ge 1$ or displays a parabolic point depending continuously on $p$ if $r=1$.
\end{claim}

For every  $\u n_0\in \{0,1\}^k$ and  $\eta>0$, 
let $\eta'>0$ and let $(\tilde f_p)_p$ be the family whose restriction to $[-\eta',\eta']^k+ p_0$, for $p_0$ in the finite set $\{\eta'\u  n_0 + 2\eta' \u n : \u n\in  \Z^k\}\cap B_k'$, is given by the above claim applied to $(\mathring f_p)_p$. Note that $(\tilde f_p)_p$ is of class $C^r$ and  $\eta$-distant from $(\mathring f_p)_p$. 
Furthermore it satisfies:
%given by the above claim 
%we apply the above claim for $p_0$ in the finite set $\{\eta'\u  n_0 + 2\eta' \u n : \u n\in  \Z^k\}\cap B_k'$. This provide families which differ from $(f_{p})_p$ only for 
%$p\notin ([-\eta',\eta']^k+ p_0) $. So we can merge them to obtain 
%% We merge the family given by this Claim to obtain 
% the existence of 
%%This Claim  implies that for any $\u n_0\in \{0,1\}^k$ and  $\eta>0$, there exist
% $\eta'>0$ and  an $\eta$-dense subset of $(\tilde f_{p})_{p} \in \hat {\cal M}$ satisfying the following property.\medskip

\begin{enumerate} \item[$\mathcal P(\u n_0)$]
For every $\u n\in \Z^k$,  there exists a $2m$-periodic point $\ss(\u n)\in \Sigma$ such that for $p\in \{\eta' (\u  n_0 + 2\u n) +(- \frac23 \eta', \frac23 \eta')^k\}\cap B_k'$,  the restriction of $\tilde f^{2m}_p$ to the fiber $ \tilde j_p(\{\ss(\u n)\}\times P^1(\R)) $   displays a parabolic point depending continuously on $p$ and if $r\ge 2$ this circle diffeomorphism is parabolic. \end{enumerate} 

Now assume we embedded $2^k=\Card \{0,1\}^k$ such normally hyberbolic fibrations at different places of $M$   and label each by distinct  numbers $\u n_0\in \{0,1\}^k$.  Then we apply  Claim \ref{ParaImportantclaim} to each of these fibrations.   This provides $2^k$-different perturbations at uniformly distant subsets of $M$. Thus there is a  $C^r$-perturbation of the family of dynamics on the whole phase space whose restriction to the normally hyperbolic fibration indexed by $\u n_0$ satisfies $\mathcal P(\u n_0)$,  for every $\u n_0\in \{0,1\}^k$. As:
\[ \bigcup_{\u n_0\in   \{0,1\}^k, \u n \in \Z^k}  \left\{\eta' (\u n_0 + 2\u n) +\left(- \frac23 \eta', \frac23 \eta'\right)^k\right\}= \bigcup_{\u n \in \Z^k} \left \{\eta' \u n +\left(- \frac23 \eta', \frac23 \eta'\right)^k\right\}=\R^k\; ,\] 
we have shown the existence of a $C^r$-locally dense set $\hat D_0$ of smooth  families $(\tilde f_p)_{p\in B_k'}$ satisfying the assumption of \cref{ThmS3} with $B_k'$ instead of $B_k$. To get $B_k$ instead of $B_k'$, we regard $\hat D:=\{(\tilde f_{\tau\cdot  p})_{p\in B_k}: (\tilde f_{p})_{p\in B_k}\in \hat  D_0\}$ with $\tau>0$  so that $B_k'\supset \tau\cdot B_k$. 
\end{proof} 

\begin{proof}[Proof of Claim \ref{ParaImportantclaim}]
First recall that we can assume that $M=  P^1(\R)\times (-1,1)^{n-1}$. Also:
\begin{fact} \label{new coord}
Up to a $C^{r+1}$-family of coordinates change close to the identity, we can assume that for any  $p\in B_k$,  $\tilde j_p( \Sigma\times \{0\})\subset \{0\}\times  (-1,1)^{n-1}$ and 
if the $0$ letter of $\ss \in \Sigma$ is $\sc$, the map $\tilde g_p^\ss$ fixes~$0$.  \end{fact}
\begin{proof} 
Let $R_p$ be the source of the fiber of $\tilde {\cal L}_p$ based at $\sc^\infty$. Using a coordinate change close to the identity we can assume that it is in 
$\{0\}\times (-1,1)^{n-1}$. 
 Also its strong unstable manifold $W^{uu}_{loc} (R_p, \tilde f_p)$ is of class $C^{\infty}$ and depends $C^\infty$ on $p$.  By \cref{pour variete Wssuu}, we can assume it one-dimensional and that the local  strong stable manifold of its points form a $C^{r+1}$  submanifold $N_p$ which depends $C^{r+1}$ on the parameter $p$. This family of submanifolds is $C^{r}$ close to the constant family $\{0\}\times U$ for   an open subset $U$  of $(-1,1)^d$ and intersects transversally each fiber of $\tilde {\cal L_p}$ at a point close to $0$. Moreover it is locally invariant:
  $N_p\cap \tilde f_p^{-1}(N_p)%\cap \tilde f_p(N_p)
$ is an open set of $N_p$.

We achieve the proof using a coordinate change $C^{r+1}$-close to the identity of $M$ sending  $N_p$ into  $\{0\}\times  (-1,1)^{n-1}$ and changing  the coordinate of the fibers so that $N_p$ intersects them at $0$.
\end{proof}

Similarly to the proof of  Theorem \ref{theoA}, let $\sP_m$ be the set of  $2m$-periodic points of $\Sigma$ whose $\sA$-spelling is of the form $\ss=(\sb_m \cdot \sc^m)^\infty$: its  letters at the positions $0, ..., m-1$ form a word $\sb_m\in (\sB_r^\lambda)^m$ and its $m$ next letters are $\sc$.  Then we denote:
\begin{equation}\label{rceil rfloor}
    \tilde g^{\sb_m }_{p}:=\tilde g^{\ss^{m-1}}_{p}\circ \cdots \circ \tilde g^{\ss^0}_{p}\qand \tilde g^{\sc^m}_{p}:= \tilde g^{\ss^{2m-1}}_{p}\circ \cdots \circ \tilde g^{\ss^{m}}_{p},\quad \text{with  }\ss^j:= \sigma^{j}(\ss)\text{ for every }j\ge 0\, .\end{equation}
Note that when $(\tilde g_{p})_{p}= (g_{p})_{p}$, it holds 
$  \tilde g^{\sc^m}_{p} =(g^{\sc}_p)^m$ for every $p\in B_k$ and $\tilde g^{\sb_m }_{p}= g^{\sb_m}_p$ where $g^{\sb_m}_p$ is the element of the semi-group spanned by $(g_p^\sb)_{\sb\in \sB_r^\lambda}$. 
Using Fact \ref{new coord}, we have $  \tilde g^{\sc^m}_{p}(0)=0$. Let:
\[\tilde x_{p}(\sb_m )=  \tilde g^{\sb^m}_{p}(0)\; ,\quad \tilde \lambda_{p}(\sb_m )= \frac1m \log |D_{0}\tilde g^{\sb_m }_{p}|\qand 
\tilde \lambda_{p}(\sc^m)=\frac1m \log |D_{0}\tilde g^{\sc^m}_{p}|\; .\]
We notice that $\tilde x_{p}(\sb_m )$ belongs to $(-2,2)$, also $\tilde \lambda_{p}(\sc^m)\approx \log \frac32$ and $ \tilde \lambda_{p}(\sb_m )<0$.  Note that if 
$\tilde x_{p}(\sb_m )=0$, then it is the unique point of  $\tilde g^{\sb^m}_{p}$ in $(-2,2)$. 
%For $p$ sufficiently small and $(\tilde g_p)_p$ sufficiently close to $(g_p)_p$, 
%we recall that $\tilde g^{\sb_m }_{p}$ is a composition of contractions of $[-2,2]$ into its interior, and so it displays a unique fixed point 
%$\tilde x_{p}(\sb_m )\in (-2,2)$. Likewise $(\tilde g_ w^{\sc^m})^{-1}$ is a composition of contractions of $[1/4,1/4]$ into its interior, and so it displays a unique fixed point 
%$\tilde x_{p}(\sc^m)\in (-1/4,1/4)$. Let $\tilde \lambda_{p}(\sb_m )$ and $\tilde \lambda_{p}(\sc^m)$ be the Lyapunov exponent of these fixed points:
%\[\tilde \lambda_{p}(\sb_m )= \frac1m \log |D_{\tilde x(\sb_n )}\tilde g^{\sb_m }_{p}|\qand 
%\tilde \lambda_{p}(\sc^m)=\frac1m \log |D_{\tilde x(\sc^m)}\tilde g^{\sc^m}_{p}|\]
The proof of the following is the same as the one of   \cref{Importantclaimskewproduct}:
\begin{proposition}\label{Importantclaimskewproductpara}
Let $r\ge 2$. For $\hat {\cal M}$ and $B_k'$ small enough, for every $(\sb_m \cdot \sc^m)^\infty\in \sP_m$ with $m$ large, for every $(\tilde f_{p})_{p}\in \hat {\cal M}$ and $p\in B'_k$, the map $ \breve h_p:= \tilde g_{p}^{\sc^m}\circ \tilde g_{p}^{\sb_m } $
  is parabolic if: 
\[ \tilde x_{p}(\sb_m )=0 \qand   \tilde \lambda_{p}(\sc^m )+ \tilde \lambda_{p}(\sb_m )=0\, .\]
 \end{proposition}
Hence,  Claim \ref{ParaImportantclaim}   is a consequence of the next Claim for every $r\ge 1$.
\end{proof}
\begin{claim}\label{ParaImportantclaim2} For $\hat {\cal M}$ and $B_k'$ small enough,
for all $\eta>0$ and $(\mathring f_{p})_{p} \in \hat {\cal M}$,  
there exists $\eta'>0$ s.t.:

For every $p_0\in B'_k$, there exists  $(\tilde  f_{p})_{p}\in \hat {\cal M}$ which is  $C^r$-$\eta$-close to  $(\mathring f_{p})_{p}$, satisfies  
$\tilde  f_{p}=\mathring  f_{p}$ for every   $p\notin ([-\eta',\eta']^k+ p_0) $, and there exists  $(\sb_m \cdot \sc^m)^\infty\in \sP_m$ with $m$ large at which: 
\[ \tilde x_{p}(\sb_m )=0 \qand  \tilde \lambda_{p}(\sc^m)+\tilde \lambda_{p}(\sb_m )=0\quad \forall p\in  \left[-\frac23\eta', \frac23\eta'\right]^k+ p_0.\]
\end{claim}
 \begin{proof} The following applies  the $C^r$-$\lambda$-parablender property; it assumes  $\hat {\cal M}$ and $B_k'$ small enough:
 % but  we will use the $\lambda$-$C^r$-parablender property instead of the $\lambda$-blender property, and secondly the final perturbation will be done using an operator of Banach spaces.  The following assume $\hat {\cal M}$ and $B_k'$ small enough.
%fixes $(\sb_m\cdot \sc^m)^\infty$ in function of $(\tilde f_p)_p\in \hat {\cal M}$, $p_0\in B_k'$ and $m\ge 1$:
\begin{lemma}\label{approx jet}  For every $p_0\in B_k'$ and 
$(\tilde f_p)_p\in \hat {\cal M}$, 
when $m$ is large, there exists  $ (\sb_m\cdot \sc^m)^\infty \in \sP_m$ such that  $J^r(\tilde x_{p}(\sb_m ))_p$ is small
 and 
$J^{r-1} _{p_0}( \tilde \lambda_{p}(\sc^m))_p+ J^{r-1} _{p_0}( \tilde \lambda_{p}(\sb_m ))_p$ is small. 
\end{lemma}
\begin{proof} Assume that $\Sigma = \sA^\N$.  We recall that $\sc^\infty\in \Sigma$ denotes the point whose $\sA$ spelling is formed uniquely of the letter $\sc$. We recall that $0$ is the repulsive fixed point of $g^{\sc^\infty}_p$, let  $\tilde \lambda_p(\sc^\infty)$ be its Lyapunov exponent. For $\hat {\cal M}$ and $B_k'$ small enough, it holds that  $(-\tilde \lambda_p(\sc^\infty))_p$ is  $C^{r-1}$-close to  $(\log 2/3)_p$.  Then by  \cref{exem lambda parablender skewproduct}, the jet $(\sc^\infty, 0, -J_{p_0}^{r-1}(\tilde \lambda_p(\sc^\infty))_p)$  is included  in the Lyapunov fibration of the unstable set of the $\lambda$-$C^r$-parablender equal to the continuation of   $(\sB_r^\lambda)^\N\times [-1,1]$. Thus with $\arr \Sigma'$ the space of $\sigma$-preorbits  with $0$ coordinate in $\sB_r^\lambda$ (see  \cref{exem lambda parablender skewproduct}),  there is $\u \ss=(\ss_{-m})_{m\ge 1}\in \arr \Sigma'$ such that with $v_{-m} =(\ss_{-m}, 0)$:
\[ \lim_{m\to \infty}  J^r_{p_0}(g_{p}^{m}(v_{-m}))_p= (c^\infty, 0 ) 
\qand 
 \lim_{m\to \infty} 
J^{r-1}_{p_0} 
( \frac1{m} \log |\partial_x  g_{p}^{m}(v_{-m})  |)_p=-J^{r-1}_{p_0} ( \tilde \lambda_p (\sc^\infty))_p \; .\]  
Note that  each $\ss_{-m}$ is the concatenation of $\sb_m\in \sB_r^\lambda$ with $\sc^\infty$. %Put $\tilde G_p^m:=\tilde g_p^{\sb_1\cdot \sc^ \infty}\circ \cdots \circ \tilde g_p^{\sb_m\cdot \sc^ \infty}  $.
By  \cref{suite choisie para} (whose statement and argument are still valid in the present setting),  this implies:
\[   \lim_{m\to \infty} J^r_{p_0}(\tilde x_p(\sb_m))_p\underset{m\to \infty}=0
  \qand    \lim_{m\to \infty} J^{r-1}_{p_0}(\tilde \lambda_p(\sc^m ))_p+J^{r-1}_{p_0}(\tilde \lambda_p(\sb_m))_p =0
  \; .\]
If $\Sigma =\sA^\Z$, we observe that the statement of the lemma is invariant by the $C^r$-conjugacy of $P^1(\R)\times \Sigma$ given by  \cref{to monoinfinite shift} which leaves invariant $0$ by  Fact \ref{new coord}. So   we can use the case  $\Sigma = \sA^\N$ to get the case  $\Sigma = \sA^\Z$ as we did in Claim \ref{le claim a prouver0}. 
\end{proof} 
Now we fix $m$ large and $\sb_m$ given by the previous lemma. 
 the idea is to perturb the dynamics so that $\tilde x_{p}(\sb_m )=0$ and  $ \tilde \lambda_{p}(\sc^m )+ \tilde \lambda_{p}(\sb_m )=0$. Contrarily to the proof of Claim \ref{le claim a prouver0}, 
 it is difficult to handle such a perturbation using  new parameters, since we have  only  $C^{r-1}$-bounds on the family $( \tilde \lambda_{p}(\sc^m )+ \tilde \lambda_{p}(\sb_m ))_p$. 
 So the perturbation  technique is going to be extrinsic. 

Recall that  $\ss:= (\sb_m\cdot  \sc^m)^\infty$ and  $ \ss^j := \sigma^j(\ss)$. Note that $\ss^m= (\sc^m\cdot  \sb_m)^\infty$. Let:
\[S_p:= \tilde j_p( \ss^m, 0) \qand    S'_p=\tilde j_p(\ss^m, \tilde x_p(\sb_m)) \; .\]
Observe that $\tilde f_p^{2m}(S_p)=S'_p$ and $J^r_{p_0} ( S'_p)_p$ is close to $J^r_{p_0} ( S_p)_p$. Indeed $\tilde f^m_p(S_p)=\tilde j_p( \ss, 0)$ and so $\tilde f^{2m}_p(S_p)=\tilde j_p( \ss^m, \tilde x_p(\sb_m))=S'_p$. As $J^r_{p_0} (\tilde x_p(\sb_m))_p$ is small, it comes that $J^r_{p_0} ( S'_p)_p$ is close to $J^r_{p_0} ( S_p)_p$. The following lemma makes a perturbation of $(\tilde f_p)_p$ so that $S'_p=S_p$ when $p$ is uniformly close to $p_0$. Then $S_p$ will be a $2m$-periodic point or equivalently $\tilde x_p(\sb_m)=0$ as sought by the first equality of   Claim \ref{ParaImportantclaim2}.
\begin{lemma}
There exists $\eta'>0$ which is independent of $p_0\in B_k'$ such that when $m$ is large, there is an $\eta/2$-$C^r$- perturbation of $(\tilde f_p)_p$ which is supported by $p\in p_0+ [-\eta', \eta']^k$ and such that for every $p\in p_0+[-\frac23 \eta', \frac23 \eta']^k$, the points $S'_p$ and $S_p$ coincide. In other words $\tilde x_p(\sb_m)=0$. Moreover the families  $(\tilde \lambda_p(\sc^m))_p$ and $(\tilde \lambda_p(\sb_m))_p$ are unchanged. 
\end{lemma}
\begin{proof} 
It suffices to notice that the $\tilde {\cal L_p}$-fiber of the point $\tilde  f_p^{2m-1}(S_p)$ is isolated w.r.t. to the fibers of the other points $(\tilde f_p^{k}(S_p))_{0\le k\le 2m-2}$. Indeed  the letters at the $0$ and $1$ position of $\ss^{m-1}$ are in respectively $\sB^\lambda_r$ and $\{\sc\}$ while it is not the case for the other $\ss^k$. 
Thus we can use a  translation along this fiber so that $S'_p$ is sent to $S_p$, and extends this translation at a neighborhood of this fiber  (using a $C^r$-tubular neighborhood).  Then using a bump function we perform  a local perturbation of $(\tilde f_p)_p$ so that this perturbation is supported by the product of  a small neighborhood of $\tilde f_{p_0}^{2m-2}(S_{p_0})$ with $ p_0+ [-\eta', \eta']^k$, while
$\tilde  f_p^{2m-1}(S_p)$ is sent to $S_p$ for every $p\in p_0+[-\frac23 \eta', \frac23 \eta']^k$. Note that such perturbation does not changed the fiber of the orbit of $\ss$. Furthermore, as  $J^r_{p_0} ( S'_p)_p$ is close to $J^r_{p_0} ( S_p)_p$, the family of translation is $C^r$-small, and so is its product with the bump function when $\eta'$ is sufficiently small.  
 
 Actually $\eta'$ does not depend on $m$ large nor $p_0$. Indeed $(S_p)_{p\in B_k}$  and $(\tilde f_p^{m}(S_p))_{p\in B_k}$ belongs to the set of $C^r$-families $(\tilde j_p(\ss', 0))_{p\in B_k}$. These depend continuously among  $\ss'$ in the compact set $\Sigma$. So the modulus of continuity of their $r^{th}$ derivatives is uniform. The same occurs for the  set of families $(\tilde f_p^{j}(S_p))_{p\in B_k}$  among  $m\le j\le 2m $, which is  included in the set of fixed points of  a compact families of fiber contractions. Thus the moduli of continuity of the  $r^{th}$ derivatives of $(S_p)_p$ and $(S_p')_p$ are bounded independently of $p_0\in B_k'$ and  $m$  large. This gives the sought property.   \end{proof} 
Now it remains to perturb the dynamics to make the periodic point $S_p$ semi-parabolic (which is now equivalent to  $\tilde \lambda_{p}(\sc^m)+\tilde \lambda_{p}(\sb_m )=0$). There are two difficulties. First the perturbation should be done in many fibers since we have to perturb the \emph{average} central Lyapunov exponent: this must be done at many fibers.
Secondly this average Lyapunov exponent is only small in the $C^{r-1}$-topology while we have to handle a $C^r$-perturbation. However we have the following:
\begin{lemma}
There is a Lipschitz map which sends $C^{r-1}$ families  $\lambda=(\lambda_p)_{p\in B_k}$ of numbers $\lambda_p$  to $C^r$-families $(h^\lambda_p)_{p\in B_k}$ of self-maps $h^\lambda_p$ of $P^1(\R)$ such that $h^0_p=identity$ and for every $\lambda$,  
\[ h^\lambda_p(0)=0\qand \log |D_0h^\lambda_p|=\lambda_p\; .\]
\end{lemma}
\begin{proof}  
Given a point $(0,p_1)\in \{0\}\times B_k$, the following is a  $C^r$-jet of a real function on  $\R \times B_k$:
\[J^r_{0,p_1} u(\lambda)(x,p)=  x\cdot \sum_{|j|\le r-1}\frac1{(j+1)!} \partial_p^j \exp \lambda_{p_1}\cdot  (p-p_1)^{j+1} -x.\]
By Whitney extension Theorem \cite[Thm 2.3.10]{Ho90}, there is a Lipshitz, linear  operator $u$  which sends $\lambda$ to  a function $u(\lambda)$ from  $P^1(\R)\times B_k$ into $\R$, such that  its $C^r$-jets at  $ \{0\}\times B_k$ matches with the above expression. Then using a bump function we can assume that $u(\lambda)$ is compactly supported. Then put $h_p^\lambda:  x\in P^1(\R)  \mapsto x+  u(\lambda)(x, p)\in P^1(\R)$.
\end{proof}
Recall that we assumed that $M=P^1(\R)\times (-1,1)^{n-1}$ and that the fibers of 
$\tilde {\cal L}_p$   are transverse to the fibers of $P^1(\R)\times (-1,1)^{n-1}\to P^1(\R)$. This does not change the limit of $J^{r-1}_{p_0}(\tilde \lambda_p(\sc^m ))_p+J^{r-1}_{p_0}(\tilde \lambda_p(\sb_m))_p$ when $m\to \infty$ to assume that the norm on the tangent space of $\tilde {\cal L}_p$ is induced by the first coordinate projection $P^1(\R)\times (-1,1)^{n-1}\to P^1(\R)$.  
 
Now we use the latter lemma, with $\lambda= (-2\tilde \lambda_p(\sb_m)-2 \tilde \lambda_p (\sc^m))_p$. This provides a family of maps $h_p^\lambda: P^1(\R) \to P^1(\R)$  fixing $0$ and having the jet defined by $\lambda$. 
Let $(H_p^\lambda)_p$ be a $C^r$-family of maps of $M$  which 
leaves invariant the trivial  fibration $P^1(\R)\times (-1,1)^{n-1}\to (-1,1)^{n-1}$, which is $C^r$-close to the identity, coincides with the identity nearby any fiber with  $0$-letter   in $\sB^\lambda_r$ and equal to $h^\lambda_p$ nearby any other fibers.  Then $S_p$  is still a $2m$ periodic point for $H_p^\lambda\circ \breve f_p$. Also  half of  points in  the orbit of $S_p$ have their derivatives along the fibration   multiplied  
by  $\exp(2\lambda_p)$ while the others are unchanged.    Thus $S_p$   is  semi-parabolic for $H_p^\lambda\circ \breve f_p$. 
Note that after this perturbation, the fibration  might have changed. However, the $2m$-periodic point $S_p$  shadows the normally hyperbolic,  $2m$-periodic    fiber of 
$\ss^m$ and so must lie in it. 
\end{proof}

%Rajouter deux parametres tq $\tilde g^{\sc^\infty}_{p,w}(x)=v\cdot \rho(\|x-\tilde x_r(p)-u\|)\cdot(x-\tilde x_r(p)-u)+ \tilde g^{\sc^\infty}_{p,w}(x-u)+u $.
%
%\[\tilde g^{\sc^\infty}_{p,w}(x)=\exp(v\cdot \rho(\|x-\tilde x_r(p)-u\|))\cdot( \tilde g^{\sc^\infty}_{p,w}(x-u)-\tilde x_r(p))+\tilde x_r(p)+u .\]
%
%\[\tilde g^{\sc^\infty}_{p,w}(x)=\exp\left(v_p\star \rho\left(\frac{p}{\|x-\tilde x_r(p)-u\|}\right)\right)\cdot( \tilde g^{\sc^\infty}_{p,w}(x-u)-\tilde x_r(p))+\tilde x_r(p)+u .\]
%si tout est ok avec le point repul: il est eqale à 0. 
%
%\[\tilde g^{\sc^\infty}_{p,w}(x)=\exp\left(v_p\star \rho\left(\frac{p}{\|x\|}\right)\right)\cdot  \tilde g^{\sc^\infty}_{p,w}(x)  .\]

\section{Perturbation of families of parabolic circle maps to constant rotations}\label{ProofthmS4}
This section is devoted to the proof of Theorem \ref{thmS4}.
Let $k\in \N$  and let $ B\subset \R^k$ be an open subset. Let   $(g_p)_{p\in B}$ be  a $C^\infty$-family of circle diffeomorphisms so that for every $p\in B$ the map $g_p$ is parabolic.
Given any $B'\Subset B$, we want to find a $C^\infty$-perturbation $(\tilde g_p)_{p\in B}$ of $(g_p)_{p\in B}$ such that $\tilde g_p$ displays a Diophantine rotation number which does not depend on $p\in B'$.

We shall work with the coordinates given by a one-point compactification $P^1(\R)$ of $\R$. Hence given $a<b\in \R$, we will denote by $(a,b)$ the segment of $\R\subset P^1(\R)$ and by  $(b,\infty,a)$ the arc of $P^1(\R)$ containing $\infty$ and with endpoints $\{a,b\}$. \medskip

\noindent {\bf Sketch of proof}: 
The proof of the theorem is done by the following steps:
\begin{enumerate}
\item First we show that in a  $C^\infty$-family of coordinates, we can assume  that for every $p\in B$:
\begin{enumerate}
\item The map $g_p$ fixes $0$, satisfies  $D^2g_p(0)= 2$    and sends $\{1/2,1,-1\}$ to $\{1,-1,-1/2\}$.
\item Up to   a $C^\infty$ perturbation of $(g_p)_p$,   the restriction of $g_p$ to $[-1,1/2]$ coincides with the time one flow of a smooth vector field $X_p$ on $[-1,1]$ depending $C^\infty$ on $p$.  
%\item The map $(x,a)\mapsto g_p(x)$ is of class $C^\infty$ .
\end{enumerate}  

\item  %Hypothesis $(b)$ will enable us to show that $g_p|(-1,1)$ is the time one of a flow given by a vector field $X_p|P^1(\R)\setminus \{-1\}$ depending smoothly on the parameter $p$.  
We look at a family of perturbations $(g_{p\; \eta})_p$ for $\eta$-small defined by:
\begin{enumerate}[$\bullet$]
\item  $g_{p\; \eta}$ is equal to $g_p$ on $(\frac12,\infty,-1)$,
\item  $g_{p\; \eta}$ coincides on $[-1, \frac12]$ with the time one map of the flow of the vector field $X_p+\rho\cdot \eta^2$  for a fixed $\rho\in C^\infty(\R,[0,1])$  supported by $(-1/2,1/2)$ and such that $\rho(0)=1$. 
\end{enumerate}
Then, we show that in some coordinates (given by the 
two canonical extensions of $X_p$ over $(1,\infty,-1)$) the first return map $G_{p\;\eta}$ of $g_{p\; \eta}$ in $(1,\infty,-1)$ is of the form $G_{p\;0}+ \omega^+_p(\eta)$, with $(p,\eta)\mapsto  \omega_p^+(\eta)$ of class $C^\infty$. 
 
\item  We show the existence of a $C^\infty$-family $(\Omega_p)_p$ of functions $\Omega_p\in C^\infty([0,\infty),\R)$ such that $\Omega_p(0)= 1/\pi$ and $\omega_p^+(\eta)= \Omega_p(\eta)/\eta \mod 1$ for every $p$; using an 
integral  formula on $\omega^+_p(\eta)$. 

\item We use the following KAM's Theorem of Herman-Yoccoz:

 \begin{theorem}[Herman-Yoccoz]\label{KAM}
 For every $\beta\in \R$ Diophantine, let $V_\beta$ be the set of circle diffeomorphisms in $\Diff^\infty(\R/\Z)$ whose rotation number is $\beta$. 
 Then $V_\beta$ is a smooth submanifold of codim 1.  Moreover, for every $f\in V_\beta$, the family $(f+b)_{b\in \R}$ is transverse to $V_\beta$ at $b=0$. 
\end{theorem} 
\begin{proof} By Yoccoz' Theorem \ref{AHthm}, we can assume that $f$ is the rotation of angle $\beta$. Then the Theorem is stated in  Remark 3.1.3 of \cite{Bo84}.\end{proof}
Then given a Diophantine number $\beta$, we can define implicitly an arbitrarily $C^\infty$-small function $p\mapsto \eta(p)$ (for the compact-open topology) so that the return map $G_{p\; \eta(p)}$ displays the rotation number $\beta$ for every $p\in B'$. This implies that the rotation number of $g_{p\; \eta(p)}$ is of the form $1/(N+\beta)$, and so is Diophantine as well.
In other words, with $g'_p:= g_{p\; \eta(p)}$, the  family $(g'_p)_p$  satisfies the sought properties.  
\end{enumerate}

\noindent{\bf Step 1: Setting}.
a) Let $x_p\in P^1(\R)$ be the fixed point of $g_p$ for $p\in B$. As the map $g_p$ is parabolic,
 $\partial_xg_p(x_p)=1$ and 
 $\partial_x^2g_p(x_p)\not= 0$. Hence, $x_p$ can be defined (locally) as the zero of:
\[(x,p)\mapsto   \partial_x g_p(x)-1\; .\]
Hence by the implicit function theorem, the map $p\mapsto x_p$ is of class $C^\infty$. Thus by a smooth coordinate change,  we can assume that $x_p=0$ for every $p\in B$. Then  we can conjugate the dynamics by the Moebius function $x\mapsto \frac{2\cdot x}{D^2g_p( x_p)}$ for every $p\in B$, so that   for every $p\in B$ it holds:
\[ g_p(0)=0\quad ,\quad Dg_p(0)=1\quad ,\quad  D^2g_p(0)=2\; .\]
Now we conjugate $g_p$ by a smooth family of diffeomorphisms, equal to the identity on a neighborhood of $0$ and sending $\frac12$, $g_p(\frac12)$, $g_p^2(\frac12)$ and $g_p^3(\frac12)$ to respectively $\frac12,$ $1,$ $-1 $ and $-\frac12$. 
 \begin{figure}[h]
    \centering
        \includegraphics[width=9cm]{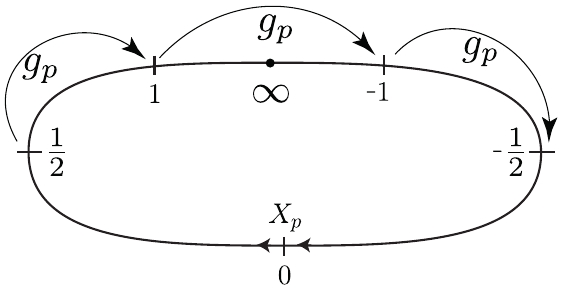}
\caption{}
      \label{notaparabolicrenor0}
\end{figure}

b) Let $d\ge 2$ and $\frak G_d= \{\sum_{j=2}^r \frac{g_j}{j!} x^j\colon(g_j)_{j=2}^d\in \R^{d-1}\} $. Given $\frak X\in \frak G_d$, let $\phi_{\frak X}$ be  the time one flow  of the vector field $x+\frak X(x)$. Let $J^d_0\phi_{\frak X}$  be the $C^d$-jet of $\phi_{\frak X}$:
\[J^d_0\phi_{\frak X} = \sum_{j=0}^d \frac{\partial^j \phi_{\frak X}}{j!} x^j= x+ \sum_{i=2}^d \frac{\partial^j \phi_{\frak X}}{j!} x^j\; .\]
In other words, $J^d_0 \phi_{\frak X}$ belongs to the space $G_d$ of $C^d$-jets of parabolic maps at $0$:
\[ G_d=\{x+\sum_{j=2}^d \frac{p_j}{j!} x^j+o(x^d)\colon(p_j)_{j=2}^d\in \R^{d-1}\}  \; .\]
\begin{proposition}
The map $\frak X\in \frak G_d\mapsto \phi_{\frak X}\in C^\infty(\R,\R)$ is smooth. 
Moreover, the following map is a diffeomorphism onto its image:
$$\Psi\colon \frak X\in \frak G_d\mapsto J^d_0\phi_{\frak X}\in G_d\; .$$
\end{proposition}
\begin{proof}
The first statement of this proposition is a simple consequence of the Cauchy-Lipschitz Theorem.  The second part of the proposition involves the Lie group theory. Indeed, the $C^d$-jet space $G_d$ endowed with the composition rules is a Lie group. Moreover it satisfies:
\begin{fact}
The group $G_d$ is connected, simply connected and nilpotent.
\end{fact}
\begin{proof}
The group $G_d$ is homeomorphic to $\R^{d-1}$, hence it is connected and simply connected.    
Let  $G^{(s)}_d := \{\phi\in G_d:\; \phi(x)= x+O(x^{s+2})\}$ for $s\ge 0$.  A computation gives $[G_d,G^{(s)}_d]\subset G^{(s+1)}_d$. As  $G_d^{(d-1)}$ is trivial, $G_d$  is nilpotent with rank $\le d-1$.  
\end{proof}
We notice that $\frak G_d$ is the Lie algebra of the group $G_d$. Moreover, the jet of  $\phi_{\frak X}$ is the image by the exponential map $\exp$ of $\frak X\in \frak G_d$.  Indeed, if $\phi_{\frak X}^t= \exp(t \cdot \frak X)$ and so $\phi_{\frak X}=\phi_{\frak X}^1$, it holds: 
$$\phi_{\frak X}^{t+\delta t}=\phi_{\frak X}^{\delta t}\circ \phi_{\frak X}^t=(id+\delta t \cdot \frak X)\circ \phi_{\frak X}^t+o(\delta  t)$$
Finally we infer that for a simply connected and nilpotent  Lie group, the exponential map  is an analytic diffeomorphism from the Lie algebra onto the group \cite[pp. 13, Thm 1.2.1]{CG90}.%, the proposition is proved.
\end{proof}

\begin{corollary} There exists a $C^\infty$-map $p\mapsto X_p\in \frak G_d$ so that $J^d_0 \phi_{X_p}= J^d_0 g_p$ for every $p$.\end{corollary}
%Let  $g^t_p= \exp(t\cdot X_p)$. Note that the restriction of  $g_p$ to a connected neighborhood $N$ of $0$ is equal to the one of $g_p^1$.  As the differential of the exponential map at $0$ is the identity, $g^{t+dt}_p=g^{dt}_p\circ g^{t+dt}_p=(id+dt \cdot X_p)\circ g^{t+dt}_p +o(dt)$. Consequently:
%\[\partial_t g^t_p|N=  X_p\circ g^t_p|N\; .\]
%In other words,  $g_p|N$ is the time one map given by the flow $X_p$. 

%
%We observe that when $\epsilon>0$ is small and $d$ is large, the following map is $C^\infty$-close to   $g_p|(-\epsilon,\epsilon)$:
%\[J^d_0 g_p:= x\mapsto x+\sum_{j=2}^d \frac{D^j f_p(0)}{j!} x^j \; .\]
%Moreover, the families $( g_p|(-\epsilon,\epsilon))_{p\in B}$ and $(J^d_0  g_p|(-\epsilon,\epsilon))_{p\in B}$ are $C^\infty$ close.

Let us fix a bump function $\rho$ equal to $1$ on $[-\frac14,\frac14]$ and with support in $(-\frac 1 2,\frac 1 2)$, and for $\epsilon>0$ small,  put:
\[\tilde g_p\colon x\mapsto \rho(x/\epsilon) \cdot \phi_{X_p}(x)+ (1-\rho(x/\epsilon) )\cdot g_p(x)\; .\] 
\begin{lemma}
For $d$ large and then $\epsilon>0$  small, 
the family $(\tilde g_p)_{p\in B}$ is $C^\infty$-close to $(g_p)_{p\in B}$. 
\end{lemma}
\begin{proof}
We have $g_p(x)-\tilde g_p(x)= \rho( x/\epsilon)(g_p(x)- \phi_{X_p}(x))$.  For every $j\le d$, we notice that $(p,x)\in B\times [-\epsilon,\epsilon] \mapsto g_p(x)- \phi_{X_p}(x)$ has its $j^{th}$-derivative which is small w.r.t. $\epsilon^{d-j}$. 
 On the other hand, the $j^{th}$-derivative of $\rho(x/\epsilon)$ is dominated by $\epsilon^{-j}$. Hence by Leibnitz formula, the $C^{d}$-norm of $(p,x)\in B\times (-\epsilon,\epsilon)\mapsto g_p(x)-\tilde g_p(x)$ is small when $\epsilon$ is small. 
\end{proof}
Observe that  $(\tilde g_p)_{p\in B}$  satisfies Condition $(a)$, and is equal to the time one map of the flow defined by $X_p$ at the neighborhood of $0$. Using the formula $D\tilde g_p\circ X_p\circ \tilde g_p^{-1}=X_p$, we pull back  $X_p$ to extend it to $[-1,0]$ and we push forward to extend it on $[0,1]$. This 
defines a smooth family of vector field on $[-1,1]$ so that  Condition $(b)$ holds true.   
Hence we can suppose that after perturbation $(g_p)_{p\in B}$ satisfies Conditions $(a)$, $(b)$ and furthermore:
 \begin{claim}\label{ClaimX2}
 It holds $X_p(x)= x^2+O(x^3)$ for every $p$.
 \end{claim}
 \begin{proof} When  $x\to 0$,   $J^d_0 g^2_p(x)-J^d_0 g_p(x)$ is equivalent to both $\frac{D^2  g_p(0)}2 x^2=x^2$ and $X_p(x)$.\end{proof}
\noindent{\bf Step 2: Definition of  $(g_{p\; \eta})_{p,\eta}$ and uniform bound on its first return map.}
Recall that $\rho$ is  a bump function equal to $1$ on  $(-\frac1 4,\frac 1 4)$ and with support in $(-\frac12,\frac12)$. For $\eta>0$ small, let:
\[X_{p\; \eta}\colon x\in [-1,1]\mapsto  X_p(x)+\eta^2\cdot \rho(x)\; .\]
We recall that for $x$ in $[-1,\frac12]$ the time one map of the flow of $X_p$ is well defined and equal to $g_p$. Also $g_p$ sends  $-1$  to $-\frac12$ and $\frac12$ to $1$ and the images by the flow of $X_p$ of these points during  times in $[0,1]$  are respectively $[-1,-\frac12]$ and $[\frac12,1]$. Thus for $\eta$-small, the time one map $\phi^{1}_{p\; \eta}(x)$ of the flow is well defined on $[-1, \frac12]$ and coincides with $g_p$ nearby $\{-1, \frac 12\}$. This implies:
% is well defined for $x$ in $[-1, \frac12]$. 
%
%As $g_p$ sends  $-1$  to $-\frac12$ and $\frac12$ to $1$, and the images by the flow of $X_p$ of these points during  times in $[0,1]$  are respectively $[-1,-\frac12]$ and $[\frac12,1]$. Also nearby $-1$    and $\frac12$, the time one map of the vector fields $X_p$ and $X_{p,\eta}$ coincide. This implies:
%
\begin{claim} The following family $(g_{p\; \eta})_{p,\eta}$ of $C^\infty$-dynamics $g_{p\; \eta}$ is well defined and smooth:
\[g_{p\; \eta}\colon x\mapsto \left\{\begin{array}{lc}\phi^{1}_{p\; \eta}(x)& \text{if } x\in [-1,\frac12]\; ,\\
 g_p(x)& \text{if } x\in (\frac12,\infty, -1)\; .\end{array}\right.\]
\end{claim}

For $\eta>0$, let us study the first return map $T_{p\; \eta}$ of $g_{p\; \eta}$ into $[1,\infty,-1]=[1,\infty,g_{p\; \eta}(1)]
$. 
To this aim,   we shall work with two different possible extensions of  $X_p|[-1,1]$ on $(1,\infty,-1)$.  Let:
\[\left\{\begin{array}{c}
X_p^+ := x\in [1,\infty,-1) \mapsto Dg_p \circ X_p\circ g_p^{-1}(x)\; ,
\\
X_p^- := x\in (1,\infty,-1] \mapsto  Dg^{-1}_p \circ X_p\circ g_p(x)\; .\end{array}\right.\]

In general $X_p^+$ and $X_p^-$ are different.  As $X_p|[-1,1]$ is equal to $X_{p\; \eta}| [-1,-\frac12]\cup [ \frac12,1]$, it holds : 
\begin{fact}
The vector field  $X_p^+$  extends smoothly $X_{p\; \eta}|(-1,1)$ to one denoted  by $X^+_{p\; \eta}$ on 
 $P^1(\R)\setminus \{-1\}$.
 Also,  $X_p^-$ extends smoothly $X_{p\; \eta}$ to one denoted by $X^-_{p\; \eta}$ on $P^1(\R)\setminus \{+1\}$.
\end{fact}
We now study the first return map $T_{p\, \eta}$ in $(1,\infty,-1]$ induced by $g_{p\; \eta}$. The idea is to glue the endpoints $1$ and $-1=g_{p\, \eta} (1)$ of this interval by the dynamics, so that the quotient $(1,\infty,-1]/\sim$ is a circle, and the action of $T_{p\, \eta}$ on it  enjoys nice bound bounds when $\eta\to 0$.  

Let $N=N(\eta)$ be the first return time  of $-1$ into $(1,\infty,-1]$, let  $\omega_p(\eta)=g_{p\; \eta}^N(-1)\in (1,\infty,-1]$. As $g_{p\, \eta}$ is orientation preserving, the point at the left of $-1$ (see \cref{notaparabolicrenor}) are sends by $g_p^N$ at the left of $ 
\omega_p(\eta)$ until they exit of $[-1, \infty , 1]$. Let 
$\alpha_p(\eta)\in [1,\infty,-1]$ be the point sent by $g_{p\; \eta}^N$  to $1$. Hence:
\[T_{p\, \eta}(-1)=\omega_p(\eta)\qand  T_{p\, \eta}(\alpha_p(\eta))\sim 1\]
\begin{figure}[h]
    \centering
        \includegraphics[width=9cm]{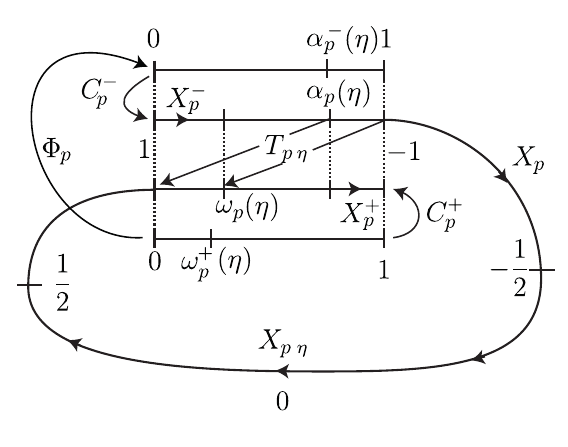}
    \caption{Notations for the parabolic renormalization.}
      \label{notaparabolicrenor}
\end{figure}
%
%preimage $ \alpha_p(\eta):= (g_{p\; \eta}^N)^{-1}(1)\in[1,\infty,-1]$.
%For $\eta>0$, there is a first return time $N>0$ of $-1$ into $[1,\infty,-1]$, let  $\omega_p(\eta)=g_{p\; \eta}^N(-1)\in [1,\infty,-1]$.
% Let $\alpha_p(\eta)\in [1,\infty,-1]$ be the preimage $\color{red}  \alpha_p(\eta):= (g_{p\; \eta}^N)^{-1}(1)\in[1,\infty,-1]$.   Note that 
% $T_{p\; \eta} (-1)= \omega_p(\eta)$ and  
% $T_{p\; \eta} (\alpha_p (\eta))=1$ as depicted in \cref{notaparabolicrenor}. 
To bound  $T_{p\, \eta}$, we consider the diffeomorphisms $C_p^+$ and $C_p^-$ from $[0,1]$ onto  $ [1,\infty,-1]$ 
which at $t\in [0,1]$ associate the image of  $1$  by the time $t$ of the  flow of respectively\footnote{In this definition, we considered the continuous extensions of $X_p^+|[1,\infty,-1)$ and $X_p^-|(1,\infty,-1]$ to $[1,\infty,-1]$.} $X_p^+$ and  $X_p^-$. %These are diffeomoprhisms form $[0,1]$ onto $ [1,\infty,-1]$:
\[C_p^+: [0,1] \to [1,\infty,-1]\qand C_p^-: [0,1] \to [1,\infty,-1]\]
%Note that both maps $C_{p}^\pm$ send $[0,1]$ onto $[1,\infty,-1]$.
Remark also that both $C_p^\pm$  do not depend on $\eta$ since $X_p^\pm$ does not depend of $\eta$.

Let $T^-_{p\; \eta} :=(C_p^-)^{-1} \circ T_{p\; \eta} \circ  C_{p}^{-}$ be the first return map $T_{p\; \eta}$ seen in the coordinates $C_p^-$. 
Let $\alpha^\pm_p(\eta)$ and $\omega^\pm_p(\eta)$ be the preimages of $\alpha_p(\eta)$ and $\omega_p(\eta)$ by $C_p^\pm$.  We observe that (see \cref{notaparabolicrenor}):
\[T^-_{p\; \eta} (1)= \omega_p^-(\eta)\quad \text{and}
\quad T^-_{p\; \eta} (\alpha_p^-(\eta))=0\; .\]

%For $\eta>0$, there is a first return time $N>0$ of $-1$ into $[1,\infty,-1]$, let  $\omega_p(\eta)=g_{p\; \eta}^N(-1)\in [1,\infty,-1]$.
% Let $\alpha_p(\eta)\in [1,\infty,-1]$ be the preimage $\color{red}  \alpha_p(\eta):= (g_{p\; \eta}^N)^{-1}(1)\in[1,\infty,-1]$.   Note that 
% $T_{p\; \eta} (-1)= \omega_p(\eta)$ and  
% $T_{p\; \eta} (\alpha_p (\eta))=1$ as depicted in \cref{notaparabolicrenor}. 

 We define the following coordinates change:
\[\Phi_p\colon [0,1]\mapsto (C_{p}^-)^{-1}\circ C_{p}^+(t)\in [0,1]\; .\]
%
%
%
%
% 
%We depict in  \cref{notaparabolicrenor} the notations involved in this renormalization. 

\begin{claim}\label{claim 5.9}
The first return map $T^-_{p\; \eta}$ satisfies:
\[\left\{\begin{array}{cc}
T^-_{p\; \eta}(s)= \Phi_p(s+1-\alpha^-_p(\eta)) = \Phi_p(s+\omega^+_p(\eta) ) &\text{if } s<\alpha_p^-(\eta)\; ,\\
T^-_{p\; \eta}(s)= \Phi_p(s-\alpha_p^-(\eta))= \Phi_p(s+\omega^+_p(\eta) -1)&\text{if } s>\alpha_p^-(\eta)\; .\\
\end{array}\right.\]
\end{claim}
\begin{proof}
As depicted in \cref{notaparabolicrenor},  there is a segment which projects  canonically twice on $[-1,\infty, 1]$ and once on $[-1,1]$, so that $X_p^+$, $X_p^{-1}$ and $X_{p\, \eta}$ define a smooth vector field on it. The image of $\alpha_p(\eta)$ and $-1$ by the time $N$ of its flow  are respectively $1$ and $\omega_p(\eta)$. Hence the time to needed go from $\alpha_p(\eta)$ to $-1$ equal to the the times needed to go from $1$ to $\omega_p(\eta)$. This means:
%
% On this segment, we can define define a vectyor field equal to   
%
%by commutativity of the flow $X_{p\; \eta}$ with $g_{p\; \eta}^N$, it holds: 
\[ 1-\alpha_p^-(\eta)= \omega_p^+(\eta)\; .\] 
If $s>\alpha_p^-(\eta)$, then the first return of $x= C_p^-(s)$ in $[1,\infty,-1]$ is the image by the flow  $X_p^+$ of $1$ after a time $s-\alpha_p^-(\eta)$. In the coordinate $C_p^+$, it is  $s-\alpha_p^-(\eta)$. In the coordinate $C_p^-$, it is $\Phi_p(s-\alpha_p^-(\eta))$. 

If $s<\alpha_p^-(\eta)$, then the first return of $x= C_p^-(s)$ in $[1,\infty,-1]$ is the image by the flow  $X_p^+$ of $\omega_p(\eta)$ after a time $s$. In the coordinate $C_p^+$, it is  $\omega^+_p(\eta) +s$. In the coordinate $C_p^-$, it is $\Phi_p(\omega^+_p(\eta) +s)$.\end{proof}

\begin{claim} After gluing the endpoints of $[0,1]$ by the translation by $1$, the map $T^-_{p\; \eta}$ is a smooth map of the circle $\R/\Z$, and the family $(T^-_{p\; \eta})_{p, \eta}$ is of class $C^\infty$. 
\end{claim}
\begin{proof}
It is classical \cite{Yo95} that the first return map $g_{p\; \eta}$ into $[1,\infty, g_{p\, \eta}(1)]=[1,\infty, -1]$ is projected to a smooth map  of the circle obtained by gluing the endpoints of $[1,\infty, g_{p\, \eta}(1)]$  using $g_{\eta\; a}$. Seen in the chart $C_p^-$, this corresponds to glue the endpoints of $[0,1]$ using the translation by $+1$. 
In this specific parabolic context, this map was called "the essential map"  in \cite{ST00}.
\end{proof}

With the conjugacy $s\mapsto s+\omega^+_p(\eta)$ and by Claim \ref{claim 5.9}, we obtain:\begin{corollary}
The first return map $T_{p\; \eta}$ of $g_{p\, \eta}$  is smoothly conjugated to:
\[R_{p\; \eta}\colon s\in \R/\Z\mapsto \Phi_p(s)+\omega^+_p(\eta)\in \R/\Z\]
\end{corollary}

\noindent{\bf Step 3.}
We shall study the derivative of $\omega_p^+(\eta) \mod 1$ with respect to $\eta$   when $\eta\to 0$.  We recall that the time needed for the flow $X_{p\; \eta}$ to go from $-1$ to $1$, is:
\[\tau_p(\eta):=\int_{[-1,1]} \frac{1}{X_{p\; \eta}} d\leb\; .\]
We notice that $\omega_p^+(\eta)+\tau_p(\eta)= N\in \Z$ and so:
\[\omega_p^+(\eta)= -\tau_p(\eta)\mod\, 1 \; .\]

By Claim \ref{ClaimX2}, there exists $X^1_{p,\eta}\in C^\infty ([-1,1],\R)$ such that:
\[X_{p,\eta} (x)= x^2  X^1_{p,\eta}(x)+\eta^2\; .\]
As $g_p$ has a unique fixed point at $0$, the field $X^1_{p\; 0}$ is positive on $[-1,1]$, with value $1$ at $0$ by Claim \ref{ClaimX2}. Thus for $\eta$ small, there exists  $C>0$ such that:
\begin{equation}\label{minoration}
X^1_{p,\eta}(0)=1\quad \text{and}\quad  X^1_{p,\eta}\ge C>0\; .\end{equation}
In these notations, it holds:
\[\eta\cdot\tau_p(\eta)= \int_{[-1,1]} \frac\eta{s^2\cdot X^1_{p\; \eta}(s) +\eta^2} ds\; .\]

We put $s= \eta \cdot t$ and we have:
\[ \eta\cdot \tau_p(\eta) =\int_{-1/\eta}^{1/\eta} \frac1{1+t^2 X^1_{p\, \eta}(\eta t) }dt,\quad \forall \eta\not=0\; .\]  

Let $\Psi(t,p,\eta):= 1+t^2 X^1_{p\, \eta}(\eta t) $. 
\begin{lemma}\label{Cn}
For every $n\ge 0$, there exists $C_n>0$ so that for every $p\in B'$ and $\eta$ small:
\[\left| \partial^n_{p\,\eta} \frac1{\Psi(t,p,\eta)}\right| \le 
\frac{C_n}{1+t^2}\]
\end{lemma}
\begin{proof}
The case $n=0$ is an immediate consequence of Inequality (\ref{minoration}).  Let $n\ge 0$ and assume by induction that Lemma \ref{Cn} holds for every $k\le n$. By Leibniz formula  applied to $\Psi/\Psi$ , it holds:
\[\sum_{k=0}^{n+1} C^k_{n+1} 
 \partial^{k}_{p\, \eta} \frac1{\Psi(t,p,\eta)}\cdot 
\partial^{n+1-k}_{p\, \eta} {\Psi(t,p,\eta)}=0\]
\[\Rightarrow
\partial^{n+1}_{p\, \eta} \frac1{\Psi(t,p,\eta)}=
-\frac1{\Psi(t,p,\eta)}  \sum_{k=0}^n C^k_{n+1} 
 \partial^{k}_{p\, \eta} \frac1{\Psi(t,p,\eta)}\cdot 
\partial^{n+1-k}_{p\, \eta} {\Psi(t,p,\eta)}\; .
\]
It is easy to see that for $0\le k\le n$, the derivative $\partial^{n+1-k}_{p\, \eta} {\Psi(t,p,\eta)}$ is bounded by a certain $C'\cdot t^2$. Hence the induction hypothesis gives:
 \[\left| \partial^{n+1}_{p\, \eta} \frac1{\Psi(t,p,\eta)}\right|\le 
  \frac{C_0}{1+t^2}   \sum_{k=0}^n C'\cdot  t^2\cdot  \frac{ C_k}{1+t^2}\]
 Hence the above sum is bounded from above by 
$C_{n+1} \frac1{1+t^2} $ for $C_{n+1}\in \R$ independent of $t$, $p\in B$ and $\eta$ small. 
\end{proof}

We notice that by the dominated function theorem, the following function is of class $C^\infty$:
\[(p,\eta)\mapsto \eta\cdot \tau_p(\eta) \; .\]
 Moreover it holds:
\[\lim_{\eta\to 0}  \eta \cdot \tau_p(\eta)  = \int_\R \frac1{\Psi(t,p,0)}dt= 
\int_\R \frac1{1+t^2}dt= \pi\; .\]
Consequently it holds:
\begin{fact}
There exists a $C^\infty$-family $(\Omega_p)_{p}$ of $C^\infty$-functions  $\Omega_p\in C^\infty([0,\infty),\R)$ so that 
$$\omega_p^+(\eta)=   \Omega_p(\eta)/\eta  \mod 1\; .$$
\end{fact}
\begin{proof}
Indeed from the above discussion, the number  $\Omega_p(\eta):= -\eta \cdot \tau_p(\eta) $ depends smoothly on $p\in \R$ and $\eta\ge 0$.
\end{proof}
\noindent{\bf Step 4.}
We recall that the coordinates change $\Phi_p$ projects to a $C^\infty$-diffeomorphism of the torus $\R/\Z$, and $(\Phi_p)_p$ is of class $C^\infty$.  
 Here is an immediate consequence of Herman-Yoccoz'  \cref{KAM}:
\begin{corollary}
For every   Diophantine number $\beta$, there exists a $C^\infty$-function $p\in B \mapsto \theta(p)$  such  that the map $s\mapsto \Phi_p(s)+\theta(p)\in \R/\Z$ has rotation number equals to $\beta$.
\end{corollary}  
 Hence it remains to solve implicitly:
\[\omega_p^+(\eta(p))=\theta(p)\mod 1\Longleftrightarrow \frac{\Omega_p(\eta(p))}{\eta(p)}= \theta(p) \mod1\; .\]
As $\Omega_p(0)=\pi$, for every $p_0\in B$, there exists $\eta(p_0)$ arbitrarily small so that 
\[\frac{\Omega_p(\eta(p_0))}{\eta(p_0)}= \theta(p_0)\mod 1\; .\]
Note that the following derivative is large for $\eta$ small (since   $\Omega_p(0)=\pi$ and $\partial_\eta\Omega_p(0)$ is bounded). 
\[\partial_\eta \left( \frac{\Omega(\eta)}{\eta}\right)= \frac{\eta \cdot \partial_\eta\Omega_p(\eta)-\Omega_p(\eta)}{\eta^2}\sim - \frac\pi{\eta^2}\; .\]

Hence, the implicit function theorem enables us to conclude that for every $B'\Subset B$ there exists a small smooth function  $p\in B'\mapsto \eta(p)$ so that $R_{p\; \eta(p)}$ has rotation number equal to $\beta$. 
Then, $g_{p\, \eta(p)}$ has a rotation number of the form $  1/(N+\beta)$ which is Diophantine as well.  Note that $(g_{p\, \eta(p)})_{p\in B'}$ is $C^\infty$-close to $(g_p)_{p\in B'}$. Thus Theorem \ref{thmS4} is proved.
$\qed$

\begin{appendix}
 \section{Extrinsic definition of $(\lambda)$-blender and $C^r$-$(\lambda)$-parablender}\label{appendix}

Let us give for the first time the extrinsic definition of the $\lambda$-blender, $C^r$-parablender and $\lambda$-$C^r$-parablender in the diffeomorphism case. The endomorphisms cases of these objects (but not their $\lambda$-version) were extrinsically defined in \cite{BE15,BCP16}.

Let $f$ be a diffeomorphism of a manifold $M$ and let $K$ be a hyperbolic basic set. We assume that  $K$ is partially hyperbolic with a contracting central direction: there exists a $Df$-invariant  splitting $TM|K= E^{ss} \oplus E^{c} \oplus E^u$ such that for every $x\in K $:
\[ \|D_xf|E^{ss} \|< \|(D_xf|E^{c})^{-1} \|^{-1}\le  \|(D_xf|E^{c}) \|<1 < \|(D_xf|E^{u})^{-1} \|^{-1}\]
We fix a continuous family of strong stable and unstable manifolds $(W^{ss}_{loc} (x; f))_{x\in K}$ and  $(W^{u}_{loc} (x; f))_{x\in K}$. 
The lamination $W^{u}_{loc} (K; f):= \bigcup_{x\in K} W^{u}_{loc} (x; f)$ is invariant by $f^{-1}$, and  so is its tangent bundle. 
So for every $y\in W^{u}_{loc} (K; f)$, with $y_{-n}:= f^{-n}(y)$ we can consider the action $[D_{y_{-n}} f^n]$ of $D_{y_{-n}} f^n$ on the quotient vector spaces 
$T_{y_{-n}} M/ T_{y_{-n}}W^{u}_{loc} (K; f)$ onto $T_{y} M/ T_{y}W^{u}_{loc} (K; f)$. 

\begin{definition}[$(\lambda)$-Blender]
The hyperbolic compact set $K$ is  a \emph{blender} if there is  $C^1$-neighborhood $V^{ss}$ of a strong local stable manifold of $K$   such that for every $C^1$-perturbation $\tilde f$ of $f$ and every $W\in V^{ss}$ there exists $x\in K$ such that $W$ intersects $W^{u}_{loc}(x; \tilde f)$.

The blender  $K$  is a \emph{$\lambda$-blender} if moreover $\dim E^c=1$ and there exists a nonempty  open subset $U^\lambda\subset \R$ 
 such that  for every $C^1$-perturbation $\tilde f$ of $f$ and every $W\in V^{ss}$ and $\ell \in U^\lambda$, there exists $x\in K$  such that $W$ intersects $W_{loc}^{u}(x; \tilde f)$ at a point $y$ and it holds:
 \[\lim_{n\to \infty} \frac1n \log \|[D_{y^{-n}} \tilde f^{ n}]\|=\ell,\quad  \text{with  }y^{-n}:= \tilde f^{-n}(y)\; .\]
 \end{definition}
 
Let us now give the parametric version of these definitions. Assume that $(f_p)_{p\in B_k}$ is  a $C^r$-family such that $f_0= f$ and assume that the hyperbolic continuation $K_p$ of $K=K_0$ is well defined for every $p\in B_k$. Given $x_0\in K_0$ we denote $x_p\in K_p$ its continuation for $p\in B_k$.

\begin{definition}[$(\lambda)$-$C^r$-Parablender]
The continuation $(K_p)_{p\in B_k}$ is  a \emph{$C^r$-parablender} at $  {p_0}\in B_k $
 if the following condition is satisfied.
There is  a $C^r$-neighborhood $\hat V^{ss}$ of the continuation $(W^{ss}_{loc} (z_p; f_p))_p$ of a strong stable manifold   a point $z_0\in K_0$ such that for every $(\tilde f_p)_p$ which is $C^r$-close to $(f_p)_p$ and every $(W_p)_p\in \hat V^{ss}$, there exist a $C^r$-family of points $\hat P=(P_p)_p$ in $(W_p)_p$ and a $C^r$-family of points $ (Q_p)_p$ in $(W^u_{loc} (x_p; \tilde f_p))_{p}$ with  $x_p\in K_p$ such that:
\[J^r_{p_0} (Q_p)_p = J^r_{p_0} ( P_p)_p\; .\]
%\gamma\] 
%d(\gamma(a), Q_p)= o(\|a-a_0\|^r)\; .\]

This continuation  $(K_p)_{p\in B_k}$  is a \emph{$\lambda$-$C^r$-parablender}  at $p_0$ if moreover $\dim E^c=1$ and there exists a nonempty  open subset $J^r_{p_0}U^\lambda\subset J^{r-1}_{p_0} \R$   such that   for every $(\tilde f_p)_p$ which is $C^r$-close to $(f_p)_p$ and every $(W_p)_p\in \hat V^{ss}$   and $\ell \in J^r_{p_0} U^\lambda$,  there exist a $C^r$-family of points $\hat P=(P_p)_p$ in $(W_p)_p$ and a $C^r$-family of points $ (Q_p)_p$ in $(W^u_{loc} (x_p; \tilde f_p))_{p}$ with  $x_p\in K_p$ such that:
\[J^r_{p_0} (Q_p)_p = J^r_{p_0} ( P_p)_p\qand  \lim_{n\to \infty} J^{r-1} \frac1n \log \|[D_{y_p^{-n}} \tilde f_p^{n}]\|=\ell,\quad  \text{with  }y_p^{-n}:= \tilde f_p^{-n}(y_p)\; .\]
 \end{definition}
% We leave to the reader the exercise to show that  an intrinsic $(\lambda)$-$(C^r$-para)blender  (see \cref{def: blender intrinic,def: lambda blender intrinic,def parablender,def lambda parablender}) embedded in a normally hyperbolic fibration for a diffeomorphisms is a $(\lambda)$-$(C^r$-para)blender in the above senses. The converse is a priori not true. 
\end{appendix}
\bibliographystyle{alpha}
\bibliography{references}

   \end{otherlanguage}
\end{document}